\documentclass[10pt]{article}

\usepackage{amsmath,amssymb,url,xspace,amsthm}
\usepackage{mathrsfs,euscript}

\def\ES#1{\EuScript{#1}}

\def\wt#1{\widetilde{#1}}
\def\bs#1{\boldsymbol{#1}}

\usepackage[all]{xy}
\usepackage{enumitem}
\usepackage{authblk}

\makeatletter
\newcommand{\subjclass}[1]{%
  \let\@oldtitle\@title%
  \gdef\@title{\@oldtitle\footnotetext{\textbf{\emph{Mathematics subject classification (2010):}} #1.}}%
}
\newcommand{\keywords}[1]{%
  \let\@@oldtitle\@title%
  \gdef\@title{\@@oldtitle\footnotetext{\textbf{\emph{Key words and phrases:}} #1.}}%
}
\makeatother

\title{Matching of orbital integrals (transfer)\\ and Roche Hecke algebra isomorphisms}
\subjclass{22E50}
\keywords{Hecke algebra isomorphisms, geometric transfer, local data}

\author[1]{Bertrand Lemaire\thanks{The author acknowledge partial support by the Agence Nationale de la Recherche, 
projet ANR--13--BS01--00120--02 FERPLAY}}
\author[2]{Manish Mishra}

\affil[1]{\footnotesize Aix Marseille Universit\'e, CNRS, Centrale Marseille, I2M, UMR 7373, 163 Avenue de Luminy, 
Case 901, 13288 Marseille, France. \textit{Email:} Bertrand.Lemaire@univ-amu.fr.}

\affil[2]{\footnotesize Department of Mathematics, 
Indian Institute of Science Education and Research (IISER) Pune, India. \textit{Email:} manish@iiserpune.ac.in}

\date{\today}

\begin{document}

\maketitle

\begin{abstract}
Let $F$ be a non--Archimedan local field, $G$ a connected reductive group defined and split over $F$, and $T$ a maximal $F$--split torus in $G$. Let 
$\chi_0$ be a depth zero character of the maximal compact subgroup $\ES{T}$ of $T(F)$. It gives by inflation a character $\rho$ of an Iwahori 
subgroup $\ES{I}\subset \ES{T}$ of $G(F)$. From Roche \cite{R}, $\chi_0$ defines a reductive $F$--split group $\wt{G}'$ 
whose connected component 
$G'$ is an endoscopic group of $G$, and there is an isomorphism of ${\Bbb C}$--algebras 
$\ES{H}(G(F),\rho) \rightarrow \ES{H}(\wt{G}'(F),1_{\ES{I}'})$ where $\ES{H}(G(F),\rho)$ is the Hecke algebra of compactly supported $\rho^{-1}$--spherical functions on $G(F)$ 
and $\ES{I}'$ is an Iwahori subgroup of $G'(F)$. This isomorphism gives by restriction an injective morphism $\zeta: \ES{Z}(G(F),\rho)\rightarrow \ES{Z}(G'(F),1_{\ES{I}'})$ between the 
centers of the Hecke algebras. We prove here that a certain linear combination of morphisms analogous to $\zeta$ realizes the transfer (matching of strongly $G$--regular semisimple orbital integrals). 
If ${\rm char}(F)=p>0$, our result is unconditional only if $p$ is large enough.
\end{abstract}

\newtheorem*{rem}{Remark}
\newtheorem*{exam}{Example}
\newtheorem*{rem1}{Remark 1}
\newtheorem*{rem2}{Remark 2}
\newtheorem*{rem3}{Remark 3}
\newtheorem*{rem4}{Remark 4}
\newtheorem*{thm}{Theorem}
\newtheorem*{thm*}{Theorem*}
\newtheorem*{thm**}{Theorem**}
\newtheorem*{thm1}{Theorem 1}
\newtheorem*{thm2}{Theorem 2}
\newtheorem*{prop1}{Proposition 1}
\newtheorem*{prop2}{Proposition 2}
\newtheorem*{prop}{Proposition}
\newtheorem*{defi}{Definition}
\newtheorem*{lem}{Lemma}
\newtheorem*{lem1}{Lemma 1}
\newtheorem*{lem2}{Lemma 2}
\newtheorem*{lem3}{Lemma 3}
\newtheorem*{lem4}{Lemma 4}
\newtheorem*{nota}{Notation}
\newtheorem*{cor}{Corollary}
\newtheorem*{conj}{Conjecture}

\tableofcontents

\addcontentsline{toc}{section}{Introduction}
\section*{Introduction}

\noindent{\bf 0.1.} --- Let $F$ be a locally compact non--Archimedean field, $G$ a connected reductive $F$--group, and $\underline{G}'$ an endoscopic datum of $G$ with underlying group $G'$. 
If ${\rm char}(F)=0$, we know \cite{N,W2,W1} that for any function $f\in C^\infty_{\rm c}(G(F))$, there exists a function $f'\in C^\infty_{\rm c}(G'(F))$ such that $f$ and $f'$ have matching 
(semisimple) strongly $G$--regular orbital integrals. Such a function $f'$ is called an {\it endoscopic transfer}, or simply a {\it transfer}, of $f$. If ${\rm char}(F)>0$, the existence of transfer is not known in general. 

Even if the transfer is not unique, it is expected that for each $f$, there is a transfer $f'$ which is --- in some sense --- better than the others. For example if $G$ and $\underline{G}'$ are unramified, the {\it fundamental lemma} states that for any $K$--spherical function $f\in C^\infty_{\rm c}(G(F))$, there exists a unique $K'$--spherical function $f'\in C^\infty_{\rm c}(G'(F))$ which is a transfer of $f$; where $K$ and $K'$ are hyperspecial maximal compact sugbroups of $G(F)$ and $G'(F)$. Moreover in that case, the correspondence $f\mapsto f'$ is given by a morphism of algebras. If ${\rm char}(F)=0$, other examples of ``distinguished transfers'' are given in \cite{Hai1,Hai2} and \cite{KV}.

Our goal here is to produce ``distinguished transfers'' in the following case (studied by Roche \cite{R}): ${\rm char}(F)\geq 0$, $G$ is $F$--split, $G'$ is an ($F$--split) endoscopic group of $G$ associated 
to a depth zero character of the maximal compact subgroup of a maximal split torus in $G(F)$, and $f$ is in the center of the Hecke algebra associated to this character.

\vskip2mm
\noindent{\bf 0.2.} --- Let $p$ be the residual characteristic of $F$. From now on, we suppose $G$ is $F$--split. 
Let $T$ be a maximal $F$--split torus in $G$, and $\chi_0$ a depth zero character of the maximal compact subgroup $\ES{T}$ of $T(F)$. This character can be extended to a $W_{\chi_0}$--invariant character $\chi$ of $T(F)$, where $W_{\chi_0}$ is the subgroup of the Weyl group $W=N_G(T)/T$ formed of those elements which fix $\chi_0$. We fix an hyperspecial maximal compact subgroup $\ES{K}$ of $G(F)$ corresponding to a vertex in the apartment of the Bruhat--Tits bulding of $G(F)$ associated to $T$. Let $B$ be a Borel subgroup of $G$ containing $T$. 
The Borel pair $(B,T)$ of $G$ determines an Iwahori subgroup $\ES{I}\subset \ES{K}$ of $G(F)$. The character $\chi_0$ of $\ES{T}$ defines by inflation a character $\rho=\rho_{\chi_0}$ of $\ES{I}$. It defines also a connected reductive $F$--split group $G'$ with maximal $F$--split torus $T'=T$, which turns out to be an endoscopic group of $G$. We fix a hyperspecial maximal compact 
subgroup $\ES{K}'$ of $G'(F)$ as above. Then to $\ES{I}$ is associated an Iwahori subgroup $\ES{I}'\subset \ES{K}'$ of $G'(F)$. Let $\ES{Z}(G(F),\rho)$ be the center of the Hecke algebra $\ES{H}(G(F),\rho)$ of compactly supported $\rho^{-1}$--spherical functions on $G(F)$, and $\ES{Z}(G'(F),1_{\ES{I}'})$ the center of the Hecke algebra $\ES{H}(G'(F),1_{\ES{I}'})$ of compactly supported $\ES{I}'$--biinvariant functions on $G'(F)$. We have a natural injective morphism of ${\Bbb C}$--algebras (cf. \ref{Bernstein center})
$$
\zeta_\chi: \ES{Z}(G(F),\rho) \rightarrow \ES{Z}(G'(F),1_{\ES{I}'}).
$$
On the other hand, the image of the Langlands parameter $\varphi_\chi: W_F \rightarrow \check{T}$ of $\chi$ is contained in the center $Z(\check{G}')$ of the 
dual group (in the sense of Langlands) $\check{G}'$ of $G'$. Thus it defines a character $\omega'_\chi$ of $G'(F)$. The question is: does the map
$$
\ES{Z}(G(F),\rho)\rightarrow C^\infty_{\rm c}(G'(F)),\,f\mapsto (\omega'_\chi)^{-1}\zeta_\chi(f)
$$
realize the transfer?

\vskip2mm
\noindent{\bf 0.3.} --- 
In \cite{Hai2}, Haines answered a different but similar question, relative to unramified base change (in characteristic zero). 
In the case of unramified base change, the study is made simpler by the following property: for an unramified extension $F_r/F$ of degree $r$, 
all the elements of the Weyl group $W_r$ of $G(F_r)$ that give a non--trivial contribution to the geometric transfer from $G(F_r)$ to $G(F)$, belong to the Weyl 
group $W$ of $G(F)$. This is no longer true for our problem. So we have to consider, for each element $w\in W$, a variant of the morphism $\zeta_\chi$: let $\chi_{w,0}$ be the character $({^w\!\chi_0})\chi_0^{-1}$ of $\ES{T}$, and $\rho'_w$ the character of $\ES{I}'$ extending $\chi_{w,0}$ by inflation. Let also ${^w\!\rho}$ be the character $\rho\circ {\rm Int}_{n_w^{-1}}$ of the Iwahori subgroup ${^w\ES{I}}= {\rm Int}_{n_w}(\ES{I})$, where $n_w$ is a representative element of $w$ in $\ES{K}\cap N_G(T)$. We have a (similar) natural injective 
morphism of ${\Bbb C}$--algebras
$$
\zeta_{w,\chi}: \ES{Z}(G(F),{^w\!\rho})\rightarrow \ES{Z}(G'(F),\rho'_w).
$$
We also have an isomorphism of ${\Bbb C}$--algebras
$$
\ES{Z}(G(F),\rho)\rightarrow \ES{Z}(G(F), {^w\!\rho}),\, f\mapsto {^w\!f}= f\circ {\rm Int}_{n_w^{-1}}.
$$
For $f\in \ES{Z}(G(F),\rho)$, let $\bs{\zeta}_\chi(f)\in C^\infty_{\rm c}(G'(F))$ be the function defined by
$$
\bs{\zeta}_\chi(f)=\vert W'\vert^{-1} \sum_{w\in W}\zeta_{w,\chi}({^w\!f}).
$$
Note that 
the function $(\omega'_\chi)^{-1}\bs{\zeta}_\chi(f)\in C^\infty_{\rm c}(G'(F))$ depends only on $\chi_0$ (lemma of \ref{the theorem}), and not on the choice of the $W_{\chi_0}$--invariant character $\chi$ of 
$T(F)$ extending $\chi_0$. We denote it by
$$
\bs{\xi}_{\chi_0}(f)= (\omega'_\chi)^{-1}\bs{\zeta}_\chi(f).
$$ 
Our main result is the following:
\begin{thm}
Let $f\in \ES{Z}(G(F),\rho)$. The function $\bs{\xi}_{\chi_0}(f)$ is a transfer of $f$. 
\end{thm}
The proof of the theorem uses a local--global argument based on a simple trace formula, whose geometric side sees only the elliptic regular elements. Thus we have to reduce the theorem to the following ``elliptic version'' of it:
\begin{thm*}
Assume that the endoscopic group $G'$ of $G$ is elliptic, and let $f$ be a function in $\ES{Z}(G(F),\rho)$. Then $f$ and $\bs{\xi}_{\chi_0}(f)$ 
have matching elliptic strongly $G$--regular orbital integrals.
\end{thm*}
For technical reasons, we will have to prove a slight generalization of the elliptic version above. Let us fix another depth zero character $\psi_0$ of $\ES{T}$. For each element $w\in W$, the depth zero 
character $\xi_{w,0}={^w(\psi_0\chi_0)} \chi_0^{-1}$ of $\ES{T}$ inflates to a character $\rho'_{\xi_{w,0}}$ of $\ES{I}'$, and we have a natural injective morphism of ${\Bbb C}$--algebras
$$
\zeta_{w,\chi}: \ES{Z}(G(F), {^w\!(\rho_{\psi_0\chi_0})})\rightarrow \ES{Z}(G'(F),\rho'_{\xi_{w,0}}).
$$
For $f\in \ES{Z}(G(F),\rho_{\psi_0\chi_0})$, we define $\bs{\xi}_{\chi_0}= (\omega'_\chi)^{-1}\bs{\zeta}_\chi(f)\in C^\infty_{\rm c}(G'(F))$ as above. Then we have the ``variant with two characters'' of the theorem*:
\begin{thm**}
Assume that the endoscopic group $G'$ of $G$ is elliptic, 
and let $f$ be a function in $\ES{Z}(G(F),\rho_{\psi_0\chi_0})$. Then $f$ and $\bs{\xi}_{\chi_0}(f)$ 
have matching elliptic strongly $G$--regular orbital integrals.
\end{thm**}
We prove that the theorem is implied by the theorem**. The local--global argument is used in the proof of the theorem**. In fact, it allows us to prove a slightly stronger result: 
assuming the endoscopic group $G'$ is elliptic, for all functions $f\in \ES{Z}(G(F), \rho_{\psi_0\chi_0})$, the function $\bs{\xi}_{\chi_0}(f)$ is a transfer of $f$. 
We also prove that this result remains true without assuming $G'$ elliptic (i.e. the ``variant with two characters'' of the theorem).

\vskip2mm
\noindent{\bf 0.4.} --- 
The local--global argument in question is the existence of Hales local data for our local problem, obtained as in \cite{Hal} by comparing two simple trace formulas\footnote{Note that these ideas are due to Langlands: the comparison of two trace formulas already appears in the base change for $GL(2)$, and also the extraction of local informations by a local--global argument after separation of the ``strings of Hecke eigenvalues''.}, one for $G$ and the other for $G'$. In characteristic zero, Waldspurger pointed out to us that the spectral transfer, proved by Arthur in \cite{A}, provides these local data by a purely local argument. The proof of Waldspurger (see \ref{walds}) is much shorter but --- as he explained to us --- it is essentially the same because Arthur obtained the spectral transfer using a local--global argument very close to that of Hales. Note that (in characteristic zero) the proof of Waldspurger makes unnecessary the reduction\footnote{This reduction, however, had a heuristic role for us, as it was the descent formulas of \ref{a result of descent} that enabled us to find the right candidate for the transfer.} to an elliptic endoscopic group $G'$.

The rest of the proof is a simple adaptation of the one of Haines \cite{Hai2}. As in loc.~cit, we introduce some Labesse elementary functions --- for $G$ and for $G'$ --- 
whose orbital integrals on strongly regular elements and traces on irreducible smooth representations are very easy to calculate. In particular, for each $B$--dominant regular co--character $\nu$ of 
$T$, we have a function $\mathfrak{f}_\nu \in \ES{H}(G(F), \rho_{\psi_0\chi_0})$ and a function $\bs{\mathfrak{f}}'_\nu\in \sum_{w\in W}\ES{H}(G'(F), \rho'_{\xi_{w,0}})$ such that 
$(\omega'_\chi)^{-1}\bs{\mathfrak{f}}'_\nu$ is a transfer of $\mathfrak{f}_\nu$. Thanks to the existence of local data, these functions allow to rewrite the geometric identities of transfer in terms of characters identities. By linear independence of characters, viewed as linear combinations of characters on the set of $B$--dominant regular co--characters $\nu$, we obtain the same identity for $\nu=0$. In others words, denoting by $e_{\rho_{\psi_0\chi_0}}$ the unit element of the Hecke algebra $\ES{H}(G(F),\rho_{\psi_0\chi_0})$, and by $e_{\rho'_{\xi_{w,0}}}$ 
the unit element of the Hecke algebra $\ES{H}(G'(F),\rho'_{\xi_{w,0}})$, the function $(\omega'_\chi)^{-1}\vert W'\vert^{-1} \sum_{w\in W} e_{\rho'_{\xi_{w,0}}}$ is a transfer of $e_{\rho_{\psi_0\chi_0}}$. 
Now for $f\in \ES{Z}(G(F),\rho_{\psi_0\chi_0})$, thanks again to the existence of local data, the characterization of $\bs{\zeta}_{w,\chi}(f)$ for $w\in W$ by actions on principal series 
gives the desired result: $\bs{\xi}_{\chi_0}(f)= (\omega'_\chi)^{-1}\bs{\zeta}_{\chi}(f)$ is a transfer of $f$. This proves the ``variant with two characters'' of the theorem in the case of an elliptic endoscopic group $G'$, and hence in general.

\vskip2mm
\noindent{\bf 0.5.} --- 
Let us briefly describe the contents of the various sections.

In section \ref{main result} we introduce the objects and state the various versions of the theorem. 
We prove the reduction of the theorem to the theorem**. We also prove that we can suppose (in the theorem**) that the derived group $G_{\rm der}$ of $G$ is simply connected and the scheme--theoretic  center of $G$ is a torus; alternatively, we can suppose $G=G_{\rm AD}$. 

In section \ref{proof of the theorem assuming local data}, we introduce the elementary functions adapted to our problem, for $G$ and for $G'$ (we do not suppose $G'$ elliptic). Then we calculate their orbital integrals on strongly regular elements and their traces on smooth irreducible representations. Together with the existence of local data, these calculations allow us to prove the variant with two characters of the theorem. When ${\rm char}(F)=0$, we also give the proof of Waldspurger of the existence of local data.

In section \ref{existence of local data}, we prove the existence of local data when ${\rm char}(F)\geq 0$ as in \cite{Hal}, 
assuming $G$ semisimple and $G'$ elliptic (this reduction is possible, as we have seen in section \ref{main result}). 
For the last part of this introduction, let us change notations: we denote by $F_0$, $G_0$, $G'_0$ the objects previously denoted by $F$, $G$ and $G'$. We can 
realize our local situation $(G_0,G'_0)$ as the localization at a finite place $v_0$ of a global field $F$ 
of a global situation $(G,G')$, where $G$ is a $F$--split connected reductive group and $G'$ is a $F$--split elliptic endoscopic group of $G$. Thus we have 
$F_{v_0}\simeq F_0$, $G_{v_0}\;(=G\times_F F_{v_0})\simeq G_0$ and $G'_{\!v_0}\simeq G'_0$. Suppose for a moment that ${\rm char}(F_0)=0$. Hence $F$ is a number field. Since the global endoscopic group $G'$ is elliptic and unramified {\it at almost all}\,\footnote{This is a necessary condition to make the proof of Hales correct: one has to verify that the local situation of \cite{Hal} --- an unramified quasi--split $F_0$--group and an unramified elliptic endoscopic datum of it --- can be embeded in a global situation such that this condition is satisfied.}--- in fact at all --- finite places of $F$, it is possible to isolate it in the stabilized elliptic part of the trace formula for $G$. Here we used the (well--known?) following result: {\it two (global) endoscopic data of $G$ which are equivalent almost everywhere, are globally equivalent}. An unpublished proof of this result is due to Waldspurger. It will be written for publication shortly. In the meanwhile, the incredulous reader may consider this result as an additional hypothesis. Once we have isolated $G'$, we prove the existence of local data following \cite{Cl, Lab1} and \cite{Hal}. If ${\rm char}(F_0)=p>0$, the proof works in the same way provided that the following local results are available: the fundamental lemma for the unit (at almost all place $v$ of $F$); at a given place $v\neq v_0$, 
the existence of a non--zero linear combination $f_v$ of matrix coefficients of irreducible supercuspidal representations of $G(F_v)$ that transfers to $G'(F_v)$; at a given place $w\neq v_0$, the existence of a function on $G(F_w)$ that transfers to a non--zero linear combination $f'_w$ of matrix coefficients of irreducible supercuspidal representations of $G'(F_w)$. These local results are proved if $p$ is large  enough.

\vskip2mm
\noindent{\bf 0.6.} --- 
Let us emphasize that our proof is modeled on that of Haines \cite{Hai2}, even if the technical complications are relevant. 
We wrote this proof carefully because we hope to be able to apply it to much more general situations, e.g. for a character $\chi_0$ of positive depth, 
or for a general depth zero Bernstein block. 

In characteristic $p>0$, the result is weaker because of the restriction on $p$. It is also more interesting because we do not know the existence 
of transfer in general\footnote{Let us mention the work of Gordon and Hales \cite{GH}, who prove the existence of transfer for ``large $p$'', but with a different meaning: 
the constraint on $p$ is not effective, and for a given locally compact non--Archimedean field $F$ of characteristic $p$, it is not possible to say if it applies or not to $F$. In our method, 
the restriction on $p$ can be made explicit (the group $G$ being fixed): we only need the hypothesis (H1), (H2), (H3) in \ref{local data in char p} to be satisfied.}, but only for particular functions (cf. 0.5).
As a consequence, for $G$ split, our result shows the existence of transfer relative to certain split endoscopic groups $G'$ 
(see remark 3 of \ref{some results of Roche}), in characteristic $p>0$ large enough (see remark of \ref{local data in char p}),
for certain more functions than the one guaranteed by the fundamental lemma.

Note also that our result should make possible an explicit description of the spectral endoscopic transfer from $G'(F)$ to $G(F)$ for 
the representations in the depth zero Bernstein blocks involved here. 

\vskip3mm
\noindent{\bf Acknowledgements.} We warmly thank J.--L. Waldspurger for his sagacity to build counter--examples that destroyed our expectations. It is also thanks to his proof of the existence of local data in characteristic zero (see \ref{walds}) that we decided to extend our result to the positive characteristic case. We also thank J.--L. Colliot--Th\'el\`ene, T.~Haines, G.~Henniart, J.--P. Labesse, A.~Roche and R.~Weissauer for their helpful comments.

Finally, we thank the anonymous referee for his many helpful remarks and suggested corrections (mathematics as well as grammatical).

\section{Main result and various reductions}\label{main result}

In this section \ref{main result}, $F$ is a non--Archimedean local field, and $G$ is a connected reductive $F$--split (i.e. defined and split over $F$) group. 
We identify $G$ with $G(\overline{F})$ for a fixed algebraic closure $\overline{F}$ of $F$. 
Then $G(F)$ identifies with the subgroup $G(F^{\rm sep})^{\Gamma_F}\subset G$ of elements of $G(F^{\rm sep})$ fixed by 
$\Gamma_F={\rm Gal}(F^{\rm sep}/F)$, where $F^{\rm sep}$ is the separable closure of $F$ in $\overline{F}$.  

\subsection{Some results of Roche \cite{R}}\label{some results of Roche}
Let $\mathfrak{o}$ be the ring of integers of $F$, $\mathfrak{p}$ the maximal ideal of $\mathfrak{o}$, and $\mathfrak{K}= \mathfrak{o}/\mathfrak{p}$ the residue field. We denote by $\vert \,\vert_F$ the absolute value on $F$ normalized by $\vert \varpi \vert_F = q^{-1}$ for a uniformizer element $\varpi$ of $F$, where $q$ is the number of elements of $\mathfrak{K}$. Let $(B,T)$ be a Borel pair of $G$, defined and split over $F$, and $U$ the radical unipotent of $B$. Let $X$, resp. $\check{X}$, be the group of algebraic characters, resp. co--characters, of $T$. Let $\ES{T}= T(\mathfrak{o})$ be the maximal compact subgroup of $T$, and $\ES{T}_+\subset \ES{T}$ the kernel of the natural morphism (reduction modulo $\mathfrak{p}$) $T(\mathfrak{o})\rightarrow T(\mathfrak{K})$. We fix a character $\chi_0$ of $\ES{T}$. We assume that $\chi_0$ is {\it depth zero}, that is, trivial on $\ES{T}_+$.

We fix a hyperspecial maximal compact subgroup $\ES{K}$ of $G(F)$ corresponding to a hyperspecial vertex 
$x_0$ in the apartment of the Bruhat--Tits building of $G(F)$ associated to $T$. The base--point $x_0$ allows us to identify this 
apartment with the vector space 
$V= \check{X}\otimes_{\Bbb Z}{\Bbb R}$. Thus we have
$$
\ES{K}\cap B= (\ES{K}\cap T)(\ES{K}\cap U), \quad \ES{K}\cap T = \ES{T}.
$$

Let $\Phi=\Phi(G,T)$ be the set of roots of $T$ in $G$, and $\check{\Phi}$ the set of coroots. Let $\Phi_+=\Phi(B,T)$ be the set of roots of $T$ in $B$, and $\Delta\subset \Phi_+$ the corresponding basis of $\Phi$. Let $\ES{I}\subset \ES{K}$ be the Iwahori subgroup of $G(F)$ corresponding to the alc\^{o}ve
$$
\ES{A}= \{v\in V: 0<\alpha(v)<1,\, \forall \alpha \in \Phi_+\},\quad V= \check{X}\otimes_{\Bbb Z}{\Bbb R}.
$$
Hence $\ES{I}$ is in {\it good position} with respect to the Borel pair $(B,T)$: 
we have the triangular decomposition
$$
\ES{I}= (\ES{I}\cap \overline{U})(\ES{I}\cap T)(\ES{I}\cap U)
$$
where $\overline{U}$ is the unipotent radical of the Borel subgroup $\overline{B}$ opposite to $B$ with respect to $T$. Moreover, we have $\ES{T}= \ES{I}\cap T$ and $\ES{T}_+= \ES{I}_+ \cap T$, where $\ES{I}_+$ is the pro--unipotent radical of $\ES{I}$, and the inclusion $\ES{T}\subset \ES{I}$ induces an isomorphism $\ES{T}/\ES{T}_+\buildrel \simeq\over{\longrightarrow} \ES{I}/\ES{I}_+$. So the character $\chi_0$ of $\ES{T}$ can be viewed as a $1$--dimensional representation of $\ES{I}$, denoted $\rho=\rho_{\chi_0}$. It is given by
$$
\rho(\bar{u}t u)= \chi_0(t),\quad \bar{u}\in \ES{I}\cap \overline{U},\, t\in \ES{T}, u \in \ES{I}\cap U.
$$

Let $\chi$ be a character\footnote{i.e. a continuous morphism into ${\Bbb C}^\times$.} of $T(F)$ extending $\chi_0$. Let $\mathfrak{s} = \mathfrak{s}_{\chi_0}$ be the $G(F)$--inertial class $[T,\chi]_{G}$ of the cuspidal pair $(T,\chi)$ of $G(F)$. It does not depend on the choice of $\chi$, and we know by Roche \cite{R} that the pair $(\ES{I},\rho)$ is an $\mathfrak{s}$--type in the sense of Bushnell--Kutzko: for an irreducible (smooth, complex) representation $\pi$ of $G(F)$, we have ${\rm Hom}_{\ES{I}}(\rho,\pi\vert_{\ES{I}})\neq 0$ if and only if $\pi$ is isomorphic to a subquotient of a principal series representation $i_B^G(\chi\xi)$ of $G(F)$ for some character $\xi$ of $T(F)$ trivial on $\ES{T}$; where $i_B^G$ denotes the normalized parabolic induction functor. 

\begin{rem1}
{\rm In Roche's proof, there is a restriction on the residue characteristic on $F$. However, in the depth zero situation, no such restriction is needed (see \cite{Hai2}).
}
\end{rem1}

Let $N=N_G(T)$ be the normalizer of $T$ in $G$, and $W= N/T \;(= N(F)/T(F))$ the Weyl group of $G$. We denote by $N(F)_{\chi_0}$, resp. $W_{\chi_0}$, the subgroup of $N(F)$, resp. $W$, formed by the elements which fix $\chi_0$. The group $N(F)_{\chi_0}$ contains $T(F)$, and we have $W_{\chi_0}= N(F)_{\chi_0}/T(F)$. 

The character $\chi_0$ of $\ES{T}$ extends to a ($W_{\chi_0}$--invariant) character $\tilde{\chi}$ of $N(F)_{\chi_0}$. In fact, $\chi_0$ extends to a character $\tilde{\chi}_0$ of the subgroup $\ES{N}_{\chi_0}$ of $\ES{N}= N(\mathfrak{o})$ formed by the elements which fix $\chi_0$ \cite[6.11]{HL}. The choice a uniformizer element $\varpi$ of $F$ gives an indentification $N(F)= \check{X}\rtimes \ES{N}$, which induces an identification $N(F)_{\chi_0}=  \check{X}\rtimes \ES{N}_{\chi_0}$. Let $\tilde{\chi}= \tilde{\chi}_0^\varpi$ be the character of $N(F)_{\chi_0}$ extending $\tilde{\chi}_0$ trivially on $\check{X}$. It is $W_{\chi_0}$--invariant by construction, and the restriction $\tilde{\chi}\vert_{T(F)}$ is the 
character $\chi^\varpi_0$ of $T(F)$ extending $\chi_0$ trivially on $\check{X}$ for the identification $T(F)= \check{X}\times \ES{T}$ given by $\varpi$. From now on, we suppose 
$\chi = \chi_0^\varpi$ and $\tilde{\chi}= \tilde{\chi}_0^\varpi$ for a character $\tilde{\chi}_0$ of $\ES{N}_{\chi_0}$ extending $\chi_0$ and a uniformizer element $\varpi$ of $F$.

Let
$$
\Phi'= \{\alpha\in \Phi: \chi_0\circ \check{\alpha}\vert_{\mathfrak{o}^\times}=1\},
$$
$$
\check{\Phi}'= \{\check{\alpha}\in \check{\Phi}: \chi_0\circ \check{\alpha}\vert_{\mathfrak{o}^\times}=1\}.
$$
Hence we have $\check{\Phi}'= \{\check{\alpha} \in \Phi: \alpha \in \Phi'\}$. We also denote by $W'$ the subgroup of $W_{\chi_0}$ generated by the reflections $s_\alpha$ for $\alpha\in \Phi'$. By definition, $\check{\Phi}'$ is a {\it closed} sub--root system of $\check{\Phi}$. Hence $\Phi'$ is a sub--root system of $\Phi$, with Weyl group $W'$. Let $G'$ be a connected reductive group with root datum $(X,\Phi',\check{X},\check{\Phi}')$, defined and split over $F$. We identify $T$ with a maximal split torus $T'$ in $G'$. Then $\Phi'_+= \Phi'\cap \Phi_+$ is a subset of positive roots of $\Phi'$. It defines a basis $\Delta'$ of $\Phi'$, and a Borel subgroup $B'$ of $G'$ containing $T'$. We fix a hyperspecial maximal compact subgroup 
$\ES{K}'=\ES{K}_{G'}$ of $G'(F)$ corresponding to a hyperspecial vertex $x'_0$ in the apartment of the Bruhat--Tits building of $G'(F)$ associated to $T'$. 
The base--point $x'_0$ allows us to identify this apartment with $V\;(= \check{X}\otimes_{\Bbb Z}{\Bbb R})$. 
Let $\ES{I}'=\ES{I}_{G'}\subset \ES{K}'$ be the Iwahori subgroup of $G'(F)$ corresponding to the alc\^{o}ve
$$
\ES{A}'= \{v\in V: 0<\alpha(v)<1,\, \forall \alpha \in \Phi'_+\}.
$$
Put $C_{\chi_0}= \{w\in W_{\chi_0}: w(\Phi'_+)=\Phi'_+\}$. Then we have \cite[lemma 8.1]{R}
$$
W_{\chi_0}= W' \rtimes C_{\chi_0},
$$
and the choice of an $F$--\'epinglage of the based root datum $(X,\Phi',\Delta',\check{X},\check{\Phi}',\check{\Delta}')$ gives an embedding of $C_{\chi_0}$ in ${\rm Aut}_F(G')$. We put
$$
\smash{\wt{G}}'= G'\rtimes C_{\chi_0}.
$$
Then we have \cite{R}:
\begin{enumerate}
\item[(1)] $G'$ is an endoscopic group of $G$;
\item[(2)] $\ES{H}(\smash{\wt{G}}'(F),1_{\ES{I}'})\simeq \ES{H}(G'(F),1_{\ES{I}'})\wt{\otimes}\,{\Bbb C}[C_{\chi_0}]$;
\item[(3)] there is a support--preserving isomorphism of ${\Bbb C}$--algebras
$$
\Psi_{\tilde{\chi}}:\ES{H}(G(F),\rho)\buildrel\simeq \over{\longrightarrow} \ES{H}(\smash{\wt{G}}'(F),1_{\ES{I}'}).
$$
\item[(4)] if the scheme--theoretic center of $G$ is a torus, then we have $C_{\chi_0}=\{1\}$.
\end{enumerate}
In (1), $G'$ is the underlying group of an endoscopic datum $\underline{G}'=(G',\ES{G}',s)$ of $G$, where $s$ is an element of $\check{T}= {\rm Hom}(\check{X},{\Bbb C}^\times)$ whose connected centralizer $\check{G}_s$ in $\check{G}$ (the complex dual group of $G$) is a connected reductive group over ${\Bbb C}$ with root system $(\check{X},\check{\Phi}',X,\Phi')$, and $\ES{G}'= \check{G}_s \times W_F\;(\subset {^LG}= \check{G}\times W_F)$. 
In (2), the twisted tensor product $\wt{\otimes}$ is given by the conjugation action of $C_{\chi_0}$ on $\check{X}\rtimes W'$ (cf. \cite[8]{R}). In (3), $\ES{H}(G(F),\rho)$ is the convolution ${\Bbb C}$--algebra of compactly supported functions on $G(F)$ which are $\rho^{-1}$--spherical, and $\ES{H}(\smash{\wt{G}}'(F),1_{\ES{I}'})$ is the convolution ${\Bbb C}$--algebra of compactly supported $\ES{I}'$--biinvariant functions on $\smash{\wt{G}}'(F)$. The isomorphism $\Psi_{\tilde{\chi}}$ of (3) is described explicitely in \ref{Hecke algebra isomorphism} (proposition 2). In (4), the scheme--theoretic center of $G$ is the diagonalizable $F$--group scheme whose group of characters is $X/{\Bbb Z}\Phi$. We denote it by $\mathfrak{Z}_G$. Hence $\mathfrak{Z}_G$ is a torus if and only if $X/{\Bbb Z}\Phi$ is torsion free, which is equivalent to $\check{G}$ having simply connected derived group. Note that the reduced subscheme $(\mathfrak{Z}_G)_{\rm red}$ is a smooth group scheme that coincides with the ``center'' $Z(G)$ of $G$ in the sense of Borel (whose group of characters is the quotient of $X/{\Bbb Z}\Phi$ by its $p$--torsion). Hence $\mathfrak{Z}_G$ is a torus if and only if it is reduced and connected.

\begin{rem2}
{\rm Let us make precise (1). The group $\check{T}$ is the complex torus 
dual (in the sense of Langlands) to $T$: we have $X(\check{T})= \check{X}$ and $\check{X}(\check{T})= X$. Let $W_F\subset \Gamma_F$ be the absolute Weil group of $F$, and
$\tau_F:W_F\rightarrow W_F^{\rm ab}
\buildrel \simeq\over{\longrightarrow} F^\times$ 
the morphism of local class field theory. 
The Langland parameter of $\chi:T(F)\rightarrow {\Bbb C}^\times$ 
is the morphism $\varphi_\chi: W_F \rightarrow \check{T}$ defined by
$$
\nu(\varphi_\chi(\sigma))= \chi(\nu(\tau_F(\sigma))),\quad \nu \in \check{X},\,\sigma \in W_F.
$$
The restriction $\varphi_\chi\vert_{I_F}$ to the inertia group $I_F=\tau_F^{-1}(\mathfrak{o}^\times)$ of $F$, depends only on $\chi_0$. 
We denote it by $\varphi_{\chi_0}$. Since the group $\mathfrak{K}^\times$ is cyclic, the image ${\rm Im}(\varphi_{\chi_0})$ is a cyclic subgroup of $\check{T}$. 
Let us choose a generator $s$ of ${\rm Im}(\varphi_{\chi_0})$. Then we have
$$
\check{G}'= Z_{\check{G}}({\rm Im}(\varphi_{\chi_0}))^\circ = \check{G}_s.
$$
Let $\phi\in W_F$ be a Frobenius element such that $\tau(\phi)= \varpi$. Then $W_F = I_F \rtimes \langle \phi \rangle$ and $\varphi_\chi: W_F\rightarrow \check{T}$ 
is the morphism extending $\varphi_{\chi_0}$ trivially on $\langle \phi \rangle$.
}
\end{rem2}

\begin{rem3}
{\rm From remark 2, we see that any {\it split} endoscopic group $G'$ of $G$ such that $\check{G}'= \check{G}_s$ for a semisimple element $s$ whose order is finite 
and divides $\vert \kappa_F^\times\vert$,  
can be realized as the endoscopic group associated to a depth zero character $\chi_0$ of $\ES{T}$ as above.}
\end{rem3}

\subsection{Hecke algebra isomorphism}\label{Hecke algebra isomorphism}

Let $\ES{W}= N(F)/\ES{T}$ be the Iwahori--Weyl group of $G$. Recall that we have a (non--canonical) identification $N(F)= \check{X}\rtimes \ES{N}$ given by the choice of the uniformizer element $\varpi$ of $F$. Since $\ES{N}/\ES{T}= N/T =W$, this identification gives also an identification $\ES{W}= \check{X}\rtimes W$. Let $\ES{W}_{\chi_0}=\check{X}\rtimes W_{\chi_0}$ be the subgroup of $\ES{W}$ formed by the elements which fix $\chi_0$. As we have defined the subgroup $W'\subset W_{\chi_0}$, we define a subgroup $\ES{W}'\subset \ES{W}_{\chi_0}$ as follows. Let $\Phi'_{\rm aff}$ be the set of affine roots $a= \alpha + n$ where $\alpha\in \Phi'$ and $n\in {\Bbb Z}$. Then $\ES{W}'$ is the subgroup of $\ES{W}_{\chi_0}$ generated by the affine reflections $s_a$ for $a\in \Phi'_{\rm aff}$. 
Hence we have
$$
\ES{W}'= \check{X}\rtimes W'.
$$
Recall that we have defined in \ref{some results of Roche} an alc\^ove $\ES{A}'$ in $V$. This alc\^ove is a chamber for the decomposition of $V$ induced by $\Phi'_{\rm aff}$: $\ES{A}'$ is a connected component of $V\smallsetminus \bigcup_{a\in \Phi'_{\rm aff}}H_a$ where $H_a= \{v\in V: a(v)=0\}$. It also defines a partial order on $\Phi'_{\rm aff}$: for $a\in \Phi'_{\rm aff}$, write $a>0$ if $a(v)>0$ for all $v\in \ES{A}'$. Let $\Delta'_{\rm aff}$ be the subset of $\Phi'_{\rm aff}$ formed of those affine roots $a$ which are minimal positive for this partial order.  Put
$$
\ES{S}'= \{s_a: a\in \Delta'_{\rm aff}\},
$$
$$
\ES{C}_{\chi_0}= \{w\in \ES{W}_{\chi_0}: w(\ES{A}')= \ES{A}'\}.
$$
From \cite[lemma 6.2]{R}, we have the

\begin{lem}
The pair $(\ES{W}'\!,\ES{S}')$ is a Coxeter system --- i.e. $\ES{W}'$ is a Coxeter group with $\ES{S}'$ as fundamental set of generators ---, and there is a canonical decomposition $\ES{W}_{\chi_0}= \ES{W}' \rtimes \ES{C}_{\chi_0}$. The length function $l'$ on $\ES{W}'$ 
can be naturally extended to $\ES{W}_{\chi_0}$ in such a way that $\ES{C}_{\chi_0}$ consists of the length--zero elements.
\end{lem}

Let $\ES{H}(\ES{W}'\!,\ES{S}')$ be the Hecke algebra of the Coxeter system $(\ES{W}'\!,\ES{S}')$, that is the associative ${\Bbb C}$--algebra with ${\Bbb C}$--basis $\{T_w: w\in \ES{W}'\}$ and multiplication given by
$$
T_{w_1w_2}= T_{w_1}T_{w_2} \quad \hbox{if $l'(w_1w_1)=l'(w_1)+l'(w_2)$},
$$ 
$$
T_sT_s = (q-1)T_s + qT_1\quad \hbox{if $s\in \ES{S}'$}.
$$
The group $\ES{C}_{\chi_0}$ acts by ${\Bbb C}$--algebra automorphisms on $\ES{H}(\ES{W}'\!,\ES{S}')$ via
$$
\eta\cdot T_{w}= T_{\eta w \eta^{-1}},\quad \eta\in \ES{C}_{\chi_0},\, w\in \ES{W}'.
$$
This action naturally extends to an injective morphism of ${\Bbb C}$--algebras
$$
{\Bbb C}[\ES{C}_{\chi_0}] \rightarrow {\rm End}_{{\Bbb C}-{\rm alg}}(\ES{H}(\ES{W}'\!,\ES{S}')).
$$
Via this morphism, we get a ${\Bbb C}$--algebra
$$
\ES{H}(\ES{W}_{\chi_0},\ES{S}') = \ES{H}(\ES{W}'\!,\ES{S}')\wt{\otimes}\,{\Bbb C}[\ES{C}_{\chi_0}]
$$
where the multiplication in the twisted tensor product is given by
$$
(T_{w_1}\otimes e_{\eta_1})(T_{w_2}\otimes e_{\eta_2})= T_{w_1}T_{\eta_1w_2\eta_1^{-1}}\otimes e_{\eta_1\eta_2}
$$
for $w_1,\, w_2\in \ES{W}'$ and $\eta_1,\, \eta_2\in \ES{C}_{\chi_0}$, where we have denoted by $e_{\eta_1}$ the element of ${\Bbb C}[\ES{C}_{\chi_0}]$ given by $\eta_{\eta_1}(\eta)= \delta_{\eta_1,\eta}$ 
(Kronecker symbol). For $w\in \ES{W}'$ and 
$\eta\in \ES{C}_{\chi_0}$, we write the basis element $T_w\otimes e_\eta\in \ES{H}(\ES{W}_{\chi_0},\ES{S}')$ as $T_{w\eta}$. Thus we have $T_wT_\eta = T_{w\eta}$ and 
$T_\eta T_wT_{\eta^{-1}} = T_{\eta w \eta^{-1}}$

For an element $w\in \ES{W}_{\chi_0}=N(F)_{\chi_0}/\ES{T}$, let $f_{\tilde{\chi},w}$ be the the unique element of the Hecke algebra $\ES{H}(G(F),\rho)$ supported on $\ES{I}w\ES{I}=\ES{I}n\ES{I}$ and having value $q^{-l(w)/2}\tilde{\chi}(n)^{-1}$ at $n$ for any representative element $n$ of $w$ in $N(F)_{\chi_0}$, where $l$ is the length function on the affine Weyl group $\ES{W}=N(F)/\ES{T}$ of $G$. The functions $f_{\tilde{\chi},w}$ for $n\in \ES{W}_{\chi_0}$ generate the ${\Bbb C}$--vector space $\ES{H}(G(F),\rho)$. On the other hand, let $\phi'_w$ be the element $q^{-l'(w)/2}T_w\in \ES{H}(\ES{W}_{\chi_0}, \ES{S}')$, 
where $l'$ is the length function on $\ES{W}_{\chi_0}$ defined by the lemma. From \cite[theorem 6.3]{R}\footnote{For a depth zero character $\chi_0$, it is an easy consequence of a 
previous work of Morris \cite{M} which does not depend on the residual characteristic.}, we have

\begin{prop1}
There is an isomorphism of ${\Bbb C}$--algebras
$$
\vartheta_{\tilde{\chi}}:\ES{H}(\ES{W}_{\chi_0}, \ES{S}') \buildrel\simeq \over{\longrightarrow} \ES{H}(G(F),\rho)
$$ 
which sends $\phi'_w$ to $f_{\tilde{\chi},w}$ for all $w\in \ES{W}_{\chi_0}$. Moreover, the restriction $\vartheta_{\tilde{\chi}}\vert_{\ES{H}(\ES{W}'\!,\ES{S}')}$ 
is the unique injective morphism of ${\Bbb C}$--algebras $\ES{H}(\ES{W}'\!, \ES{S}') \rightarrow \ES{H}(G(F),\rho)$ which is support--preserving --- in the sense 
that the image of $T_w$ has support $\ES{I}w\ES{I}$ for all $w\in \ES{W}'$.
\end{prop1}

Now let us recall the dual interpretation of the proposition 1 \cite{R}. The group $\ES{W}'= \check{X}\rtimes W'$ is the Iwahori--Weyl group of $G'$. Put
$$
\ES{C}'= \{w\in \ES{W}': w(\ES{A}') = \ES{A}'\}.
$$
Recall we have the decomposition \cite[lemma 8.1]{R} $W_{\chi_0}= W' \rtimes C_{\chi_0}$. From loc.~cit., 
we have also the decomposition
$$
\ES{C}_{\chi_0}= \ES{C}' \rtimes C_{\chi_0}.
$$
This gives a decomposition
$$
{\Bbb C}[\ES{C}_{\chi_0}]= {\Bbb C}[\ES{C}']\, \wt{\otimes}\, {\Bbb C}[C_{\chi_0}]
$$
where the tensor product is defined by the conjugation action of $C_{\chi_0}$ on $\ES{C}'$. Hence we have
$$
\ES{H}(\ES{W}'\!,\ES{S}') \,\wt{\otimes}\, {\Bbb C}[\ES{C}_{\chi_0}]\, \simeq\, \left(\ES{H}(\ES{W}'\!,\ES{S}')\, \wt{\otimes}\, {\Bbb C}[\ES{C}']\right) \wt{\otimes}\,
{\Bbb C}[C_{\chi_0}].\leqno{(1)}
$$
By Iwahori--Matsumoto, we know there is a support--preserving isomorphism of ${\Bbb C}$--algebras
$$
 \ES{H}(\ES{W}'\!,\ES{S}')\, \wt{\otimes}\, {\Bbb C}[\ES{C}']\buildrel\simeq \over{\longrightarrow}\ES{H}(G'(F),1_{\ES{I}'}).\leqno{(2)}
$$
For an element $w\in \ES{W}_{\chi_0}=N(F)_{\chi_0}/\ES{T}$, let $f'_w$ be the unique element of 
the Hecke algebra $\ES{H}(\wt{G}'(F),1_{\ES{I}'})$ supported on $\ES{I}'w\ES{I}'= \ES{I}'n\ES{I}'$ and having value $q^{-l'(w)/2}$ at $n$, for any representative element $n$ 
of $w$ in $N(F)_{\chi_0}$. From (1) and (2), there is an isomorphism of ${\Bbb C}$--algebras
$$
\ES{H}(\ES{W}_{\chi_0},\ES{S}')= \ES{H}(\ES{W}'\!,\ES{S}') \,\wt{\otimes}\, {\Bbb C}[\ES{C}_{\chi_0}]
\buildrel\simeq \over{\longrightarrow} \ES{H}(\wt{G}'(F),1_{\ES{I}'}).\leqno{(3)}
$$
which sends $\phi'_w$ to $f'_w$ for all $w\in \ES{W}_{\chi_0}$. Composing (3) with the isomorphism $\vartheta_{\tilde{\chi}}^{-1}$ of the proposition 1, 
we obtain the

\begin{prop2}
There is an isomorphism of ${\Bbb C}$--algebras
$$
\Psi_{\tilde{\chi}}:\ES{H}(G(F),\rho)\buildrel\simeq \over{\longrightarrow} \ES{H}(\wt{G}'(F),1_{\ES{I}'})
$$ 
which sends $f_{\tilde{\chi},w}$ to $f'_w$ for all $w\in \ES{W}_{\chi_0}$.
\end{prop2}

\begin{rem}
{\rm 
As in \cite{R}, we chose a character $\tilde{\chi}$ of $N(F)_{\chi_0}$ extending $\chi_0$ of the form $\tilde{\chi}= \tilde{\chi}_0^\varpi$ for a 
character $\tilde{\chi}_0$ of $\ES{N}_{\chi_0}$ extending $\chi_0$ and a uniformizer element $\varpi$ of $F$. Now let $\chi^\natural$ be another character of $T(F)$ such that 
$\chi^\natural\vert_{\ES{T}}= \chi_0$. Then there exists a character $\tilde{\chi}^\natural$ of $N(F)_{\chi_0}$ which extends $\chi^\natural$ if and only if $\chi^\natural$ is $W_{\chi_0}$--invariant 
\cite[lemma 9.2.3]{HR}, i.e. 
if and only if $\chi^\natural = \xi \chi$ for a $W_{\chi_0}$--invariant unramified character $\xi$ of $T(F)$. For such a $\tilde{\chi}^\natural$, we can define in the same manner an isomorphism of ${\Bbb C}$--algebras 
$\Psi_{\tilde{\chi}^\natural}:\ES{H}(G(F),\rho)\buildrel\simeq \over{\longrightarrow} \ES{H}(\wt{G}'(F),1_{\ES{I}'})$.
}
\end{rem}

\subsection{Bernstein center}\label{Bernstein center}
Most of the material in this subsection come from \cite[4]{Hai2}.

We denote by $\mathfrak{X}(\mathfrak{s})$ the complex algebraic variety --- in fact a complex torus --- associated to $\mathfrak{s}$, and by $\mathfrak{R}^\mathfrak{s}(G(F))$ the corresponding category of (complex, smooth) representations of $G(F)$. 
The elements of $\mathfrak{s}$ are the $G(F)$--conjugacy classes $(T,\psi)_G$ of cuspidal pairs $(T,\psi)$ of $G(F)$ such that $\psi$ is a character of $T(F)$ which extends a $W$--conjugate of $\chi_0$.
Recall we chose a base--point $(T,\chi)_G\in \mathfrak{s}$ such that $\chi$ is $W_{\chi_0}$--invariant. 
Then, writing $\Psi(T)= {\rm Hom}(T(F)/\ES{T},{\Bbb C}^\times)$ --- the group of unramified characters of $G(F)$ ---, we have an isomorphism of algebraic varieties
$$
\Psi(T)/W_{\chi_0} \buildrel \simeq \over{\longrightarrow} \mathfrak{X}(\mathfrak{s}),\, \xi\mapsto (T,\xi\chi)_G.
$$
We also have an isomorphism of algebraic varieties
$$
\Psi(T)/W' \buildrel \simeq \over{\longrightarrow} \mathfrak{X}(\mathfrak{s}'),\, \xi\mapsto (T,\xi)_{G'},
$$
where $\mathfrak{s}'=[T,1]_{G'}$ is the $G'(F)$--inertial class of the cuspidal pair $(T,1)$ of $G'(F)$.
We deduce from these two isomorphisms a surjective morphism of algebraic varieties
$$
\mathfrak{X}(\mathfrak{s}')\rightarrow \mathfrak{X}(\mathfrak{s}),\, (T,\xi)_{G'}\mapsto (T, \xi\chi)_G,
$$
which gives dually an injective morphism of ${\Bbb C}$--algebras
$$
\eta_\chi:{\Bbb C}[\mathfrak{X}(\mathfrak{s})]\rightarrow {\Bbb C}[\mathfrak{X}(\mathfrak{s}')].\leqno{(1)}
$$

On the other hand, there is a natural equivalence between $\mathfrak{R}^\mathfrak{s}(G(F))$ and the category of left $\ES{H}(G(F),\rho)$--modules --- given by $\pi\mapsto \pi^\rho$ where $\pi^\rho$ denotes the $\rho$--isotypic component of $\pi$ ---, which gives an 
isomorphism between the centers of these categories (cf. \cite[3.1]{Hai2}), say
$$
\beta: {\Bbb C}[\mathfrak{X}(\mathfrak{s})]\buildrel \simeq \over{\longrightarrow} \ES{Z}(G(F),\rho).
$$
Here $\ES{Z}(G(F),\rho)$ denotes the center of the (unital) algebra $\ES{H}(G(F),\rho)$. We also have an isomorphism
$$
\beta': {\Bbb C}[\mathfrak{X}(\mathfrak{s}')]\buildrel \simeq \over{\longrightarrow} \ES{Z}(G'(F),1_{\ES{I}'}).
$$
The isomorphism $\Psi_{\tilde{\chi}}$ of \ref{some results of Roche}.(3) gives by restriction an isomorphism of ${\Bbb C}$--algebras
$$
\ES{Z}(G(F),\rho)\buildrel\simeq\over{\longrightarrow} \ES{Z}(\smash{\wt{G}}'(F),1_{\ES{I}'}).\leqno{(2)}
$$
On the other and, let
$$
\zeta_\chi = \zeta_\chi^{\ES{I}}: \ES{Z}(G(F),\rho) \rightarrow \ES{Z}(G'(F),1_{\ES{I}'})
$$
be the unique morphism of ${\Bbb C}$--algebras which makes the following diagram
$$
\xymatrix{
{\Bbb C}[\mathfrak{X}(\mathfrak{s})] \ar[r]^(.45)\beta \ar[d]_{\eta_\chi} & \ES{Z}(G(F),\rho) \ar[d]^{\zeta_\chi}\\
{\Bbb C}[\mathfrak{X}(\mathfrak{s}')] \ar[r]^(.45){\beta'} & \ES{Z}(G'(F),1_{\ES{I}'})
}\leqno{(3)}
$$
commutative. It is injective by construction. 

\begin{rem1}
{\rm If $W_{\chi_0}=W'$, i.e. if $\wt{G}'=G'$, it is proved in \cite[9]{HR} that $\zeta_{\chi}$ coincides with $\Psi_{\tilde{\chi}}\vert_{\ES{Z}(G(F),\rho)}$. In particular in that case, 
$\Psi_{\tilde{\chi}}\vert_{\ES{Z}(G(F),\rho)}$ depends only on $\chi$ and not on the choice of $\tilde{\chi}$. We shall see later (lemma bellow) that this 
remains true if we remove the hypothesis $W_{\chi_0}=W'$.  
}
\end{rem1}

Now let $\psi$ be a character of $T(F)$ which extends a $W$--conjugate of $\chi_0$. Put $\psi^*= \delta_B^{1/2}\psi$. Let $i_B^G(\psi)^\rho$ be the space of functions $\phi: G(F)\rightarrow {\Bbb C}$ such that
$$
\phi(tugk)= \psi^*(t) \phi(g) \rho(k),\quad t\in T(F),\, u\in U(F),\, g\in G(F),\, k\in \ES{I}.
$$
Then every function $f\in \ES{Z}(G(F),\rho)$ acts on the left on $i_B^G(\psi)^\rho$ by multiplication by the scalar
$$
\lambda_\psi(f)= \beta^{-1}(f)((T,\psi)_G),
$$
where $\beta^{-1}(f)$ is viewed as a regular function on the complex algebraic variety $\mathfrak{X}(\mathfrak{s})$. In particular, we have the equality $\lambda_{^w\psi}= \lambda_\psi$ as functions on $\ES{Z}(G(F),\rho)$. In the same manner, for an unramified character $\xi$ of $T'(F)=T(F)$, every function $f'\in \ES{Z}(G'(F),1_{\ES{I}'})$ acts on the left on the space $i_{B'}^{G'}(\xi)^{\ES{I}'}=i_{B'}^{G'}(\xi)^{1_{\ES{I}'}}$ by multiplication by the scalar
$$
\lambda'_{\xi}(f')= \beta'^{-1}(f')((T'\!,\xi)_{G'}).
$$
Now take $\psi = \xi \chi $ for an unramified character $\xi$ of $T(F)$. From (3) we have
$$
\lambda_{\xi \chi}(f)= \lambda'_\xi(\zeta_\chi(f)),\quad f\in \ES{Z}(G(F),\rho).\leqno{(4)}
$$
In other words, the morphism of ${\Bbb C}$--algebras $\zeta_\chi: \ES{Z}(G(F),\rho)\rightarrow \ES{Z}(G'(F),1_{\ES{I}'})$ is characterized by actions on principal series: for $f\in \ES{Z}(G(F),\rho)$, $\zeta_\chi(f)$ is the unique element of $\ES{Z}(G'(F),1_{\ES{I}'})$ which acts on the left on the space $i_{B'}^{G'}(\xi)^{\ES{I}'}$ by multiplication by the scalar 
$\lambda_{\chi\xi}(f)$, for all unramified characters $\xi$ of $T(F)$.

\begin{lem}
The center $\ES{Z}(\wt{G}'(F),1_{\ES{I}'})$ is contained in the center $\ES{Z}(G'(F),1_{\ES{I}'})$, and 
$\zeta_{\chi}$ coincides with $\left(\ES{Z}(\wt{G}'(F),1_{\ES{I}'}) \hookrightarrow \ES{Z}(G'(F),1_{\ES{I}'})\right)\circ \Psi_{\tilde{\chi}}\vert_{\ES{Z}(G(F),\rho)}$. In particular, 
the restriction $\Psi_{\tilde{\chi}}\vert_{\ES{Z}(G(F),\rho)}$ depends only on $\chi$ and not on the choice of $\tilde{\chi}$.
\end{lem}

\begin{proof}Let
$$
\Psi_\chi^T: \ES{H}(T(F),\chi)\buildrel\simeq\over{\longrightarrow} \ES{H}(T'(F),1_{\ES{T}'})
$$ 
be the isomorphism of ${\Bbb C}$--algebras obtained by replacing $G$ by $T$ in the definition of $\Psi_{\tilde{\chi}}$. From \cite[lemma 9.3]{R} (see also 
\cite[9.3.4]{HR}), 
there are morphisms of ${\Bbb C}$--algebras
$$
\theta_{B,\chi}: \ES{H}(T(F),\chi) \rightarrow \ES{H}(G(F),\rho)
$$
and
$$
\tilde{\theta}_{B'\!,1}:\ES{H}(T'(F),1_{\ES{T}'})\rightarrow \ES{H}(\wt{G}'(F),1_{\ES{I}'})
$$
which make the following diagram
$$
\xymatrix{
\ES{H}(G(F),\rho) \ar[r]^(.45){\Psi_{\tilde{\chi}}}_(.45){\simeq}  & \ES{H}(\wt{G}'(F),1_{\ES{I}'}) \\
\ES{H}(T(F),\chi) \ar[u]^{\theta_{B,\chi}}\ar[r]^(.45){\Psi_\chi^{T}}_(.45){\simeq} & \ES{H}(T'(F),1_{\ES{T}'})\ar[u]_{\tilde{\theta}_{B'\!,1}}
}\leqno{(5)}
$$
commutative. It induces a commutative diagram of functors
$$
\xymatrix{
\hbox{$\ES{H}(G(F),\rho)$--Mod} \ar[r]^(.46){(\Psi_{\tilde{\chi}})_*}_(.46){\simeq}  & \hbox{$\ES{H}(\wt{G}'(F),1_{\ES{I}'})$--Mod} \\
\hbox{$\ES{H}(T(F),\chi)$--Mod} \ar[u]^{(\theta_{B,\chi})_*}\ar[r]^(.48){(\Psi_\chi^{T})_*}_(.48){\simeq} & \hbox{$\ES{H}(T'(F),1_{\ES{T}'})$--Mod}\ar[u]_{(\tilde{\theta}_{B'\!,1})_*}
}\leqno{(6)}
$$
where ?--Mod denotes the category of left $?$--modules. 
From \cite[theorem 9.4]{R}, $(\theta_{B,\chi})_*$ coincides with the induction functor $i_B^G$ (via the natural equivalences of categories), 
and $(\tilde{\theta}_{B'\!,1})_*$ coincides with 
$i_{B'}^{\tilde{G}}= i_{G'}^{\wt{G}'}\circ i_{B'}^{G'}$ 
where
$$
i_{G'}^{\wt{G}'}: \hbox{$\ES{H}(G'(F),1_{\ES{I}'})$--Mod}\rightarrow \hbox{$\ES{H}(\wt{G}'(F),1_{\ES{I}'})$--Mod}
$$
is the ordinary induction functor. Here we used the equivalence of categories (cf. \cite[9, page 400]{R})
$$
\mathfrak{R}^{1_{\ES{I}'}}(\wt{G}'(F))\rightarrow \hbox{$\ES{H}(\wt{G}'(F),1_{\ES{I}'})$--Mod},\, W\mapsto W^{\ES{I}'}
$$
where $\mathfrak{R}^{1_{\ES{I'}}}(\wt{G}'(F))$ denotes the category of representations of $\wt{G}'(F)$ generated by their $\ES{I}'$--fixed vectors.
In other words, for all unramified characters $\xi$ of $T(F)$, we have
$$
(\Psi_{\tilde{\chi}})_*(i_B^G(\xi\chi)^{\ES{I}})= i_{B'}^{\wt{G}'}(\xi)^{\ES{I}'}.
$$
Hence for $f\in \ES{Z}(G(F),\rho)$, $\Psi_{\tilde{\chi}}(f)$ acts on the space
$$
i_{B'}^{\wt{G}'}(\xi)^{\ES{I}'}= \ES{H}(\wt{G}'(F),1_{\ES{I}'})\otimes_{\ES{H}(G(F),1_{\ES{I}'})} i_{B'}^{G'}(\xi)^{\ES{I}'}
$$
by multiplication by the scalar $\lambda_{\chi \xi}(f)$, for all unramified characters $\xi$ of $T(F)$. But 
$\zeta_\chi(f)= \zeta_\chi(f) \wt{\otimes}1 \in \ES{H}(\wt{G}'(F),1_{\ES{I}'})$ acts on the space $i_{B'}^{\wt{G}'}(\xi)$ by multiplication by the scalar 
$\lambda_{\chi \xi}(f)$, for all unramified characters $\xi$ of $T(F)$. Hence 
$$
\Psi_{\tilde{\chi}}(f)- \zeta_\chi(f)
$$
acts by $0$ on each irreducible subquotient of the $\ES{H}(\wt{G}'(F),1_{\ES{I}'})$--module 
$i_{B'}^{\wt{G}'}(\xi)^{\ES{I}'}$, for all unramified characters $\xi$ of $T'(F)$. This implies $\Psi_{\tilde{\chi}}(f)=\zeta_\chi(f)$ as elements 
of $\ES{H}(\wt{G}'(F),1_{\ES{I}'})$. The lemma is proved.  
\end{proof}

\begin{rem2}
{\rm Let $\chi^\natural$ be another $W_{\chi_0}$--invariant character of $T(F)$ extending $\chi_0$. It defines as above an injective morphism of ${\Bbb C}$--algebras 
$\eta_{\chi^\natural}:{\Bbb C}[\mathfrak{X}(\mathfrak{s})] \rightarrow {\Bbb C}[\mathfrak{X}(\mathfrak{s}')]$ which gives, via the isomorphisms $\beta$ and $\beta'$, an injective morphism 
of ${\Bbb C}$--algebras $\zeta_{\chi^\natural}: \ES{Z}(G(F),\rho)\rightarrow \ES{Z}(G'(F),1_{\ES{I}'})$. 
Moreover, if $\tilde{\chi}^\natural$ is a character of $N(F)_{\chi_0}$ that extends $\chi^\natural$, then we have $\zeta_{\chi^\natural}={\Psi_{\tilde{\chi}^\natural}}\vert_{\ES{Z}(G(F),\rho)}$ (cf. the remark of \ref{Hecke algebra isomorphism}).
}
\end{rem2}

\subsection{Conjugation by $w\in W$}\label{conjugation by W}
This subsection is a simple variant of \cite[4.3]{Hai2}.

Let $w$ be an element of $W$. It can be lifted to an element $n\in N \cap \ES{K}$. The character ${^w\!\chi_0}=\chi_0\circ {\rm Int}_{w^{-1}}$ of $\ES{T}$ can be identified as in \ref{some results of Roche} with a character of the Iwahori subgroup ${^w\ES{I}}={\rm Int}_n(\ES{I})$ of $G(F)$ --- the inflation of ${^w\!\chi_0}$ to ${^w\ES{I}}$. This character is ${^w\!\rho}= \rho\circ {\rm Int}_{n^{-1}}$. The group ${^w\ES{I}}$ and the character ${^w\!\rho}$ of ${^w\ES{I}}$, and more generally the function ${^w\!f}= f\circ {\rm Int}_{n^{-1}}\in \ES{H}(G(F),{^w\!\rho})$ for any $f\in \ES{H}(G(F),\rho)$, do not depend on the choice of the representative element $n\in \ES{K}\cap N$ of $w$: since $\ES{K}\cap T= \ES{T}$, another lift of $w$ in $N\cap \ES{K}$ belongs to $n\ES{T}= \ES{T}n$. The map
$$
\ES{H}(G(F),\rho)\rightarrow \ES{H}(G(F),{^w\!\rho}),\, f\mapsto {^w\!f}
$$
is an isomorphism of ${\Bbb C}$--algebras which restricts to an isomorphism
$$
[w]:\ES{Z}(G(F),\rho)\buildrel \simeq\over{\longrightarrow} \ES{Z}(G(F),{^w\!\rho}).
$$
From \cite[4.3.1]{Hai2}, the following diagram
$$
\xymatrix{{\Bbb C}[\mathfrak{X}(\mathfrak{s})] \ar[r]^(.45)\beta \ar@{=}[d] & \ES{Z}(G(F),\rho)\ar[d]^{[w]} \\
{\Bbb C}[\mathfrak{X}(\mathfrak{s})] \ar[r]^(.45){\beta} & \ES{Z}(G(F), {^w\!\rho})
}\leqno{(1)}
$$
is commutative. Here, the lower $\beta$ is the isomorphism of ${\Bbb C}$--algebras given by the natural equivalence between $\mathfrak{R}^{\mathfrak{s}}(G(F))$ and the category of left $\ES{H}(G(F),{^w\!\rho})$--modules. As for the morphism $\zeta_\chi$ in \ref{Bernstein center}.(1), the isomorphism of ${\Bbb C}$--algebras $[w]: \ES{Z}(G(F),\rho)\rightarrow \ES{Z}(G(F),{^w\!\rho})$ can be characterized by actions on principal series (cf. \ref{Bernstein center}): for $f\in \ES{Z}(G(F),\rho)$, $[w]( f) = {^w\!f}$ is the unique element of $\ES{Z}(G(F),{^w\!\rho})$ which acts on the left on $i_{^w\!B}^G({^w\psi})^{^w\!\rho}\simeq i_B^G(\psi)^\rho$ by multiplication by the scalar
$$
\lambda_{^w\psi}({^w\!f})\;\left(=\lambda_\psi({^w\!f})= \lambda_\psi(f)\right),
$$
for all characters $\psi$ of $T(F)$ extending a $W$--conjugate of $\chi_0$.

If $\psi = \xi \chi$ for an unramified character $\xi$ of $T(F)$, then we have ${^w\psi}= ({^w\xi}) \chi_w \chi$ with
$$
\chi_w= ({^w\chi})\chi^{-1}.
$$
Note that $\chi_w$ is a character of $T(F)$ extending $\chi_{w,0}= ({^w\chi_0})\chi_0^{-1}$. If $w\in W_{\chi_0}$, then we have 
$\chi_{w,0}=1$ and $\chi_w=1$. Moreover, if $w_0\in W_{\chi_0}$, since ${^{w_0}\chi}=\chi$ we have
$$
\chi_{ww_0}=\chi_w,\quad \chi_{w_0w}= \chi_{w_0ww_0^{-1}}.
$$
Let $\mathfrak{s}'_w=\mathfrak{s}'_{\chi_{w,0}}$ be the $G'(F)$--inertial class $[T(F),\chi_w]_{G'}$ of the cuspidal pair $(T(F), \chi_w)$ of $G'(F)$. 
Recall that we have an isomorphism of algebraic varieties
$$
\Psi(T)/W' \buildrel \simeq \over{\longrightarrow} \mathfrak{X}(\mathfrak{s}'),\, \xi\mapsto (T,\xi)_{G'}.
$$
We also have an isomorphism of algebraic varieties
$$
\Psi(T)/W'_{^w\!\chi_0} \buildrel \simeq \over{\longrightarrow} \mathfrak{X}(\mathfrak{s}'_w),\, \xi\mapsto (T,({^w\xi})\chi_w)_{G'},
$$
where $W'_{^{w\!\chi_0}}= W' \cap wW_{\chi_0} w^{-1}$ is the subgroup of $W'$ formed by the elements which fix ${^w\!\chi_0}$. 
The surjective morphism of algebraic varieties
$$
\mathfrak{X}(\mathfrak{s}'_w)\rightarrow \mathfrak{X}(\mathfrak{s}'),\, (T,({^w\xi}) \chi_w)_{G'}\mapsto (T,\xi)_{G'}
$$
gives by duality an injective morphism of ${\Bbb C}$--algebras
$$
 \mu'_w=\mu'_{w,\chi}: {\Bbb C}[\mathfrak{X}(\mathfrak{s}')]\rightarrow {\Bbb C}[\mathfrak{X}(\mathfrak{s}'_w)],
$$
which is the identity if $w\in W'$. The natural isomorphism $\ES{T}/\ES{T}_+ \buildrel \simeq \over{\longrightarrow} \ES{I}'/\ES{I}'_+$ allows us to view 
$\chi_{w,0}$ as a character $\rho'_w= \rho'_{\chi_{w,0}}$ of $\ES{I}'$ trivial on $\ES{I}'_+$ (cf. \ref{some results of Roche}) --- the inflation of $\chi_{w,0}$ to $\ES{I}'$. Let
$$
\{w\}'=\{w\}'_\chi: \ES{Z}(G'(F),1_{\ES{I}'})\rightarrow \ES{Z}(G'(F),\rho'_w)
$$
be the injective morphism of ${\Bbb C}$--algebras which makes the following diagram
$$
\xymatrix{{\Bbb C}[\mathfrak{X}(\mathfrak{s}')] \ar[r]^(.45){\beta'} \ar[d]_{\mu'_{w}} & \ES{Z}(G'(F),1_{\ES{I}'}) \ar[d]^{\{w\}'}\\
{\Bbb C}[\mathfrak{X}(\mathfrak{s}'_w)] \ar[r]^(.45){\beta'} & \ES{Z}(G'(F), \rho'_w)
}
$$
commutative. Here the lower $\beta'$ is the isomorphism of ${\Bbb C}$--algebras 
given by the natural equivalence between $\mathfrak{R}^{\mathfrak{s}'_w}(G'(F))$ 
and the category of left $\ES{H}(G'(F), \rho'_w)$--modules. Note that if $w\in W_{\chi_0}$, 
then $\{w\}'$ is an automorphism of $\ES{Z}(G'(F),1_{\ES{I}'})$, which is the identity if $w\in W'$.

\begin{nota}
{\rm 
For all unramified characters $\xi$ of $T(F)$, we put $\{w\}'(\xi)= ({^w\xi})\chi_w$.
}
\end{nota}

The morphism $\{w\}'$ can be characterized by actions on principal series: for $f'\in \ES{Z}(G'(F),1_{\ES{I}'})$, $\{w\}'(f')$ is the unique element of 
$\ES{Z}(G'(F),\rho'_w)$ which acts on the left on $i_{B'}^{G'}(\{w\}'(\xi))^{\rho'_w}$ by multiplication by the scalar
$$
\lambda'_{\{w\}'(\xi)}(\{w\}'(f'))= \lambda'_{\xi}(f'),
$$
for all unramified characters $\xi$ of $T'(F)$.

For $w_0\in W_{\chi_0}$, since ${^{w_0}\!\chi}=\chi$, we have $\chi_{w_0w}= {^{w_0}(\chi_w)}$ and
$$
\{w_0w\}'(\xi)= ({^{w_0w}\xi})\chi_{w_0w}={^{w_0}(\{w\}'(\xi))}.
$$
In particular for $w'\in W'$, we have $\mathfrak{s}'_{w'w}= \mathfrak{s}'_w$ and $\mu'_{w'w}= \mu'_{w}$, and the morphism 
$\{w'w\}': \ES{Z}(G'(F),1_{\ES{I}'})\rightarrow \ES{Z}(G'(F), \rho'_{w'w})$ is given by
$$
\{w'w\}' = [w']'_w\circ \{w\}'
$$
where $[w']'_w: \ES{Z}(G'(F),\rho'_w) \rightarrow \ES{Z}(G'(F),\rho'_{w'w})$ is the isomorphism of ${\Bbb C}$--algebras 
which makes the following 
diagram
$$
\xymatrix{
{\Bbb C}[\mathfrak{X}(\mathfrak{s}'_w)] \ar[r]^(.45){\beta'} \ar@{=}[d] & \ES{Z}(G'(F),\rho'_w) \ar[d]^{[w']'_w}\\
{\Bbb C}[\mathfrak{X}(\mathfrak{s}'_w)] \ar[r]^(.45){\beta'}  & \ES{Z}(G'(F),\rho'_{w'w})
}
$$
commutative. As for the morphism $\{w\}'$ introduced above, the isomorphism $[w']'_w$ can be characterized by actions on principal series: 
for $f'\in \ES{Z}(G'(F),\rho'_w)$, $[w']_w(f')$ is the unique element of 
$\ES{Z}(G'(F),\rho'_{w'w})$ which acts on the left on $i_{B'}^{G'}(\{w'w\}'(\xi))^{\rho'_{w'w}}$ by multiplication by the 
scalar
$$
\lambda'_{\{w'w\}'(\xi)}([w']'_w(f'))= \lambda'_{\{w\}(\xi)}(f')
$$
for all unramified characters $\xi$ of $T'(F)$. 

On the other hand, for $w_0\in W_{\chi_0}$, since $\chi_{ww_0,0}= \chi_{w,0}$, we have $\mathfrak{s}'_{ww_0}= \mathfrak{s}'_w$ and 
$\mu'_{ww_0}= \mu'_w$. We also have $\rho'_{ww_0}=\rho'_w$ and $\chi_{ww_0}=\chi_w$, and
$$
\{ww_0\}'(\xi)= ({^{ww_0}\xi})\chi_w = {^w(\xi_{w_0})}\cdot \{w\}'(\xi),\quad \xi_{w_0}= ({^{w_0}\xi})\xi^{-1},
$$
for all unramified characters $\xi$ of $T(F)$. The automorphism of algebraic varieties
$$
{\Bbb C}[\mathfrak{X}(\mathfrak{s}'_w)]\rightarrow {\Bbb C}[\mathfrak{X}(\mathfrak{s}'_w)],\, (T'\!,({^{ww_0}\xi})\chi_w)_{G'} \mapsto (T'\!,({^w\xi})\chi_w)_{G'}
$$
gives by duality, and via the isomorphism $\beta':{\Bbb C}[\mathfrak{X}(\mathfrak{s}'_w)]\rightarrow \ES{Z}(G'(F), \rho'_w)$, an automorphism of 
${\Bbb C}$--algebras
$$
\{w_0\}'_w: \ES{Z}(G'(F),\rho'_w) \rightarrow \ES{Z}(G'(F),\rho'_w).
$$
By construction we have
$$
\{ww_0\}'= \{w_0\}'_w \circ \{w\}'.
$$
As before, it can be characterized via actions on principal series: for a function $f'\in \ES{Z}(G'(F),\rho'_w)$, $\{w_0\}'_w(f')$ is the unique element of 
$\ES{Z}(G'(F),\rho'_w)$ which acts on the left on $i_{B'}^{G'}({^w(\xi_{w_0})}\psi)^{\rho'_w}$ with $\psi =\{w\}'(\xi)$ by multiplication by the scalar
$$
\lambda'_{{^w\!(\xi_{w_0})}\psi}(\{w_0\}'_w(f'))= \lambda'_\psi(f')
$$
for all unramified characters $\xi$ of $T(F)$. Note that in general $\{w_0\}'_w \neq {\rm id}$, even if $w_0\in W'$. 

\vskip1mm
The surjective morphism of algebraic varieties
$$
\mathfrak{X}(\mathfrak{s}'_w)\rightarrow \mathfrak{X}(\mathfrak{s}),\, (T,({^w\xi}) \chi_w)_{G'}\mapsto (T,\xi \chi)_G
$$
gives by duality an injective morphism of ${\Bbb C}$--algebras
$$
\eta_{w,\chi}:{\Bbb C}[\mathfrak{X}(\mathfrak{s})]\rightarrow {\Bbb C}[\mathfrak{X}(\mathfrak{s}'_w)]. 
$$
Let
$$
\zeta_{w,\chi}: \ES{Z}(G(F), {^w\!\rho})\rightarrow \ES{Z}(G'(F),\rho'_w)
$$
be the injective morphism of ${\Bbb C}$--algebras which makes the following diagram
$$
\xymatrix{{\Bbb C}[\mathfrak{X}(\mathfrak{s})] \ar[r]^(.45){\beta} \ar[d]_{\eta_{w,\chi}} & \ES{Z}(G(F),{^w\!}\rho)\ar[d]^{\zeta_{w,\chi}} \\
{\Bbb C}[\mathfrak{X}(\mathfrak{s}'_w)] \ar[r]^(.45){\beta'} & \ES{Z}(G'(F), \rho'_w)
}\leqno{(2)}
$$
commutative. By construction, the following diagram
$$
\xymatrix{
\ES{Z}(G(F),\rho) \ar[r]^{[w]} \ar[d]_{\zeta_\chi} & \ES{Z}(G(F),{^w\!\rho})\ar[d]^{\zeta_{w,\chi}} \\
\ES{Z}(G'(F),1_{\ES{I}'}) \ar[r]^{\{w\}'} & \ES{Z}(G'(F),\rho'_w)
}\leqno{(3)}
$$
is commutative. As for the morphism $\zeta_{\chi}$ of \ref{Bernstein center}.(3), the morphism $\zeta_{w,\chi}$ above can be characterized by actions on principal series: for 
$f\in \ES{Z}(G(F),\rho)$, $\zeta_{w,\chi}({^wf})$ is the unique element of $\ES{Z}(G'(F),\rho'_w)$ which acts on the left on 
$\iota_{B'}^{G'}(\xi)^{\rho'_w}$ by multipication by the scalar
$$
\lambda'_\xi(\zeta_{w,\chi}({^w\!f}))= \lambda_{\xi\chi}(f)\leqno{(4)}
$$
for all characters $\xi$ of $T'(F)=T(F)$ extending a $W'$--conjugate of $\chi_{w,0}$ (note that since $W'\subset W_{\chi_0}$, 
$\xi\chi$ extends a $W'$--conjugate of ${^w\!(\chi_0)}$, and in particular it extends a $W$--conjugate of $\chi_0$). From the characterization of the morphisms $\{w\}$ and $[w]'$ 
by actions on principal series, we have also: for $f\in \ES{Z}(G(F),\rho)$, $\zeta_{w,\chi}({^w\!f})\in\ES{Z}(G(F),\rho'_w)$ acts on the left on 
$i_{B'}^{G'}(\{w\}'(\xi))^{\rho'_w}$ by multiplication by the scalar
$$
\lambda'_{\{w\}'(\xi)}(\zeta_{w,\chi}({^w\!f}))=  \lambda'_{\{w\}'(\xi)}(\{w\}'(\zeta_\chi(f))= \lambda'_\xi(\zeta_\chi(f))=\lambda_{\xi\chi}(f).\leqno{(5)}
$$
for all unramified characters $\xi$ of $T'(F)$. Note that since $\{w\}'(\xi)\chi= {^w\!(\xi\chi)}$, we have
$$\lambda_{\xi\chi}(f)= 
\lambda_{^w\!(\xi\chi)}(f) = \lambda_{\{w\}'(\xi)\chi}(f).
$$
 
 \begin{rem1}
{\rm 
For $w'\in W'$, since $\{w'\}'$ is the identity of $\ES{Z}(G'(F),1_{\ES{I}'})$, 
we have $\zeta_{w'\!,\chi}\circ [w']= \zeta_\chi$. More generally, we have
$$
\zeta_{w'w,\chi}\circ [w'w]= \{w'w\}'\circ \zeta_{\chi} = [w']'_w \circ \{w\}' \circ \zeta_{\chi}= [w']'_w\circ \zeta_{w,\chi}\circ [w].
$$
On the other hand, for $w_0\in W_{\chi_0}$, we have
$$
\zeta_{ww_0,\chi} \circ [ww_0]= \{ww_0\}'\circ \zeta_\chi = \{w_0\}'_w \circ \{w\}' \circ \zeta_\chi = \{w_0\}'_w\circ \zeta_{w,\chi}\circ [w].
$$ 
}
\end{rem1}

 \begin{rem2}
{\rm 
As in \ref{some results of Roche}, the character ${^w\chi_0}$ of $\ES{T}$ defines a $F$--split connected reductive group 
$G'_w$ with maximal $F$--split torus $T'_w=T$, an Iwahori subgroup $\ES{I}'_w$ of $G'_w(F)$, and a $F$--group $\wt{G}'_w = G'_w \rtimes C_{^w\!\chi_0}$. More precisely, $G'_w$ is given by the root datum 
$(X, w(\Phi'),\check{X}, w(\check{\Phi}'))$, $\ES{I}'_w$ corresponds to the alcove $w(\ES{A}')\subset V$, and $C_{^w\!\chi_0}= wC_{\chi_0}w^{-1}$. In other words, $w$ induces by transport of structure an $F$--isomorphism $\wt{G}' \buildrel\simeq \over{\longrightarrow} \wt{G}'_w$ which allows us to identify $G'_w(F)$ with $G(F)$ and $\ES{I}'_w$ with $\ES{I}'$. With these identifications, we have an injective morphism of ${\Bbb C}$--algebras
$$
\zeta_{^w\chi}^{^w\ES{I}}:\ES{Z}(G(F),{^w\!\rho})\rightarrow \ES{Z}(G'_w(F),1_{\ES{I}'_w})=\ES{Z}(G'(F),1_{\ES{I}'})
$$
which is given by $\zeta_{^w\chi}^{^w\ES{I}}= \zeta_\chi \circ [w]^{-1}$. Hence we have 
$$
\{w\}'\circ \zeta_{^w\chi}^{^w\ES{I}}= \zeta_{w,\chi}.
$$
In particular if $w\in W_{\chi_0}$, we have $\zeta_{w,\chi}=\{w\}'\circ \zeta^{^w{\ES{I}}}_\chi$, and if $w\in W'$, we have $\zeta_{w,\chi}= \zeta^{^w\ES{I}}_\chi$.
}
\end{rem2}

\subsection{Conjugation by $w\in W'$}\label{conjugation by W'}
Let $w$ be an element of $W'$. We choose a representative element $n'_w$ of $w$ in $N'\cap \ES{K}'$, and we put ${^w\ES{I}'}= {^w\!(\ES{I}')}={\rm Int}_{n'_w}(\ES{I}')$. 
For a function $f'\in \ES{Z}(G'(F),1_{\ES{I}'})$, we put $[w]'(f') = {^{w}\!f'}= f'\circ {\rm Int}_{n'^{-1}_w}$. Since ${^w\chi}=\chi$, by replacing $\ES{I}$ by ${^w\ES{I}}$ and $\ES{I}'$ by ${^w\ES{I}'}$ in the definition of $\zeta_\chi$, we obtain an injective morphism of ${\Bbb C}$--algebras
$$
{^w\!\zeta_\chi}: \ES{Z}(G(F), {^w\!\rho})\rightarrow \ES{Z}(G'(F),1_{^w{\ES{I}'}})
$$
which makes the following diagram
$$
\xymatrix{{\Bbb C}[\mathfrak{X}(\mathfrak{s})] \ar[r]^(.45){\beta} \ar[d]_{\eta_{\chi}} & \ES{Z}(G(F),{^w\!}\rho)\ar[d]^{^w\!\zeta_\chi} \\
{\Bbb C}[\mathfrak{X}(\mathfrak{s}')] \ar[r]^(.45){\beta'} & \ES{Z}(G'(F), 1_{^w\ES{I}'})
}\leqno{(1)}
$$
commutative. Here $\beta'$ is the isomorphism of ${\Bbb C}$--algebras given by the natural equivalence between $\mathfrak{R}^{\mathfrak{s}'}(G'(F))$ and the category of left $\ES{H}(G'(F),1_{^w\ES{I}'})$--modules. By construction, the following diagram
$$
\xymatrix{
\ES{Z}(G(F),\rho) \ar[r]^{[w]} \ar[d]_{\zeta_\chi} & \ES{Z}(G(F), {^w\!\rho}) \ar[d]^{^w\!\zeta_{\chi}}\\
\ES{Z}(G'(F),1_{\ES{I}'}) \ar[r]^{[w]'} & \ES{Z}(G'(F), 1_{^w{\ES{I}'}})
}\leqno{(2)}
$$
is commutative. 

\begin{rem}
{\rm 
From the remark 1 of \ref{conjugation by W}, we have $\zeta^{^{w}\ES{I}}_{\chi}= \zeta_{w,\chi}= \zeta_\chi \circ [w]^{-1}$. Hence we have 
${^w\zeta_\chi}= [w]'\circ \zeta_{w,\chi}$.
}
\end{rem}

\subsection{The theorem}\label{the theorem}
The theory of orbital integrals and their $\kappa$--variants has been mainly developed in characteristic zero by Harish--Chandra, Langlands, Kottwitz and Shelstad. More recently, Labesse 
simplified the study of stable conjugation of semisimple elements thanks to a systematic use of non--abelian cohomological constructions. Ultimately, all these constructions are done in the derived 
category of bounded complexes of diagonalizable groups. Most of these constructions --- in particular all those used herein --- can be realized using complexes of tori, making them valid in characteristic $p>0$ (see \cite{LL}).

For a semisimple element $\gamma\in G(F)$, let $G_\gamma= G^{\gamma,\circ}$ be its connected (i.e. stable) centralizer. For a function $f\in C^\infty_{\rm c}(G(F))$, we define the {\it orbital integral}
$$
{\bf O}_\gamma(f)= \int_{G_\gamma(F)\backslash G(F)}f(g^{-1}\gamma g) \textstyle{dg\over dg_\gamma}.
$$
It depends on the choice of Haar measures $dg$ on $G(F)$ and $dg_\gamma$ on $G_\gamma(F)$. Since the orbit $\{g^{-1}\gamma g: g\in G(F)\}\subset G(F)$ is closed for the $\mathfrak{p}$--adic topology, the integral is absolutely convergent. Two semisimple elements $\gamma,\,\gamma'\in G(F^{\rm sep})$ are {\it stably conjugate (in $G$)} if there exists an element $g\in G(F^{\rm sep})$ such that $g^{-1}\gamma g= \gamma'$ and $g\sigma(g)^{-1}\in G_\gamma$ for all $\sigma \in \Gamma_F$. In that case, the connected centralizers $G_\gamma$ and $G_{\gamma'}$ are inner forms of each other and we require the Haar measures $dg_\gamma$ and $dg_{\gamma'}$ on $G_\gamma(F)$ and on $G_{\gamma'}(F)$ to be compatible with each other, the Haar measure $dg$ on $G(F)$ being fixed. 
If moreover $\gamma$ is {\it regular}, i.e. if the connected centralizer $G_\gamma$ is a torus, then we also define the {\it stable orbital integral}
$$
{\bf SO}_\gamma(f)= \sum_{\gamma'} {\bf O}_{\gamma'}(f)
$$
where $\gamma'$ runs over the elements of $G(F)$ which are stably conjugate to $\gamma$ modulo conjugation by $G(F)$.

Two semisimple elements $\delta\in G'(F^{\rm sep})$ and $\gamma\in G(F^{\rm sep})$ {\it correspond} if their conjugacy classes (over $F^{\rm sep})$ correspond via the natural surjective morphism $
T'/W' \rightarrow T/W$. If $(\delta,\gamma)\in G'(F)\times G(F)$ is a pair of corresponding elements, the connected centralizer $G'_\delta$ of $\delta$ in $G'$ is an inner form of $G_\gamma$. 
If moreover $\gamma$ is strongly regular, i.e. if the centralizer $G^\gamma$ is a torus, then $\delta$ is said to be {\it strongly $G$--regular}.

A function $f'\in C^\infty_{\rm c}(G'(F))$ is said to be an {\it endoscopic transfer}, or simply a {\it transfer}, of a function $f\in C^\infty_{\rm c}(G(F))$ if $f$ and $f'$ have matching strongly regular orbital integrals: for all (semisimple) strongly $G$--regular elements $\delta \in G'(F)$, we have
$$
{\bf SO}_\delta(f') = \sum_\gamma \Delta(\delta,\gamma) {\bf O}_\gamma(f)\leqno{(1)}
$$
where $\gamma$ runs over the elements of $G(F)$ which correspond (in $G$) to $\delta$ modulo conjugation by $G(F)$. 
In particular, this means 
that ${\bf SO}_\delta(f')=0$ if $\delta$ does not correspond to any element of $G(F)$. In (1), $G'_\delta$ is an inner form of $G_\gamma$, and we choose compatible Haar measures on $G'_\delta(F)$ and on $G_\gamma(F)$. We use them, and the Haar measures $dg$ on $G(F)$ and $dg'$ on $G'(F)$ normalized by
$$
{\rm vol}(\ES{I},dg)=1={\rm vol}(\ES{I}',dg'),
$$
to define the stable orbital integral ${\bf SO}_\delta$ on $G'(F)$ and the orbital integral ${\bf O}_\gamma$ on $G(F)$. The transfer factor $\Delta(\delta,\gamma)$ is the canonical transfer factor defined by the choice of $\ES{K}$ as in \cite[6.3]{MW}\footnote{In loc.~cit., the orbital integrals on $G(F)$ and the stable orbital integrals on $G'(F)$ are the normalized ones. Hence the transfer factor $\Delta(\delta, \gamma)$ of (1) is the canonical transfer factor of loc.~cit. multiplied by the factor $\Delta_{\rm IV}(\delta,\gamma)=\vert D_G(\gamma)\vert^{1/2}_F \vert D_{G'}(\delta)\vert_F^{-1/2}$. 
Moreover, in loc.~cit., the base field $F$ is a finite extension of ${\Bbb Q}_p$ (without any restrictive hypothesis on the residual characteristic $p$), but they are also valid in characteristic $p>0$, cf. the beginning of \ref{the theorem}.}. 
In other words, $\Delta(\delta,\gamma)$ is the transfer factor which makes ${\rm vol}(\ES{K}',dg')^{-1}1_{\ES{K}'}$ a transfer of ${\rm vol}(\ES{K},dg)^{-1}1_\ES{K}$ 
(fundamental lemma).

\begin{nota}
{\rm 
We denote by $\wt{\bf O}_\delta(f)$ the right hand side of (1).
}
\end{nota}

Recall that for all elements $w\in W$, we have defined an injective morphism of ${\Bbb C}$--algebras $\{w\}':\ES{Z}(G'(F),1_{\ES{I}'})\rightarrow 
\ES{Z}(G'(F),\rho'_w)$ --- cf. \ref{conjugation by W}. We are interested in the composition morphism
$$
\zeta_{w,\chi}\circ [w]= \{w\}'\circ \zeta_\chi : \ES{Z}(G(F),\rho) \rightarrow \ES{Z}(G'(F),\rho'_w).
$$
For $w_0\in W_{\chi_0}$, we have $\rho'_{ww_0}=\rho'_w$, but the composition morphism depends really on $w$, and not only  
on the projection of $w$ on $W/W'$ (remark 1 of \ref{conjugation by W}): we have $\{ww_0\}'\circ \zeta_\chi = \{w_0\}'_w \circ \{w\}' \circ \zeta_\chi$ where $\{w_0\}'_w$ is an automorphism of the ${\Bbb C}$--algebra $\ES{Z}(G'(F),\rho'_w)$ which is not trivial in general, even if $w_0\in W'$.

For $f\in \ES{Z}(G(F),\rho)$, we put
$$
\bs{\zeta}_\chi(f)= \vert W'\vert^{-1} \sum_{w\in W}\zeta_{w,\chi}({^w\!f}).\leqno{(2)}
$$
By definition, the function $\bs{\zeta}_\chi(f)$ belongs to the subspace
$$
\bs{\ES{Z}}(G'(F))_{\chi_0}=\sum_{w\in W/W_{\chi_0}}\ES{Z}(G'(F),\rho'_w)
$$
of $\ES{H}(G'(F),1_{\ES{I}'_+})$. For $w_1,\, w_2\in W$, we have $\mathfrak{s}'_{w_1}= \mathfrak{s}'_{w_2}$ if and only if there exists an element $w'\in W'$ such that
${^{w'}\!(\chi_{w_1,0})}= \chi_{w_2,0}$, i.e. $\chi_{w'w_1,0}= \chi_{w_2,0}$. But the last equality is equivalent to $w_2^{-1}w'w_1\in W_{\chi_0}$. 
Hence we have
$$
\bs{\ES{Z}}(G'(F))_{\chi_0} = \bigoplus_w \ES{Z}(G'(F),\rho'_w)\leqno{(3)}
$$
where $w$ runs over a set of representative elements of $W'\backslash W/ W_{\chi_0}$ in $W$.

\vskip1mm
By definition, the image of the Langlands parameter $\varphi_\chi: W_F \rightarrow \check{T}$ of $\chi$ is contained in the center $Z(\check{G}')$ of 
the dual group $\check{G}'$ (cf. the remark 2 of \ref{some results of Roche}), and from Langlands, it defines a character $\omega'_\chi$ of $G'(F)$. Of course $\omega'_\chi$ extends 
the character $\chi_0$ of the maximal compact subgroup $\ES{T}$ of $T'(F)=T(F)$.

\begin{lem}
For $f\in \ES{Z}(G(F),\rho)$, the function $(\omega'_\chi)^{-1} \bs{\zeta}_\chi(f)$ depends only on the depth zero character $\chi_0$ of $\ES{T}$ (not on the choice of the $W_{\chi_0}$--invariant 
character $\chi$ of $T(F)$ extending $\chi_0$).
\end{lem}

\begin{proof}
Let $\chi^\natural$ be another $W_{\chi_0}$--invariant character of $T(F)$ extending $\chi_0$. Put $\theta =(\chi^\natural)\chi^{-1}$. 
It is an unramified character of $T(F)$ which coincides with the restriction of the character $w'_{\chi^\natural}(\omega'_\chi)^{-1}$ of $G'(F)$ to $T(F)$. 
For $f\in \ES{Z}(G(F),\rho)$ and $w\in W$, from \ref{conjugation by W}.(4), 
$\zeta_{w,\chi^\natural}({^w\!f})$ is the unique element of $\ES{Z}(G'(F),\rho'_w)$ which acts on the left on 
$\iota_{B'}^{G'}(\xi)^{\rho'_w}$ by multipication by the scalar
$$
\lambda'_\xi(\zeta_{w,\chi^\natural}({^w\!f}))= \lambda_{\xi\chi^\natural}(f)= \lambda_{(\xi \theta)\chi}(f)= \lambda'_{\xi \theta}(\zeta_{w,\chi}({^w\!f})),
$$
for all characters $\xi$ of $T'(F)=T(F)$ extending a $W'$--conjugate of $\chi_{w,0}$. Since $\theta = \omega'_{\chi^\natural}(\omega'_\chi)^{-1}\vert_{T(F)}$, we obtain
$$
\zeta_{w,\chi^\natural}({^w\!f}) = \omega'_{\chi^\natural}(\omega'_{\chi})^{-1}\zeta_{w,\chi}({^w\!f}).
$$
Thus we have
$$
(\omega'_{\chi^\natural})^{-1}\bs{\zeta}_{\chi^\natural}(f)= (\omega'_\chi)^{-1}\bs{\zeta}_\chi(f).
$$
This proves the lemma.
\end{proof}

For $f\in \ES{Z}(G(F),\rho)$, we put
$$
\bs{\xi}_{\chi_0}(f) = (\omega'_\chi)^{-1}\bs{\zeta}_\chi(f).
$$
The function $\bs{\xi}_{\chi_0}(f)$ belongs to the subspace $\bigoplus_{w\in W'\backslash W/W_{\chi_0}} \ES{Z}(G'(F), \rho'_{^w\!(\chi_0)})$ of $\ES{H}(G'(F),1_{\ES{I}'_+})$.

\begin{thm}Let $f\in \ES{Z}(G(F),\rho)$. The function $\bs{\xi}_{\chi_0}(f)$ is a transfer of $f$.
\end{thm}

\begin{rem1}
{\rm 
The character $\omega'_\chi$ of $G'(F)$ is constant on each stable conjugacy class of $G'(F)$ \cite[lemma 3.1]{Hal}. Hence the theorem says we have 
$$
{\bf SO}_\delta(\bs{\zeta}_\chi(f)) = \omega'_\chi(\delta)\wt{\bf O}_\delta(f)
$$
for all strongly $G$--regular elements $\delta\in G'(F)$.
}
\end{rem1}

\begin{rem2}
{\rm 
Suppose $G=T$. In that case the theorem is true. Indeed, for $f\in \ES{H}(T(F),\chi_0)=\ES{Z}(T(F),\chi_0)$, the compactly supported function $\zeta_\chi(f)$ on 
$T(F)/\ES{T}$ is given by $\zeta_\chi(f)= \chi f$, and the character $\omega'_\chi$ of $T(F)$ is given by $\omega'_\chi =\chi$. On the other hand the transfer factor $\Delta$ on $T(F)\times T(F)$ is equal to $1$. Thus the equality (1) is true: $\omega'_\chi(t)^{-1}\zeta_\chi(f)(t)= \Delta(t,t) f(t)$ for all $t\in T(F)$.
}
\end{rem2} 

The proof is modeled on the method of Haines \cite{Hai2}. The final part of the proof uses local data, which are obtained via a local--global argument based on a simple trace formula, whose geometric side is the elliptic (strongly regular) part of the trace formula. Beforehand it is necessary to reduce the proof of the theorem to the proof of the following ``elliptic version'' of it\footnote{Recall that a semisimple element $\gamma\in G(F)$ is said to be {\it elliptic} if the maximal central $F$--split torus $A_\gamma$ in 
the connected centralizer $G_\gamma$ coin cide with the maximal central $F$--split torus $A_G=Z(G)^\circ$ in $G$. If $G'$ is elliptic and 
$(\delta,\gamma)\in G'(F)\times G(F)$ is a pair of corresponding semisimple elements, then $\delta$ is elliptic in $G'(F)$ if and only if $\gamma$ is elliptic in $G(F)$. 
}. 

\begin{thm*}Assume that the endoscopic group $G'$ of $G$ is elliptic, and let $f$ be a function in $\ES{Z}(G(F),\rho)$. The functions $f$ and $\bs{\xi}_{\chi_0}(f)$ 
have matching elliptic strongly $G$--regular orbital integrals.
\end{thm*}

For technical reasons, other reductions are needed. In fact, we will prove that we can suppose (in the theorem and also in the theorem*)
that the derived group $G_{\rm der}$ of $G$ is simply connected, and the scheme--theoretic center of $G$ is a (split) torus.

\subsection{Variant with two depth zero characters of $\ES{T}$}\label{variant with two characters}
As we will see in \ref{a result of descent}, the reduction of the theorem to the theorem* of \ref{the theorem} is based on a descent formula (proposition of 
\ref{a result of descent}) which leads us to consider a generalization of the theorem* of \ref{the theorem}. 

So let us fix another depth zero character $\psi_0$ of $\ES{T}$. Let $\rho'_{\psi_0}$ be the inflation of $\psi_0$ to the Iwahori subgroup $\ES{I}'$ of $G'$. Here, $G'$ is always the endoscopic group of $G$ defined by $\chi_0$ (cf. \ref{some results of Roche}). We define as follows an injective morphism of ${\Bbb C}$--algebras (again denoted by $\zeta_\chi$)
$$
\zeta_{\chi}= \zeta_{\chi}^{\ES{I}}: \ES{Z}(G(F),\rho_{\psi_0\chi_0})\rightarrow \ES{Z}(G'(F),\rho'_{\psi_0}).
$$
Let us choose a character $\psi$ of $T(F)$ extending $\psi_0$ (it is not necessary to take it $W_{\psi_0}$--invariant). 
Let $\mathfrak{s}_{\psi_0\chi_0}=[T,\psi\chi]_G$ be the $G(F)$--inertial class of the cuspidal pair $(T,\psi\chi)$ of $G(F)$, and $\mathfrak{s}'_{\psi_0}=[T,\psi]_{G'}$ 
the $G'(F)$--inertial class of the cuspidal pair $(T,\psi)$ of $G'(F)$. Since
$$
W'_{\psi_0} \subset W_{\psi_0}\cap W_{\chi_0} \subset W_{\psi_0\chi_0},
$$
the surjective morphism of algebraic varieties
$$
\mathfrak{X}(\mathfrak{s}'_{\psi_0})\rightarrow \mathfrak{X}(\mathfrak{s}_{\psi_0\chi_0}), (T,\xi)\mapsto (T,\xi\chi),
$$
defines by duality an injective morphism of ${\Bbb C}$--algebras (again denoted by $\eta_\chi$)
$$
\eta_\chi : {\Bbb C}[\mathfrak{X}(\mathfrak{s}_{\psi_0\chi_0})] \rightarrow {\Bbb C}[\mathfrak{X}(\mathfrak{s}'_{\psi_0})],
$$
and $\zeta_\chi$ is the morphism of ${\Bbb C}$--algebras which makes the following diagram
$$
\xymatrix{{\Bbb C}[\mathfrak{X}(\mathfrak{s}_{\psi_0\chi_0})] \ar[d]_{\eta_\chi} \ar[r]^(.45){\beta} & \ES{Z}(G(F), \rho_{\psi_0\chi_0}) \ar[d]^{\zeta_\chi}\\
{\Bbb C}[\mathfrak{X}(\mathfrak{s}'_{\psi_0})] \ar[r]^(.45){\beta'} & \ES{Z}(G'(F),\rho'_{\psi_0})
}\leqno{(1)}
$$
commutative. As for the morphism $\zeta_\chi$ of \ref{Bernstein center}.(3), the morphism $\zeta_\chi$ above can be characterized 
by actions on principal series: for $f\in \ES{Z}(G(F), \rho_{\psi_0\chi_0})$, $\zeta_\chi(f)$ is the unique element of $\ES{Z}(G'(F),\rho'_{\psi_0})$ which acts on the space 
$\iota_{B'}^{G'}(\xi)^{\rho'_{\psi_0}}$ by multiplication by the scalar
$$
\lambda'_\xi(\zeta_\chi(f))= \lambda_{\xi\chi}(f)\leqno{(2)}
$$
for all characters $\xi $ of $T'(F)=T(F)$ extending a $W'$--conjugate of $\psi_0$ (note that since $W'\subset W_{\chi_0}$, $\xi\chi$ extends a $W'$--conjugate 
of $\psi_0\chi_0$).

\begin{rem1}
{\rm If $\psi_0=1$, then $\rho'_{\psi_0}= 1_{\ES{I}'}$ and $\zeta_{\chi}$ coincides with the injective morphism of ${\Bbb C}$--algebras 
$\ES{Z}(G(F),\rho_{\chi_0})\rightarrow \ES{Z}(G'(F),1_{\ES{I}'})$ defined in \ref{Bernstein center}. On the other hand, if $\chi_0=1$, then $G'=G$ and 
$\zeta_\chi$ is the identity of $\ES{Z}(G(F),\rho_{\psi_0})$.
}
\end{rem1}

Let $w$ be an element of $W$. As in \ref{conjugation by W}, we choose a lift $n$ of $w$ in $N\cap \ES{K}$ --- recall that $\ES{K}$ is a hyperspecial maximal compact 
subgroup of $G(F)$ containing $\ES{I}$. We consider the character ${^w\!(\rho_{\psi_0\chi_0})}= \rho_{\psi_0\chi_0}\circ {\rm Int}_{n^{-1}}$ of ${^w\ES{I}}$, and the character 
$\rho'_{^w\!\psi_0}\rho'_{\chi_{w,0}}= \rho'_{({^w\!\psi_0})\chi_{w,0}}$ of $\ES{I}'$. Recall that (by definition) $\chi_{w,0}= ({^w\!\chi_0})\chi_0^{-1}$, thus $\rho'_{({^w\!\psi_0})\chi_{w,0}}$ is the inflation of the character 
$({^w\!\psi_0}) \chi_{w,0}= {^w\!(\psi_0\chi_0)}\chi_0^{-1}$ of $\ES{T}$ to $\ES{I}'$. 
We define as follows an injective morphism of ${\Bbb C}$--algebras
$$
\zeta_{w,\chi}: \ES{Z}(G(F),{^w\!(\rho_{\psi_0\chi_0})}) \rightarrow \ES{Z}(G'(F), \rho'_{({^w\!\psi_0})\chi_{w,0}}).
$$
Let $\mathfrak{s}'_{({^w\!\psi_0})\chi_{w,0}}= [T, ({^w\!\psi})\chi_w]_{G'}$ be the $G'(F)$--inertial class of the cuspidal pair $(T, ({^w\!\psi})\chi_w)$ of $G'(F)$. Since
$$
W'_{({^w\!\psi_0})\chi_{w,0}} = W'_{^w\!(\psi_0\chi_0)}\subset W_{\chi_0},
$$
for a character $\xi$ of $T(F)$ extending $\psi_0$ and for $\eta\in W'_{({^w\!\psi_0})\chi_{w,0}}$, we have
$$
{^\eta(({^w\!\xi})\chi_{w})} = {^{\eta w}(\xi\chi)}\chi^{-1}= {^w({^{w^{-1}\eta w}(\xi \chi)})}\chi^{-1}= 
{^w({^{w^{-1}\eta w}(\xi \chi)}\chi^{-1})}\chi_w
$$
(in this calculation we used ${^\eta\chi}=\chi$). 
Hence the surjective morphism of algebraic varieties
$$
\mathfrak{X}(\mathfrak{s}'_{({^w\!\psi_0}) \chi_{w,0}})\rightarrow \mathfrak{X}(\mathfrak{s}_{\psi_0\chi_0}), (T,({^w\xi}) \chi_w)_{G'}\mapsto (T,\xi\chi)_G
$$
is well defined. It
defines by duality an injective morphism of ${\Bbb C}$--algebras
$$
\eta_{w,\chi}: {\Bbb C}[\mathfrak{X}(\mathfrak{s}_{\psi_0\chi_0})] \rightarrow {\Bbb C}[\mathfrak{X}(\mathfrak{s}'_{({^w\!\psi_0}) \chi_{w,0}})],
$$
and $\zeta_{w,\chi}$ is the morphism of ${\Bbb C}$--algebras which makes the following diagram
$$
\xymatrix{
{\Bbb C}[\mathfrak{X}(\mathfrak{s}_{\psi_0\chi_0})] \ar[d]_{\eta_{w,\chi}} \ar[r]^(.45){\beta} & \ES{Z}(G(F), {^w\!(\rho_{\psi_0\chi_0})}) \ar[d]^{\zeta_{w,\chi}}\\
{\Bbb C}[\mathfrak{X}(\mathfrak{s}'_{({^w\!\psi_0})\chi_{w,0}})] \ar[r]^(.45){\beta'} & \ES{Z}(G'(F),\rho'_{({^w\psi_0})\chi_{w,0}})
}\leqno{(3)}
$$
commutative. As for the morphism $\zeta_{w,\chi}$ of \ref{conjugation by W}.(3), the morphism $\zeta_{w,\chi}$ above can be characterized by actions on principal series: for 
$f\in \ES{Z}(G(F),\rho_{\psi_0\chi_0})$, $\zeta_{w,\chi}(f)$ is the unique element of $\ES{Z}(G'(F),\rho'_{({^w\psi_0})\chi_{w,0}})$ which acts on the left on 
$\iota_{B'}^{G'}(\xi)^{\rho'_{({^w\psi_0})\chi_{w,0}}}$ by multipication by the scalar
$$
\lambda'_\xi(\zeta_{w,\chi}({^w\!f}))= \lambda_{\xi\chi}(f)\leqno{(4)}
$$
for all characters $\xi$ of $T'(F)=T(F)$ extending a $W'$--conjugate of $({^w\psi_0})\chi_{w,0}$ (note that since $W'\subset W_{\chi_0}$, 
$\xi\chi$ extends a $W'$--conjugate of ${^w\!(\psi_0\chi_0)}$, and in particular it extends a $W$--conjugate of $\psi_0\chi_0$).

\begin{rem2}
{\rm 
If $\psi_0=1$, then --- with the notation of \ref{some results of Roche} and \ref{conjugation by W} --- 
$\mathfrak{s}_{\psi_0\chi_0}= \mathfrak{s}$ and $\mathfrak{s}'_{({^w\!\psi_0})\chi_{w,0}}= \mathfrak{s}'_w$, and $\zeta_{w,\chi} $ coincides with 
the injective morphism of ${\Bbb C}$--algebras $ \zeta_{w,\chi}: \ES{Z}(G(F),{^w\!\rho})\rightarrow \ES{Z}(G'(F), \rho'_w)$ defined in \ref{conjugation by W}.(3). On the other 
hand, If $\psi_0\neq 1$, in general the composition morphism
$$
\zeta_{w,\chi}\circ [w]: \ES{Z}(G(F),\rho_{\psi_0\chi_0})\rightarrow \ES{Z}(G'(F),\rho'_{({^w\psi_0})\chi_{w,0}})
$$
does not factorize through the morphism $\zeta_\chi: \ES{Z}(G(F),\rho_{\psi_0\chi_0}) \rightarrow \ES{Z}(G'(F),\rho'_{\psi_0})$ defined above; it does if and only if the inclusion
$$
W'_{({^w\!\psi_0})\chi_{w,0}}= W'_{^w\!(\psi_0\chi_0)}\subset W'_{\psi_0}
$$
is verified. Nevertheless, the equality (4) is true even if $\psi_0\neq 1$. 
}
\end{rem2}

For $f\in \ES{Z}(G(F), \rho_{\psi_0\chi_0})$, we put
$$
\bs{\zeta}_\chi(f) = \vert W'\vert^{-1} \sum_{w\in W} \zeta_{w,\chi}({^w\!f}).\leqno{(5)}
$$
The notations are coherent: if $\psi_0=1$, from the remark 2, the function $\bs{\zeta}_\chi(f)$ defined by (5) coincides with the one defined by \ref{conjugation by W}.(2).

\begin{lem}
For $f\in \ES{Z}(G(F),\rho_{\psi_0\chi_0})$, the function $(\omega'_\chi)^{-1}\bs{\zeta}_\chi(f)$ depends only on the depth zero character $\chi_0$ of $\ES{T}$. 
\end{lem}

\begin{proof}
It is identical to the one of the lemma of \ref{the theorem}.
\end{proof}

For $f\in \ES{Z}(G(F),\rho_{\psi_0\chi_0})$, we put
$$
\bs{\xi}_{\chi_0}(f)= (\omega'_\chi)^{-1}\bs{\zeta}_{\chi}(f).
$$
Note that the function $\bs{\xi}_{\chi_0}(f)$ belongs to the subspace $\sum_{w\in W} \ES{Z}(G'(F), \rho'_{^w\!(\psi_0\chi_0)})$ of $\ES{H}(G'(F),1_{\ES{I}'})$. 
The following theorem, resp. theorem*, is a generalization of the theorem, resp. theorem*, of \ref{the theorem}: we recover the theorems of \ref{the theorem} by taking $\psi_0=1$.

\begin{thm}Let $f\in \ES{Z}(G(F),\rho_{\psi_0\chi_0})$. The function $\bs{\xi}_{\chi_0}(f)$ is a transfer of $f$.
\end{thm}

\begin{thm*}
Assume that the endoscopic group $G'$ of $G$ is elliptic, and let $f$ be a function in $\ES{Z}(G(F),\rho_{\psi_0\chi_0})$. The functions $f$ and $\bs{\xi}_{\chi_0}(f)$ 
have matching elliptic strongly $G$--regular orbital integrals.
\end{thm*}

\begin{rem3}
{\rm 
This generalization with two depth zero characters of $\ES{T}$ is not really interesting in itself but only as a tool to reduce the theorem of \ref{the theorem} to its elliptic version (theorem* of \ref{the theorem}.). 
In fact, we will prove in \ref{reduction to the variant} that the ``variant with two characters'' of the theorem* of \ref{the theorem} --- i.e. the theorem* above --- implies the theorem of \ref{the theorem}. But the local--global argument used to prove the theorem* above also makes possible to obtain a stronger result: assuming the endoscopic group $G'$ is elliptic, for all $f\in \ES{Z}(G(F),\rho_{\psi_0\chi_0})$, the function $\bs{\xi}_{\chi_0}(f)$ is a transfer of $f$. Using a descent formula (lemma 1 of \ref{a result of descent}), we will prove in \ref{reduction to G' elliptic} that the result remains true without assuming $G$ elliptic --- i.e. the theorem above.  
}
\end{rem3}

\subsection{Parabolic descent}\label{parabolic descent}
Let $P$ be a semi--standard parabolic subgroup of $G$ defined over $F$, where {\it semi--standard} means containing $T$. Let $U$ be the unipotent radical of $P$, and $M$ the semi--standard Levi component of $P$. The decomposition $P= M\rtimes U$ is defined over $F$. We denote by $W_M$ the Weyl group $N_M(T)/T$ of $M$. We put $B_M=B\cap M$ and $\ES{I}_M= \ES{I}\cap M$. The group $\ES{I}_M$ is an Iwahori subgroup of $M(F)$ in good position with respect to the Borel pair $(B_M,T)$ of $M$.

Let $\mathfrak{s}_M= [T,\chi]_M$ be the $M(F)$--inertial class of the cuspidal pair $(T,\chi)$ {\it of $M(F)$}. Let $\mathfrak{X}(\mathfrak{s}_M)$ be the complex algebraic variety associated to $\mathfrak{s}_M$. There is a canonical surjective morphism of varieties
$$
\mathfrak{X}(\mathfrak{s}_M)\rightarrow \mathfrak{X}(\mathfrak{s}),\, (T,\xi)_M\mapsto (T,\xi)_G,
$$
which gives an injective morphism of ${\Bbb C}$--algebras
$$
\eta^G_M:{\Bbb C}[\mathfrak{X}(\mathfrak{s})]\rightarrow {\Bbb C}[\mathfrak{X}(\mathfrak{s}_M)].
$$
For $f\in C^\infty_{\rm c}(G)$, we define a function $f^{(P,\ES{I})}\in C^\infty_{\rm c}(M)$ by
$$
f^{(P,\ES{I})}(m)= \delta_{P(F)}^{1/2}(m)\int_{U(F)}f(mu)du
$$
where $du$ is the Haar measure on $U$ which gives the volume $1$ to $U\cap \ES{I}$. The inclusion $\ES{T}\subset \ES{I}_M$ induces an isomorphism $\ES{T}/\ES{T}_+\buildrel\simeq\over{\longrightarrow} \ES{I}_M/\ES{I}_{M,+}$ which allows us to view $\chi_0$ as a character $\rho_M$ of $\ES{I}_M$. For $f\in \ES{H}(G(F),\rho)$, the function $f^{(P,\ES{I})}$ is contained in $\ES{H}(M(F),\rho_M)$. From \cite[prop.~5.4.1]{Hai2}, the map
$$
C^\infty_{\rm c}(G(F))\rightarrow C^\infty_{\rm c}(M(F)),\, f\mapsto f^{(P,\ES{I})}
$$
induces by restriction a morphism of ${\Bbb C}$--algebras $\zeta^G_M:\ES{Z}(G(F),\rho)\rightarrow \ES{Z}(M(F),\rho_M)$ and the following diagram
$$
\xymatrix{
{\Bbb C}[\mathfrak{X}(\mathfrak{s})] \ar[r]^\beta \ar[d]_{\eta^G_M} & \ES{Z}(G(F),\rho) \ar[d]^{\zeta^G_M}\\
{\Bbb C}[\mathfrak{X}(\mathfrak{s}_M)] \ar[r]^(.45){\beta_M} & \ES{Z}(M(F),\rho_M)
}\leqno{(1)}
$$
is commutative. In particular, $\zeta^G_M$ is injective. Moreover (loc.~cit.), if we replace $P$ by another semi--standard parabolic subgroup of $G$, defined over $F$ and having $M$ as Levi component, we obtain the same morphism $\zeta^G_M$. Note that $\zeta^G_M$ depends on $\ES{I}$ (even if $\ES{I}$ does not appear in the notation): $\rho$ is a character of 
$\ES{I}$, and the Haar measure $du$ on $U(F)$ is normalized by $\ES{I}$.

\begin{rem}
{\rm
The commutativity of (1) is obtained by comparing the actions on principal series (cf. \ref{Bernstein center}): for $f\in \ES{Z}(G(F),\rho)$, $\zeta^G_M(f)$ is the unique element of $\ES{Z}(M(F),\rho_M)$ which acts on the left on $i_{B_M}^M(\psi)^{\rho_M}$ by multiplication by the scalar
$$
\lambda_{M,\psi}(\zeta^G_M(f))= \lambda_\psi(f),
$$
for all characters $\psi$ of $T(F)$ extending a $W_M$--conjugate of $\chi_0$.
}
\end{rem}

Now we can replace $G$ by $M$ in the construction of \ref{some results of Roche} and \ref{Bernstein center}. In particular, we obtain a connected reductive $F$--split group 
$M'$ which is an endoscopic group of $M$, a group $\smash{\wt{M}}'= M'\rtimes C_{M'}$, and a support--preserving isomorphism of ${\Bbb C}$--algebras
$$
\Psi^M_{\tilde{\chi}_M}:\ES{H}(M(F),\rho_M)\buildrel\simeq \over{\longrightarrow} \ES{H}(\smash{\wt{M}}'(F),1_{\ES{I}_{M'}})\leqno{(2)}
$$
where $\tilde{\chi}_M$ is the restriction of $\tilde{\chi}$ to the group $N_M(T)(F)_{\chi_0}= M \cap N(F)_{\chi_0}$. Hence we have $\tilde{\chi}_M\vert_{T(F)}= \chi$. The isomorphism (2) gives 
by restriction an isomorphism of ${\Bbb C}$--algebras
$$
\zeta^M_{\chi}: \ES{Z}(M(F),\rho_M) \buildrel\simeq\over{\longrightarrow} \ES{Z}(\smash{\wt{M}}'(F),1_{\ES{I}_{M'}})\;(\subset \ES{Z}(M'(F),1_{\ES{I}_{M'}})).\leqno{(3)}
$$
Put $\mathfrak{s}'_{M'}= [T, 1]_{M'}$. Like in \ref{Bernstein center}, the character $\chi$ of $T(F)$ gives an injective morphism of ${\Bbb C}$--algebras 
$$
\eta^M_{\chi}: {\Bbb C}[\mathfrak{X}(\mathfrak{s}_M)] \rightarrow {\Bbb C}[\mathfrak{X}(\mathfrak{s}'_{M'})],\leqno{(4)}
$$
and the following diagram
$$
\xymatrix{
{\Bbb C}[\mathfrak{X}(\mathfrak{s}_M)] \ar[r]^(.45){\beta_M} \ar[d]_{\eta^M_{\chi}} & \ES{Z}(M(F),\rho_M) \ar[d]^{\zeta^M_{\chi}}\\
{\Bbb C}[\mathfrak{X}(\mathfrak{s}'_{M'})] \ar[r]^(.45){\beta_{M'}} & \ES{Z}(M'(F),1_{\ES{I}_{M'}})
}\leqno{(5)}
$$
is commutative. By construction, the group $M'$ can be identified with a semi--standard $F$--Levi factor in $G'$, i.e. the Levi component of a semi--standard parabolic subgroup of $G'$ defined over $F$, where semi--standard means containing $T'=T$. Hence the following diagram
$$
\xymatrix{
{\Bbb C}[\mathfrak{X}(\mathfrak{s}')] \ar[r]^{\beta'} \ar[d]_{\eta^{G'}_{M'}} & \ES{Z}(G'(F),1_{\ES{I}'}) \ar[d]^{\zeta^{G'}_{M'}}\\
{\Bbb C}[\mathfrak{X}(\mathfrak{s}'_{M'})] \ar[r]^(.45){\beta_{M'}} & \ES{Z}(M'(F),1_{\ES{I}_{M'}})
}\leqno{(6)}
$$
is commutative, where the maps $\eta_{M'}^{G'}$ and $\zeta_{M'}^{G'}$ are defined as for (1), mutatis mutandis. 
From (1), (5) and (6), we deduce that the following diagram
$$
\xymatrix{
{\Bbb C}[\mathfrak{X}(\mathfrak{s})] \ar[rrr]^\beta \ar@/^{.3pc}/[dr]_{\eta^G_M} \ar[ddd]_{\eta_\chi}&&& \ES{Z}(G(F),\rho) \ar@/_{.3pc}/[dl]^{\zeta^G_M} \ar[ddd]^{\zeta_{\chi}} \\
&{\Bbb C}[\mathfrak{X}(\mathfrak{s}_M)] \ar[r]^(.45){\beta_M} \ar[d]_{\eta^M_{\chi}}& \ES{Z}(M(F),\rho_M) \ar[d]^{\zeta^M_{\chi}}&\\
&{\Bbb C}[\mathfrak{X}(\mathfrak{s}'_{M'})] \ar[r]^(.45){\beta_{M'}} & \ES{Z}(M'(F),1_{\ES{I}_{M'}})&\\
{\Bbb C}[\mathfrak{X}(\mathfrak{s}')] \ar@/_{.3pc}/[ru]^{\eta_{M'}^{G'}}\ar[rrr]^{\beta'} &&& \ES{Z}(G'(F), 1_{\ES{I}'})\ar@/^{.3pc}/[lu]_{\zeta_{M'}^{G'}}
}\leqno{(7)}
$$
is commutative. We just have to verify the commutativity of the left and right trapezoids. But the commutativity of the left trapezoid is clear (by definition of the maps), and since all the horizontal arrows are bijective maps, it implies the commutativity of the right one. 

Suppose moreover that the Weyl group $W'$ of $G'$ is contained in the Weyl group $W_M=N_M(T)/T$ of $M$. In that case, we have $M'=G'$ and $\beta_{M'}=\beta'$,  and from (7), the following diagram
$$
\xymatrix{
{\Bbb C}[\mathfrak{X}(\mathfrak{s})] \ar[rrr]^\beta \ar@/^{.3pc}/[dr]_{\eta^G_M} \ar[dd]_{\eta_\chi}&&& \ES{Z}(G(F),\rho) \ar@/_{.3pc}/[dl]^{\zeta^G_M} \ar[dd]^{\zeta_{\chi}} \\
&{\Bbb C}[\mathfrak{X}(\mathfrak{s}_M)] \ar[r]^(.45){\beta_M} \ar@/^{.3pc}/[dl]_{\eta^M_{\chi}}& \ES{Z}(M(F),\rho_M) \ar@/_{.3pc}/[dr]^{\zeta^M_{\chi}}&\\
{\Bbb C}[\mathfrak{X}(\mathfrak{s}')] \ar[rrr]^{\beta'} &&& \ES{Z}(G'(F), 1_{\ES{I}'})
}\leqno{(8)}
$$
is commutative.

\subsection{Compatibility with conjugation by $w\in W$.}\label{compatibility by W}
Let $w$ be an element of $W$. Let $M$ be a semi--standard $F$--Levi factor in $G$, i.e. the semi--standard Levi component of a semi--standard parabolic subgroup $P$ of $G$ defined over $F$. Let ${^w\mathfrak{s}_M}= [T,{^w\chi}]_M$ be the $M(F)$--inertial class of the cuspidal pair $(T,{^w\chi})$ of $M(F)$. There is an isomorphism of algebraic varieties
$$
\mathfrak{X}({^w\mathfrak{s}_M}) \buildrel\simeq\over{\longrightarrow} \mathfrak{X}(\mathfrak{s}_M), \, (T, \psi)_M \mapsto (T, {^{w^{-1}}\!\psi})_M,
$$
which gives by duality an isomorphism of ${\Bbb C}$--algebras
$$
{^w\!\mu_M}:{\Bbb C}[\mathfrak{X}(\mathfrak{s}_M)] \buildrel \simeq\over{\longrightarrow} {\Bbb C}[\mathfrak{X}({^w\mathfrak{s}_M})].
$$
By construction, the following diagram
$$
\xymatrix{
{\Bbb C}[\mathfrak{X}(\mathfrak{s})] \ar@{=}[r] \ar[d]_{\eta^G_M} & {\Bbb C}[\mathfrak{X}(\mathfrak{s})] \ar[d]^{\eta^G_M}\\
{\Bbb C}[\mathfrak{X}(\mathfrak{s}_M)] \ar[r]^(.45){^w\!\mu_{M}}& {\Bbb C}[\mathfrak{X}({^w\mathfrak{s}_M})]}\leqno{(1)}
$$
is commutative.
The character ${^w\chi}$ of $T(F)$ can be identified as in \ref{some results of Roche} with a character $[{^w\!\rho}]_M$ of the Iwahori subgroup $[^w\ES{I}]_M = {^w\ES{I}}\cap M$ of $M(F)$. It is given by $[{^w\!\rho}]_M = {^w\!\rho}\vert_{[{^w\ES{I}}]_M}$. Let
$$
[w]_M: \ES{Z}(M(F), \rho_M)\buildrel \simeq\over{\longrightarrow} \ES{Z}(M(F), [{^w\!\rho}]_M)
$$
be the isomorphism of ${\Bbb C}$--algebras which makes the following diagram
$$
\xymatrix{{\Bbb C}[\mathfrak{X}(\mathfrak{s}_M)] \ar[r]^(.45){\beta_M} \ar[d]_{^w\!\mu_{M}} & \ES{Z}(M(F),\rho)\ar[d]^{[w]_M} \\
{\Bbb C}[\mathfrak{X}({^w\mathfrak{s}_M})] \ar[r]^(.45){\beta_M} & \ES{Z}(M(F), [{^w\!\rho}]_M)
}\leqno{(2)}
$$
commutative. Let
$$
{\zeta^G_M} :\ES{Z}(G(F), {^w\!\rho}) \rightarrow  \ES{Z}(M(F), [{^w\!\rho}]_M)
$$
be the morphism of ${\Bbb C}$--algebras given by $\zeta^G_M(f)= f^{(P,{^w\ES{I}})}$. 
From (1), (2) and \ref{conjugation by W}.(1), we obtain that the following diagram
$$
\xymatrix{
\ES{Z}(G(F),\rho) \ar[d]_(.45){\zeta^G_M} \ar[r]^{[w]} & \ES{Z}(G(F), {^w\!\rho})\ar[d]^(.45){\zeta^G_M}\\
\ES{Z}(M(F),\rho_M)  \ar[r]^(.45){[w]_M}  & \ES{Z}(M(F), [{^w\!\rho}]_M)
}\leqno{(3)}
$$
is commutative. 

Let us continue with the semi--standard $F$--Levi factor $M$ in $G$. It defines a semi--standard $F$--Levi factor $M'$ in $G'$ (cf. \ref{parabolic descent}), and 
by replacing $(G,G')$ by $(M,M')$ in the construction of the diagram \ref{conjugation by W}.(3), we obtain a similar commutative diagram
$$
\xymatrix{
\ES{Z}(M(F),\rho_M) \ar[r]^(.45){[w]_M} \ar[d]_{\zeta^M_{\chi}} & \ES{Z}(M(F),[{^w\!\rho}]_M)\ar[d]^{\zeta^M_{w,\chi}} \\
\ES{Z}(M'(F),1_{\ES{I}_{M'}}) \ar[r]^(.45){\{w\}_{M'}} & \ES{Z}(M'(F),[\rho'_w]_{M'})
}.\leqno{(4)}
$$
Here, $[\rho'_w]_{M'}$ is the restriction of the character $\rho'_w$ of $\ES{I}'$ to the Iwahori subgroup $\ES{I}_{M'}= \ES{I}'\cap M'$ of 
$M'(F)$, and $\{w\}_{M'}=\{w\}_{M'\!,\chi}$ is defined as follows. Let $\mathfrak{s}'_{M'\!,w}= [T, \chi'_w]_{M'}$ be the $M'(F)$--inertial class of the cuspidal pair $(T,\chi_w)$ of $M'(F)$, and let $\mu_{M'\!,w}=\mu_{M'\!,w,\chi}: {\Bbb C}[\mathfrak{X}(\mathfrak{s}'_{M'})]\rightarrow 
{\Bbb C}[\mathfrak{X}(\mathfrak{s}'_{M'\!,w})]$ be the injective morphism of ${\Bbb C}$--algebras defined as we have defined $\mu'_w=\mu'_{w,\chi}$ (cf. \ref{conjugation by W}) replacing $(G,G')$ by $(M,M')$. Then 
$$
\{w\}_{M'}: \ES{Z}(M'(F),1_{\ES{I}'})\rightarrow \ES{Z}(M'(F),[\rho'_w]_{M'})
$$
is the injective morphism of ${\Bbb C}$--algebras which makes the following diagram
$$
\xymatrix{
{\Bbb C}[\mathfrak{X}(\mathfrak{s}'_{M'})] \ar[r]^(.45){\beta_{M'}} \ar[d]_{\mu_{M'\!,w}} & \ES{Z}(M'(F),1_{\ES{I}_{M'}}) \ar[d]^{\{w\}_{M'}}\\
{\Bbb C}[\mathfrak{X}(\mathfrak{s}'_{M'\!,w})] \ar[r]^(.45){\beta_{M'}} & \ES{Z}(M'(F), [\rho'_w]_{M'})}
$$
commutative. By construction, the following diagram
$$
\mbox{\footnotesize
\xymatrix{
\ES{Z}(G(F),\rho)\ar[ddd]_{\zeta^G_M}\ar[rrr]^{[w]} \ar@/^{.3pc}/[rd]_{\zeta_\chi}&&& \ES{Z}(G(F),{^w\!\rho})\ar@/_{.3pc}/[dl]^{\zeta_{w,\chi}}\ar[ddd]^{\zeta^G_M}\\
&\ES{Z}(G'(F),1_{\ES{I}'}) \ar[d]_(.45){\zeta^{G'}_{M'}} \ar[r]^{\{w\}'} & \ES{Z}(G'(F), \rho'_w)\ar[d]^(.45){\zeta^{G'}_{M'}}&\\
&\ES{Z}(M'(F),1_{\ES{I}_{M'}})  \ar[r]^(.45){\{w\}_{M'}}  & \ES{Z}(M'(F), [\rho'_w]_{M'})&\\
\ES{Z}(M(F),\rho_M) \ar[rrr]^{[w]_M} \ar@/_{.3pc}/[ur]^{\zeta^M_{\chi}}&&& \ES{Z}(M(F),[^w\!\rho]_M)\ar@/^{.3pc}/[ul]_{\zeta^M_{w,\chi}}
}}\leqno{(5)}
$$
is commutative. We have to verify the commutativity of the middle rectangle and of the right trapezoid. The commutativity of the middle rectangle is clear by definition of the maps $\{w\}'$ and $\{w\}_{M'}$, and the commutativity of the right trapezoid is a consequence of 
the commutativity of the other three trapezoids and of the two rectangles. 

\begin{rem}
{\rm  
As in \ref{some results of Roche} (replacing $G$ by $M$), The character ${^w\!\chi_0}$ of $\ES{T}$ defines a $F$--split connected reductive group $M'_{w}$ with maximal $F$--split torus $T'_{w}= T$, a $F$--group $\wt{M}'_{w} = M'_{w} \rtimes C^M_{^{w}\!\chi_0}$, and an Iwahori subgroup $\ES{I}_{M'_{w}}$ of $M'_{w}(F)$. The root datum $(X, \Phi_{M'_{w}},\check{X},\check{\Phi}_{M'_{w}})$ of $M'_{w}$ is given by
$$
\Phi_{M'_{w}}=\Phi_M \cap w(\Phi'), \quad \check{\Phi}_{M'_{w}}= \check{\Phi}_M \cap w(\check{\Phi}'),
$$
where $(X,\Phi_M,\check{X},\check{\Phi}_M)$ is the root datum of $M$. Note that in general, the group $M'_w$ {\it is not} isomorphic to $M'$ (even if 
${^w\ES{I}}\cap M = \ES{I}\cap M$, in which case $[{^w\!\rho}]_M$ is a character of $\ES{I}_M= \ES{I}\cap M$ which coincides with the inflation of the character ${^w\!\chi_0}$ of $\ES{T}$ 
to $\ES{I}_M$ via the natural isomorphism $\ES{T}/\ES{T}_+ \buildrel \simeq \over{\longrightarrow }\ES{I}_M/\ES{I}_{M,+}$).
}
\end{rem}

Suppose moreover that the Weyl group $W'$ of $G'$ is contained in $M$. Then the following diagramm
$$
\mbox{\footnotesize
\xymatrix{
\ES{Z}(G(F),\rho)\ar[dd]_{\zeta^G_M}\ar[rrr]^{[w]} \ar@/^{.3pc}/[rd]_{\zeta_\chi}&&& \ES{Z}(G(F),{^w\!\rho})\ar@/_{.3pc}/[dl]^{\zeta_{w,\chi}}\ar[dd]^{\zeta^G_M}\\
&\ES{Z}(G'(F),1_{\ES{I}'})  \ar[r]^(.45){\{w\}'} & \ES{Z}(G'(F), \rho'_w)&\\
\ES{Z}(M(F),\rho_M) \ar[rrr]^{[w]_M} \ar@/_{.3pc}/[ur]^{\zeta^M_{\chi}}&&& \ES{Z}(M(F),[^w\!\rho]_M)\ar@/^{.3pc}/[ul]_{\zeta^M_{w,\chi}}
}}\leqno{(6)}
$$
is commutative.

\subsection{Compatibility with conjugation by $w\in W'$}\label{compatibility by W'}
Let $w$ be an element of $W'$, and $M$ a semi--standard $F$--Levi factor in $G$. 
Put $[{^w\ES{I}'}]_{M'}= {^w\ES{I}'}\cap M'$. By replacing $(G,G')$ by $(M,M')$ in the previous constructions, we obtain 
the following commutative diagram
$$
\mbox{\footnotesize
\xymatrix{
\ES{Z}(G(F),\rho)\ar[ddd]_{\zeta^G_M}\ar[rrr]^{[w]} \ar@/^{.3pc}/[rd]_{\zeta_\chi}&&& \ES{Z}(G(F),{^w\!\rho})\ar@/_{.3pc}/[dl]^{^w\!\zeta_\chi}\ar[ddd]^{\zeta^G_M}\\
&\ES{Z}(G'(F),1_{\ES{I}'}) \ar[d]_(.45){\zeta^{G'}_{M'}} \ar[r]^{[w]'} & \ES{Z}(G'(F), 1_{^w{\ES{I}'}})\ar[d]^(.45){\zeta^{G'}_{M'}}&\\
&\ES{Z}(M'(F),1_{\ES{I}_{M'}})  \ar[r]^(.45){[w]_{M'}}  & \ES{Z}(M'(F), [{^w\ES{I}'}]_{M'})&\\
\ES{Z}(M(F),\rho_M) \ar[rrr]^{[w]_M} \ar@/_{.3pc}/[ur]^{\zeta^M_{\chi}}&&& \ES{Z}(M(F),[^w\!\rho]_M)\ar@/^{.3pc}/[ul]_{^w\!\zeta^M_{\chi}}
}}.\leqno{(1)}
$$
The morphism ${^w{\zeta_\chi^M}}$ is defined as in \ref{conjugation by W'} and it makes the lower 
trapezoid commutative (cf. \ref{conjugation by W'}.(2)). 
The morphism $[w]_{M'}$ is defined as in \ref{compatibility by W} and 
it makes the middle rectangle commutative (cf. \ref{compatibility by W}.(3)). 
We just have to verify the commutativity of the right trapezoid, but it is a consequence of 
the commutativity of the other three trapezoids and of the two rectangles. 

Suppose moreover that the Weyl group $W'$ of $G'$ is contained in $M$ (we always assume that $w\in W'$). Then the following diagramm
$$
\mbox{\footnotesize
\xymatrix{
\ES{Z}(G(F),\rho)\ar[dd]_{\zeta^G_M}\ar[rrr]^{[w]} \ar@/^{.3pc}/[rd]_{\zeta_\chi}&&& \ES{Z}(G(F),{^w\!\rho})\ar@/_{.3pc}/[dl]^{^w\!\zeta_\chi}\ar[dd]^{\zeta^G_M}\\
&\ES{Z}(G'(F),1_{\ES{I}'})  \ar[r]^{[w]'} & \ES{Z}(G'(F), 1_{^w{\ES{I}'}})&\\
\ES{Z}(M(F),\rho_M) \ar[rrr]^{[w]_M} \ar@/_{.3pc}/[ur]^{\zeta^M_{\chi}}&&& \ES{Z}(M(F),[^w\!\rho]_M)\ar@/^{.3pc}/[ul]_{^w\!\zeta^M_{\chi}}
}}\leqno{(2)}
$$
is commutative.

\subsection{A result of descent}\label{a result of descent}
Let $M$ be a semi--standard $F$--Levi factor in $G$, and let $P$ be the {\it standard} parabolic subgroup of $G$ having $M$ as Levi component, where standard means containing $B$. The group $P$ is defined over $F$. Let $\{\dot{w}: w\in W_M\backslash W\}$ be the set of Kostant representatives for $W_M\backslash W$ in $W$, i.e. the set of minimal coset representatives for the elements of $W_M\backslash W$ with respect to the Bruhat order on $W$ (and the notion of length) defined by $B$. For each $w\in W_M\backslash W$, we fix a representative element $n_{\dot{w}}$ of $\dot{w}$ in $N\cap \ES{K}$. Then we have the refined Iwasawa decomposition (cf. \cite[5.2.1]{Hai2})
$$
G(F)= \coprod_{w\in W_M\backslash W} P(F)n_{\dot{w}}\ES{I}. \leqno{(1)}
$$

\begin{rem1}
{\rm 
For $w\in W_M\backslash W$, the Kostant representative $\dot{w}$ has a nice property \cite[4.5.2.(a)]{Hai1}:
$$
{^{\dot{w}}\ES{I}}\cap M = \ES{I}\cap M \;(= \ES{I}_M).
$$
We also have (by the same same proof of loc.~cit.)
$$
{^{\dot{w}}\!B}\cap M = B\cap M \;(= B_M).
$$
In particular, we have $\zeta_{^{\dot{w}}\chi}^{M,[{^{\dot{w}}\ES{I}}]_M}= \zeta_{^{\dot{w}}\chi}^{M,\ES{I}_M}=
\zeta_{^{\dot{w}}\chi}^M$.
}
\end{rem1}

In the same way, the group $M'$ is the semi--standard Levi component of a {\it standard} parabolic subgroup $P'$ of $G'$ defined over $F$, where standard means containing $B'$. Let 
$\{\ddot{w}: w \in W_{M'}\backslash W'\}$ be the set of Kostant representatives for $W_{M'}\backslash W'$ in $W'$, i.e. the set of minimal coset representatives for the elements of $W_{M'}\backslash W'$ with respect to the Bruhat order on $W'$ defined by $B'$. For each $w\in W_{M'}\backslash W'$, we fix a representative element $n'_{\ddot{w}}$ of $\ddot{w}$ in $N'\cap \ES{K}'$. Then we have the refined Iwasawa decomposition
$$
G'(F)= \coprod_{w\in W_{M'}\backslash W'} P'(F)n'_{\ddot{w}}\ES{I}'. \leqno{(2)}
$$

Since $W_{M'}= W'\cap W_M$, the inclusion $W'\subset W$ induces an injective map $W_{M'}\backslash W' \rightarrow W_M\backslash W$ which allows us to identifie $W_{M'}\backslash W'$ with a subset of $W_M\backslash W$. For $w\in W_{M'}\backslash W'$, $\ddot{w}$ is the element of $W_{M'}w$ of minimal length, and $\dot{w}$ is the element of $W_Mw$ of minimal length. 

For a strongly $M$--regular element $\delta\in M'(F)$, let ${\bf SO}^{M'}_\delta$, resp. $\wt{\bf O}^M_\delta$, be the distribution on $M'(F)$, resp. on $M(F)$, defined as in \ref{the theorem} replacing $(G,G')$ by $(M,M')$: for a function $\phi'\in C^\infty_{\rm c}(M'(F))$, we have
$$
{\bf SO}_\delta^{M'}(\phi')= \sum_{\delta'}{\bf O}_{\delta'}^{M'}(\phi')\leqno{(3)}
$$
where $\delta'$ runs over the elements of $M'(F)$ which are stably conjugate to $\delta$ modulo conjugation by $M'(F)$ --- since the connected centralizer $M'_{\delta'}$ is a torus, we have $e(M'_{\delta'})=1$ ---, and for a function $\phi\in C^\infty_{\rm c}(M(F))$, we have
$$
\wt{\bf O}_\delta^M(\phi)= \sum_\gamma \Delta_M(\delta,\gamma){\bf O}_\gamma^M(\phi)\leqno{(4)}
$$
where $\gamma$ runs over the strongly regular elements in $M(F)$ which correspond to $\delta$ modulo conjugation by $M(F)$. In (3), ${\bf O}_{\delta'}^{M'}$ is the orbital integral defined by $\delta'$ on $M'(F)$. In (4), $\Delta_M(\delta,\gamma)$ is the canonical transfer factor on $M'(F)\times M(F)$ defined by the hyperspecial maximal compact subgroups $\ES{K}_M=\ES{K}\cap M$ of $M(F)$, and ${\bf O}_\gamma^M$ is the orbital integral defined by $\gamma$ on $M(F)$. The invariant measures defining the distributions ${\bf SO}_\delta^{M'}$ on $M'(F)$ and ${\bf O}_\gamma^M$ on $M(F)$ are choosen as in \ref{the theorem} (replacing $(G,G')$ by $(M,M')$). 

\begin{lem1}
Let $f\in \ES{Z}(G,\rho)$, and let $\delta\in M'(F)$ be a strongly $G$--regular element. We have the descent formula
$$
\wt{\bf O}_\delta(f) = \vert D_{M'\backslash G'}(\delta)\vert_F^{-1/2}\sum_w \wt{\bf O}_\delta^M(\zeta^G_M({^{\dot{w}}\!f}))
$$
where $w$ runs over the elements of $W_M\backslash W$.
\end{lem1}

\begin{proof}Note that since $\delta$ is strongly $G$--regular, it is a fortiori strongly $M$--regular. Recall also that we have (by definition) 
$\zeta^G_M({^{\dot{w}}\!f})= [{^{\dot{w}\!}f}]^{(P,{^{\dot{w}}\ES{I}})}$. 

If $\gamma\in G(F)$ corresponds to $\delta$, then the $G(F)$--orbit of $\gamma$ meets $M(F)$, so we can suppose that $\gamma$ belongs to $M(F)$. In that case, since $M_\gamma = G_\gamma$ and since the map
$$
{\bf H}^1(F,M) \rightarrow {\bf H}^1(F,G)
$$
is injective  (cf. \cite[4.1]{Hai1}), we obtain that the natural map
$$
\ker [{\bf H}^1(F,M_\gamma)\rightarrow {\bf H}^1(F,M)] \rightarrow \ker [{\bf H}^1(F,G_\gamma)\rightarrow {\bf H}^1(F,G)]
$$
is bijective. Here we put ${\bf H}^1(F,?)= {\bf H}^1(\Gamma_F,?(F^{\rm sep}))$. In other words, there is a natural bijection from the set of $M(F)$--conjugacy classes in the stable conjugacy class of $\gamma$ in $M(F)$ to 
the set of $G(F)$--conjugacy classes in the stable conjugacy class of $\gamma$ in $G(F)$. So we have
$$
\wt{\bf O}_\delta(f)= \sum_\gamma \Delta(\delta,\gamma) {\bf O}_\gamma(f)
$$
where $\gamma$ runs over the strongly regular elements in $M(F)$ which correspond to $\delta$ modulo conjugation by $M(F)$.

On the other hand, for $\gamma\in M(F)$ corresponding to $\delta$, 
we have the descent formula \cite[6.2.3]{Hai2} 
$$
\vert D_{M\backslash G}(\gamma)\vert_F^{1/2}{\bf O}_\gamma(f)= \sum_w {\bf O}_\gamma^M(\zeta^G_M({^{\dot{w}\!}f}))\leqno{(5)}
$$
where $w$ runs over the elements of $W_M\backslash W$. Thus we have
$$
\wt{\bf O}_\delta(f)= \sum_w \sum_\gamma \vert D_{M\backslash G}(\gamma)\vert_F^{-1/2}\Delta(\delta,\gamma) {\bf O}_\gamma^M(\zeta^G_M({^{\dot{w}\!}f}))\leqno{(6)}
$$
where $w$ runs over the elements of $W_M\backslash W$ and $\gamma$ runs over the elements of $M(F)$ corresponding to $\delta$ modulo conjugation by $M(F)$. 
For such a $\gamma$, we have 
(by definition of the canonical transfer factors)
$$
 \Delta_{\rm IV}(\delta,\gamma)^{-1}\Delta(\delta,\gamma)=\Delta_{M,{\rm IV}}(\delta,\gamma)^{-1}\Delta_M(\delta, \gamma),
$$
i.e.
$$
\vert D_{M\backslash G}(\gamma)\vert_F^{-1/2}\Delta(\delta,\gamma)=\vert D_{M'\backslash G'}(\delta)\vert_F^{-1/2}\Delta_M(\delta,\gamma).\leqno{(7)}
$$
From (6) and (7), we obtain the lemma 1. 
\end{proof}

\begin{lem2}
Let $f'\in \ES{Z}(G'(F),\rho'_{w_1})$ for an element $w_1\in W$, and let $\delta\in M'(F)$ be a strongly $M$--regular element. 
We have the descent formula
$$
{\bf SO}_\delta(f') = \vert D_{M'\backslash G'}(\delta)\vert_F^{-1/2}\sum_w {\bf SO}_\delta^{M'} (\zeta^{G'}_{M'}({^{\ddot{w}\!}f'}))
$$
where $w$ runs over the element of $W_{M'}\backslash W'$.
\end{lem2}

\begin{proof}From the proof of lemma 1, 
there is a natural bijection from the set of $M'(F)$--conjugacy classes in the stable conjugacy class of $\delta$ in $M'(F)$ to 
the set of $G'(F)$--conjugacy classes in the stable conjugacy class of $\delta$ in $G'(F)$. So we have
$$
{\bf SO}_\delta(f')= \sum_{\delta'} {\bf O}_{\delta'}(f')
$$
where $\delta'$ runs over the elements in $M'(F)$ which are stably conjugate to $\delta$ modulo conjugation by $M'(F)$. For such a $\delta'$, we have the descent 
formula \cite[6.2.3]{Hai2} 
$$
\vert D_{M'\backslash G'}(\delta')\vert_F^{1/2}{\bf O}_{\delta'}(f') = \sum_w {\bf O}_{\delta'}^{M'} (\zeta^{G'}_{M'}({^{\ddot{w}\!}f'}))\leqno{(8)}
$$
where $w$ runs over the element of $W_{M'}\backslash W'$. Since
$$
\vert D_{M'\backslash G'}(\delta')\vert_F= \vert D_{M'\backslash G'}(\delta)\vert_F
$$
for all $\delta'\in M'(F)$ stably conjugate to $\delta$, we obtain the lemma 2.
\end{proof}

Using these two lemmas, we can deduce for the strongly $M$--regular elements of $M'(F)$, the matching relation we are interested with (for the pair $(G,G')$) 
from similar matching relations for the pair $(M,M')$. 

Recall that for each $w_1\in W$, we have defined in \ref{compatibility by W} an injective morphism of ${\Bbb C}$--algebras
$$
\{w_1\}_{M'}=\{w_1\}_{M'\!,\chi}:\ES{Z}(M'(F), 1_{\ES{I}_{M'}})\rightarrow \ES{Z}(M'(F), [\rho'_{w_1}]_{M'})
$$
where $[\rho'_{w_1}]_{M'}$ is the inflation of the character $\chi_{w_1,0}=({^{w_1}\!\chi_0})\chi_0^{-1}$ of $\ES{T}$ to the Iwahori subgroup 
$\ES{I}_{M'}$ of $M'(F)$. 

We defined in \ref{Bernstein center} a character $\omega'_\chi$ of $G'(F)$, associated to the $W_{\chi_0}$--invariant character $\chi$ of $T(F)$ extending $\chi_0$.  
Since $\chi$ is a fortiori $(W_M)_{\chi_0}$--invariant, it defines also a character $\omega_\chi^{M'}$ of $M'(F)$, which is the restriction of $\omega'_\chi$ to $M'(F)$. 

\begin{prop}
Let $f\in \ES{Z}(G,\rho)$ and $f'=  \zeta_{\chi}(f)\in \ES{Z}(G'\!,1_{\ES{I}'})$, and let $\delta\in M'(F)$ be a strongly $G$--regular element. Suppose 
we have
$$
\vert W_{M'}\vert^{-1}\sum_{w}\sum_{w_M\in W_M}
{\bf SO}_\delta^{M'}(\{w_M\dot{w}\}_{M'}\circ \zeta^{G'}_{M'}(f'))=
\omega'_\chi(\delta)\sum_{w}\wt{\bf O}^M_\delta(\zeta_M^G({^{\dot{w}}\!f}))
$$
where $w$ runs over the elements of $W_M\backslash W$. 
Then we have
$$
\vert W'\vert^{-1}\sum_{w_1\in W}{\bf SO}_\delta(\{w_1\}'(f'))= \omega'_\chi(\delta)\wt{\bf O}_\delta(f).
$$
\end{prop}

\begin{proof}For $w_1\in W$, we put $f'_{w_1}= \{w_1\}'(f')\in \ES{Z}(G'(F), \rho'_{w_1})$. From the lemmas 1 and 2, we have to prove the equality
$$
\vert W'\vert^{-1}\sum_{w_1\in W}\sum_{w'}
{\bf SO}_\delta^{M'} (\zeta^{G'}_{M'}({^{\ddot{w}'}\!(f'_{w_1})}))=\omega'_\chi(\delta)\sum_w\wt{\bf O}_\delta^M(\zeta^G_M({^{\dot{w}}\!f}))\leqno{(9)}
$$
where $w'$ runs over the elements of $W_{M'}\backslash W'$ and $w$ runs over the elements of $W_M\backslash W$. For $w'\in W_{M'}\backslash W'$, the function ${^{\ddot{w}'}\!(f'_{w_1})}$ belongs to $\ES{Z}(G'(F),{^{\ddot{w}'}\!(\rho'_{w_1})})$ where 
${^{\ddot{w}'}\!(\rho'_{w_1})}$ is the character $ \rho'_{w_1}\circ {\rm Int}_{\ddot{w}'^{-1}}$ of ${^{\ddot{w}'}\!\ES{I}'}$, 
and we have (by definition)
$$
\zeta^{G'}_{M'}({^{\ddot{w}'\!}(f'_{w_1})})= [{^{\ddot{w}'\!}(f'_{w_1})}]^{(P',{^{\ddot{w}}\ES{I}'})}.
$$ 
Since ${^{\ddot{w}'}\!\chi_0}=\chi_0$, ${^{\ddot{w}'}\!(\rho'_{w_1})}$ is the inflation of the character ${^{\ddot{w}'}\!(\chi_{w_1,0})}= \chi_{\ddot{w}w_1,0}$ of $\ES{T}$ to the Iwahori subgroup 
${^{\ddot{w}}\ES{I}'}$ of $G'(F)$, and since ${^{\ddot{w}}\ES{I}'}\cap M' = \ES{I}'\cap M' = \ES{I}_{M'}$, $[{^{\ddot{w}'}\!(\rho'_{w_1})}]_{M'}$ is the inflation of 
the character $\chi_{\ddot{w}w_1,0}$ of $\ES{T}$ to $\ES{I}_{M'}$. In other words, we have
$$
[{^{\ddot{w}'}\!(\rho'_{w_1})}]_{M'} = [\rho'_{\ddot{w}'w_1}]_{M'},\leqno{(10)}
$$
and the function $\zeta^{G'}_{M'}({^{\ddot{w}'\!}(f'_{w_1})})$ belongs to $\ES{Z}(M'(F),[\rho'_{\ddot{w}'w_1}]_{M'})$. 
We also have the injective morphism of 
${\Bbb C}$--algebras
$$
\{\ddot{w}'w_1\}_{M'}=\{\ddot{w}'w_1\}_{M'\!,\chi}: \ES{Z}(M'(F),1_{\ES{I}_{M'}})\rightarrow \ES{Z}(M'(F), [\rho'_{\ddot{w}'w_1}]_{M'}),
$$
and the function $\{\ddot{w}'w_1\}_{M'}\circ \zeta^{G'}_{M'}(f')$ belongs to $\ES{Z}(M'(F), [\rho'_{\ddot{w}'w_1}]_{M'})$. Recall that we have defined in \ref{conjugation by W} an 
isomorphism of ${\Bbb C}$--algebras
$$
[\ddot{w}]'_{w_1}: \ES{Z}(G'(F), \rho'_{w_1})\rightarrow \ES{Z}(G'(F),\rho'_{\ddot{w}w_1})
$$
such that
$$
\{\ddot{w}w_1\}'= [\ddot{w}]'_{w_1}\circ \{w_1\}'.
$$
From (10), we have
$$
\zeta^{G'}_{M'} \circ \{\ddot{w}w_1\}'=\zeta^{G'}_{M'} \circ [\ddot{w}]'_{w_1} \circ \{w_1\}'= 
\zeta^{G'}_{M'}\circ [\ddot{w}]'\circ \{w_1\}'.
$$
Since
$$
\zeta^{G'}_{M'} \circ \{\ddot{w}w_1\}'=\{\ddot{w}'w_1\}_{M'}\circ \zeta^{G'}_{M'},
$$
we obtain the equality
$$
\zeta^{G'}_{M'}({^{\ddot{w}'\!}(f'_{w_1})})=\{\ddot{w}'w_1\}_{M'}\circ \zeta^{G'}_{M'}(f').\leqno{(11)}
$$

On the other hand, $\ddot{w}'w_1$ can be written as $\ddot{w}'w_1= w_M \dot{w}$ for a unique $w\in W_M\backslash W$ and a unique $w_M\in W_M$. 
The map
$$
W \times (W_{M'}\backslash W') \rightarrow (W_M\backslash W) \times W_M, \, (w_1,w')\mapsto (w,w_M)
$$
defined by $\ddot{w}'w_1= w_M\dot{w}$ is surjective, with fibers of cardinality $\vert W_{M'}\backslash W'\vert$. 
Thus from (11), the left hand side of (9) is equal to
$$
\vert W'\vert^{-1} \vert W_{M'}\backslash  W'\vert \sum_{w}\sum_{w_M} 
{\bf SO}_\delta^{M'}(\{w_M\dot{w}\}_{M'}\circ \zeta^{G'}_{M'}(f'))\leqno{(12)}
$$
where $w$ runs over the elements of $W_M\backslash W$ and $w_M$ runs over the elements of $W_M$. 
The hypothesis of the proposition implies that (12) is equal to the right hand side of (9).
\end{proof}

\subsection{About the hypothesis in the proposition of \ref{a result of descent}}\label{about the hypothesis}

Let us continue with the hypothesis and notations of \ref{a result of descent}. Let $f\in \ES{Z}(G(F),\rho)$ and $f'=\zeta_\chi(f)\in \ES{Z}(G'(F),1_{\ES{I}'})$, and 
let $w$ be an element of $W_M\backslash W$. The function $\phi_{\dot{w}} = \zeta^G_M({^{\dot{w}}\!f})$ belongs to $\ES{Z}(M(F), [{^{\dot{w}}\!\rho}]_M)$ where 
$[{^{\dot{w}}\!\rho}]_M$ is the inflation of the character ${^{\dot{w}}\!\chi_0}$ of $\ES{T}$ to the Iwahori subgroup $\ES{I}_M$ of $M(F)$. Replacing 
$(G,\chi_0)$ by $(M,{^{\dot{w}}\!\chi_0})$ in the definition of $G'$, as in \ref{some results of Roche} --- cf. also 
the remark of \ref{compatibility by W} --- the character ${^{\dot{w}}\!\chi_0}$ of $\ES{T}$ defines a $F$--split connected reductive 
group $M'_{\dot{w}}$ with maximal $F$--split torus $T'_{\dot{w}}= T$, 
an Iwahori subgroup $\ES{I}_{M'_{\dot{w}}}$ of $M'_{\dot{w}}(F)$, and a $F$--group $\wt{M}'_{\dot{w}} = M'_{\dot{w}} \rtimes C^M_{^{\dot{w}}\!\chi_0}$. The group 
$M'_{\dot{w}}$ is an endoscopic group of $M$. We have an injective morphism of ${\Bbb C}$--algebras
$$
\zeta^{M}_{^{\dot{w}}\!\chi}=\zeta^{M,\ES{I}_M}_{^{\dot{w}}\!\chi}: \ES{Z}(M(F),[{^{\dot{w}}\!\rho}]_M) \rightarrow \ES{Z}(M'_{\dot{w}}(F),1_{\ES{I}_{M'_{\dot{w}}}}),
$$
and for each $w_M\in W_M$, we have an injective morphism of ${\Bbb C}$--algebras
$$
\{w_M\}_{M'_w,{^{\dot{w}}\!\chi}}: \ES{Z}(M'_{\dot{w}}(F),1_{\ES{I}_{M'_{\dot{w}}}})\rightarrow 
\ES{Z}(M'_{\dot{w}}(F),({^{\dot{w}}\!\rho})'_{w_M})
$$
where $({^{\dot{w}}\!\rho})'_{w_M}$ is the inflation of the character ${^{w_M}\!({^{\dot{w}}\!\chi_0})} ({^{\dot{w}}\!\chi_0})^{-1}$ of $\ES{T}$ to ${\ES{I}_{M'_{\dot{w}}}}$. 
The theorem of \ref{the theorem} for $M$ and the character ${^{\dot{w}}\!\chi_0}$ of 
$\ES{T}$ says that the function 
$$
\phi'_{\dot{w}}= \vert W_{M'_{\dot{w}}}\vert^{-1} \sum_{w_M\in W_M} \{w_M\}_{M'_{\dot{w}},{^{\dot{w}}\!\chi}} \circ \zeta_{^{\dot{w}}\!\chi}^M(\phi_{\dot{w}})
$$
is a transfer of $\phi_{\dot{w}}$. In other words, for all $M(F)$--regular elements $\delta\in M'_{\dot{w}}(F)$ we have
$$
\omega_{^{\dot{w}}\!\chi}^{M'_{\dot{w}}}(\delta)^{-1}{\bf SO}_\delta^{M'_{\dot{w}}}(\phi'_{\dot{w}})= \sum_\gamma \Delta_{M,{^{\dot{w}}\!\chi_0}}(\delta,\gamma) {\bf O}_\gamma^M(\phi_{\dot{w}})\leqno{(1)}
$$
where $\gamma$ runs over the semisimple (regular) elements in $M(F)$ corresponding to $\delta$, modulo conjugation by $M(F)$. Here $\Delta_{M,{^{\dot{w}}\!\chi_0}}(\delta,\gamma)$ is the canonical transfer factor defined by the hyperspecial compact subgroup $\ES{K}_M= \ES{K}\cap M$ of $M(F)$, and $\omega_{^{\dot{w}}\!\chi}^{M'_{\dot{w}}}$ is the character 
of $M'_{\dot{w}}(F)$ associated to ${^{\dot{w}}\!\chi}$ as in \ref{the theorem}.

Instead of working with the endoscopic group $M'_{\dot{w}}$ of $M$, we can replace $M$ by 
$M_{\dot{w}}= {\rm Int}_{\dot{w}^{-1}}(M)$ and consider the endoscopic group $(M_{\dot{w}})'$ of $M_{\dot{w}}$ defined by $\chi_0$ replacing $M$ by $M_{\dot{w}}$. 
Note that $M_{\dot{w}}$ is the semistandard $F$--Levi subgroup of $G$ associated to the root datum $(X,\dot{w}^{-1}(\Phi_M),\check{X},\dot{w}^{-1}(\check{\Phi}_M))$. Since
$$
{^{\dot{w}^{-1}}\!(\ES{I}_M)}= {^{\dot{w}^{-1}}({^{\dot{w}}\ES{I}} \cap M)}= \ES{I} \cap {M_{\dot{w}}} = \ES{I}_{M_{\dot{w}}},
$$
the character $[{^{\dot{w}}\!\rho}]_M\circ {\rm Int}_{\dot{w}}$ of $\ES{I}_{M_{\dot{w}}}$ coincides with $\rho_{M_{\dot{w}}}$ --- 
the inflation of the character $\chi_0$ of $\ES{T}$ to the Iwahori subgroup $\ES{I}_{M_{\dot{w}}}$ of $M_{\dot{w}}(F)$. 
The map
$$
C^\infty_{\rm c}(M(F))\rightarrow C^\infty_{\rm c}(M_{\dot{w}}(F)),\, \phi \mapsto {^{\dot{w}^{-1}}\!\phi}= \phi \circ {\rm Int}_{n_{\dot{w}}}
$$
induces by restriction an isomorphism of ${\Bbb C}$--algebras
$$
[\dot{w}^{-1}]^M_{M_{\dot{w}}}: \ES{Z}(M(F), [{^{\dot{w}}\!\rho}]_M)\rightarrow \ES{Z}(M_{\dot{w}}(F),\rho_{M_{\dot{w}}}).
$$
Replacing $M$ by $M_{\dot{w}}$ in the definition of $M'$, the character $\chi_0$ of $\ES{T}$ defines 
a $F$--split connected reductive group $(M_{\dot{w}})'$ with maximal 
$F$--split torus $(T_{\dot{w}})'= T$, an Iwahori subgroup $\ES{I}_{(M_{\dot{w}})'}$ of $(M_{\dot{w}})'(F)$, 
and a $F$--group $(M_{\dot{w}})'\rtimes C^{M_{\dot{w}}}_{\chi_0}$. 
The root datum $(X, \Phi_{(M_{\dot{w}})'},\check{X},\check{\Phi}_{(M_{\dot{w}})'})$ of $(M_{\dot{w}})'$ is given by 
$$
\Phi_{(M_{\dot{w}})'}= \dot{w}^{-1}(\Phi_M)\cap \Phi',\quad \check{\Phi}_{(M_{\dot{w}})'}= \dot{w}^{-1}(\check{\Phi}_M)\cap \check{\Phi}'.
$$
As in the remark 2 of \ref{conjugation by W}, $\dot{w}$ induces by transport of structure an $F$--isomorphism $(M_{\dot{w}})'\rtimes C^{M_{\dot{w}}}_{\chi_0}
\buildrel\simeq\over{\longrightarrow} \wt{M}'_{\dot{w}}$ which allows us to identify $(M_{\dot{w}})'(F)$ with $M'_{\dot{w}}(F)$ and $\ES{I}_{(M_{\dot{w}})'}$ 
with $\ES{I}_{M'_{\dot{w}}}$. For a semisimple regular elements $\gamma\in M(F)$, writing $\gamma_{\dot{w}}= n_{\dot{w}}^{-1}\gamma n_{\dot{w}}$, 
we have (by transport of structure)
$$
{\bf O}^M_\gamma(\phi_{\dot{w}})= {\bf O}^{M_{\dot{w}}}_{\gamma_{\dot{w}}}([\dot{w}^{-1}]^M_{M_{\dot{w}}}(\phi_{\dot{w}})).
$$
Hence for all $M$--regular elements $\delta\in M'_{\dot{w}}(F)= (M_{\dot{w}})'(F)$ and for all semisimple (regular) elements $\gamma\in M(F)$ 
corresponding to $\delta$, the canonical transfer factor $\Delta_{M,{^{\dot{w}}\!\chi_0}}(\delta,\gamma)$ is given by
$$
\Delta_{M,{^{\dot{w}}\!\chi_0}}(\delta,\gamma)= \Delta_{M_{\dot{w}}}(\delta,\gamma_{\dot{w}}).\leqno{(2)}
$$
Here $\Delta_{M_{\dot{w}}}(\delta,\gamma_{\dot{w}})= \Delta_{M_{\dot{w}},\chi_0}(\delta,\gamma_{\dot{w}})$ is the canonical transfer factor defined 
by the hyperspecial compact subgroup $\ES{K}_{M_{\dot{w}}}= \ES{K}\cap M_{\dot{w}}$ of $M_{\dot{w}}(F)$ --- since 
the representative element $n_{\dot{w}}$ of $\dot{w}$ belongs to $\ES{K}\cap N_G(T)$, we have 
$$
\ES{K}_{M_{\dot{w}}}= n_{\dot{w}}^{-1} \ES{K}_M n_{\dot{w}}.
$$
On the other hand, the character $\omega_\chi^{(M_{\dot{w}})'}$ of $(M_{\dot{w}})'(F)$ defined by $\chi$ coincides with the character 
$\omega_{^{\dot{w}}\!\chi}^{M'_{\dot{w}}}$ of $M'_{\dot{w}}(F)$ defined by ${^{\dot{w}}\!\chi}$.

\vskip1mm
Unfortunately, the formula (1) and the calculation of the transfer factor (2) are not sufficient to prove the hypothesis of the proposition of \ref{a result of descent}: 
we need a formula analogous to (1), but in terms of $M'$ instead of $M'_{\dot{w}}$ (recall that the groups $M'$ and $M'_{\dot{w}}$ are in general {\it not} isomorphic). 
This is why we have introduced in \ref{variant with two characters} the ``variant with two characters'' of the theorems of \ref{the theorem}.

\begin{prop}
If the theorem of \ref{variant with two characters} is true for $M$ (and for all pairs $(\psi_0,\chi_0)$ of depth zero characters of $\ES{T}$), then 
the hypothesis of the proposition of \ref{a result of descent} 
is verified.
\end{prop}

\begin{proof}
Let $f\in \ES{Z}(G(F), \rho)$ and $f'=\zeta_\chi(f)\in \ES{Z}(G'(F),1_{\ES{I}'})$. Let us prove that for all $w\in W_M\backslash W$ 
and $w_M\in W_M$, we have
$$
\{w_M\dot{w}\}_{M'\!,\chi} \circ \zeta^{G'}_{M'}(f') = \zeta^M_{w_M, \chi}\circ {^{w_M}\!(\zeta^G_M({^{\dot{w}}\!f}))}\leqno{(3)}
$$
where
$$
\zeta^M_{w_M,\chi}: \ES{Z}(M(F),{^{w_M}\!( [{^{\dot{w}}\!\rho}]_M)}) \rightarrow \ES{Z}(M'(F), [\rho'_{w_M\dot{w},\chi_0}]_{M'})
$$
is the morphism of ${\Bbb C}$--algebras defined in \ref{variant with two characters} taking $\psi_0= \chi_{\dot{w},0}\;(=({^{\dot{w}}\!\chi_0})\chi_0^{-1})$), 
$[{^{\dot{w}}\!\rho}]_M$ is the inflation of the character ${^{\dot{w}}\!\chi_0}= \psi_0\chi_0$ of $\ES{T}$ to $\ES{I}_M$, and $[\rho_{w_M\dot{w},\chi_0}]_{M'}$ is the inflation of the character 
$({^{w_M}\!\psi_0}) \chi_{w_M,0} = \chi_{w_M\dot{w},0}\;(= ({^{w_M\dot{w}}\!\chi_0})\chi_0^{-1})$ of $\ES{T}$ to $\ES{I}_{M'}$. Since
$$
\{w_M\dot{w}\}_{M'\!,\chi} \circ \zeta^{G'}_{M'}(f') = \{w_M\dot{w}\}_{M'} \circ \zeta^M_\chi \circ \zeta^G_M(f)
$$
and
$$
\zeta^M_{w_M, \chi}\circ {^{w_M}\!(\zeta^G_M({^{\dot{w}}\!f}))}= \zeta^M_{w_M, \chi}\circ [w_M\dot{w}]_M \circ \zeta^G_M(f),
$$
we are reduced to prove the equality
$$
\{w_M\dot{w}\}_{M'} \circ \zeta^M_\chi =  \zeta^M_{w_M, \chi}\circ [w_M\dot{w}]_M\leqno{(4)}
$$
as morphisms of ${\Bbb C}$--algebras $\ES{Z}(M(F), \rho_M)\rightarrow \ES{Z}(M'(F), [\rho_{w_M\dot{w},\chi_0}]_{M'}$. But (4) follows quite directly from the definitions: the 
injective composition morphisms of ${\Bbb C}$--algebras of the left hand side and the right hand side of (3) are given by two 
surjective morphisms of algebraic varieties
$$
\mathfrak{X}(\mathfrak{s}_{M'\!,\chi_{w_M\dot{w},0}})\rightarrow \mathfrak{X}(\mathfrak{s}_{M,\chi_0})
$$ 
which (by definition) coincide; where $\mathfrak{s}_{M,\chi_0}= [T,\chi]_M$ is the $M(F)$--inertial class of the cuspidal pair $(T,\chi)$ of $M(F)$ and 
$\mathfrak{s}_{M'\!,\chi_{w_M\dot{w},0}}= [T, \chi_{w_M\dot{w}}]_{M'}$ is the $M'(F)$--inertial class of the cuspidal pair 
$(T, \chi_{w_M\dot{w}})$ of $M'(F)$.
Now if the theorem of \ref{variant with two characters} is true for $M$ (and for all pairs $(\psi_0,\chi_0)$ of depth zero characters of $\ES{T}$), then for all 
$w\in W_M\backslash W$, by (3) the function
$$
(\omega^{M'}_\chi)^{-1}\vert W_{M'}\vert^{-1} \sum_{w_M\in W_M} \{w_M\dot{w}\}_{M'} \circ \zeta^{G'}_{M'}(f')
$$
is a transfer of 
$\zeta^G_M({^{\dot{w}}\!f})$. Taking the sum over all the elements $w\in W_M\backslash W$, we obtain that the hypothesis of the proposition of \ref{a result of descent} 
is true.
\end{proof}

\begin{cor}
Suppose $M'$ is an elliptic endoscopic group of $M$. If the theorem* of \ref{variant with two characters} is true for $M$ (and for all pairs of depth zero characters 
$(\psi_0,\chi_0)$ of $\ES{T}$), then for all strongly $G$--regular elements $\delta\in M'(F)$ which are $M$--elliptic, the hypothesis of the proposition of \ref{a result of descent} 
is verified.
\end{cor}

\subsection{An example}\label{an example}
Let us give the simplest example illustrating the result of descent: the case of a depth zero regular character $\chi_0$ of $\ES{T}$. In that case we have $G'=T'=T$. Let $f\in \ES{Z}(G(F),\rho)$. For $w\in W$ and $t\in T(F)$, we have
$$
\zeta^G_T({^w\!f})(t)= ({^w\!f})^{(B,{^w\ES{I}})}(t)=\delta_B^{1/2}(t)\int_{U(F)}{^w\!f}(t u ){^w\!du}
$$
where ${^w\!du}$ is the Haar measure on $U(F)$ which gives the volume $1$ to ${^w\ES{I}}\cap U$. 
Let $n_w$ be a representative element of $w$ in $N\cap \ES{K}$. Put $B_1= {^{w^{-1}}\!B}$ and $U_1= {^{w^{-1}}\!U}$. 
The map
$$
U(F)\rightarrow U_1(F), u \mapsto {\rm Int}_{n_w^{-1}}(u)
$$ 
is an isomorphism of $\mathfrak{p}$--adic varieties, with Jacobian $1$. Using
$$
\delta_B^{1/2}(t)\delta_{B_1}^{-1/2}({^{w^{-1}}\!t})=1,
$$
we obtain
$$
\zeta^G_T({^w\!f})(t)= \delta_{B}^{1/2}(t)\int_{U_1(F)}f(({^{w^{-1}}\!t}) u_1)du_1= 
f^{(B_1,\ES{I})}({^{w^{-1}}\!t})=\zeta^G_T(f)({^{w^{-1}}\!t})\leqno{(1)}
$$
where $du_1$ is the Haar measure on $U_1(F)$ which gives the volume $1$ to $\ES{I}\cap U_1$. Note that the above equality is just 
$\zeta^G_T\circ [w]= [w]_T \circ \zeta^G_T$.

On the other hand, for a $G$--regular element $t\in T(F)$, we have (\ref{a result of descent}.(7))
$$
\vert D_{G}(t) \vert_F^{-1/2}\Delta(t,t)= \Delta_T(t,t)=1,\leqno{(2)}
$$
and $\omega'_\chi = \chi$. From (2), the lemma 1 of \ref{a result of descent} gives
$$
\Delta(t,t){\bf O}_t(f) =  \sum_{w\in W}\zeta^G_T({^w\!f})(t).\leqno{(3)}
$$
Since
$$
\chi\cdot \zeta^G_T({^w\!f})= 
\zeta_{w,\chi}^T\circ \zeta^G_T({^w\!f})= \zeta_{w,\chi}({^w\!f}),\quad w\in W,$$
we obtain
$$
\chi(t)\wt{\bf O}_t(f)= \sum_{w\in W}\zeta_{w,\chi}({^w\!f})(t)= \bs{\zeta}_\chi(f)(t), \quad t\in T(F).\leqno{(4)}
$$
This proves the theorem of \ref{the theorem} in this case. Note that in this case, we do not need the ``variant with two characters'' for $T$. Note also that for 
$w\in W$, 
we have
$$
\zeta_{w,\chi}({^w\!f})(t) = \chi(t) \zeta^G_T(f)({^{w^{-1}}\!t})=\chi_w(t)^{-1} \zeta_\chi(f)({^{w^{-1}}\!t}), \quad t\in T(F).
$$
Thus in this case, the transfer $\bs{\xi}_{\chi_0}(f)=(\omega'_\chi)^{-1}\bs{\zeta}_\chi(f)\in C^\infty_{\rm c}(T(F))$ of $f$ is given by
$$
\bs{\xi}_{\chi_0}(f) = \sum_W \chi^{-1} \zeta_{w,\chi}({^w\!f})=\sum_{w\in W} {^w(\chi^{-1}\zeta_\chi)}.\leqno{(5)}
$$

\subsection{Reduction to the theorem* of \ref{variant with two characters}}\label{reduction to the variant}
Let us prove that the theorem of \ref{the theorem} is implied by the theorem* of \ref{variant with two characters}. 

We start with a remark on the condition of ellipticity for $G'$. Suppose the Weyl group $W'$ of $G'$ is contained in the Weyl group $W_M=N_M(T)/T$ of $M$ for 
a semi--standard $F$--Levi factor $M$ in $G$. Then we have $G'=M'$. Suppose moreover that $M$ is the {\it minimal} semi--standard $F$--Levi factor in $G$ such that 
$W'\subset W_M$, i.e. $M$ is the intersection of all the semi--standard $F$--Levi factors in $G$ containing $W'$. Let $A_M$ be the maximal central $F$--split torus in $M$. Since $M$ is $F$--split, $A_M$ coincides with the connected component $Z(M)^\circ$ of $Z(M)$. We put also $A_{G'}=Z(G')^\circ$. The inclusion $Z(M)\subset Z(G')$ induces an inclusion $A_{M}\subset A_{G'}$. The centralizer $L= Z_G(A_{G'})$ of $A_{G'}$ in $G$ is a semi--standard $F$--Levi factor in $G$, verifying the inclusions $W'\subset L \subset M$. Thus $L=M$, $A_M = A_{G'}$, and $G'$ is an elliptic endoscopic group of $M$. Reciprocally, for a semi--standard $F$--Levi factor $M$ in $G$, the equality $A_M=A_{G'}$ implies that $G'$ is an elliptic endoscopic group of $M$. 

Now let $f\in \ES{Z}(G(F),\rho)$. We want to prove that $f$ and
$$
\bs{\xi}_{\chi_0}(f) = (\omega'_\chi)^{-1}\vert W'\vert^{-1} \sum_{w\in W} \zeta_{w,\chi}({^w\!f})
$$
have matching strongly $G$--regular orbital integrals, i.e. that $\bs{\xi}_{\chi_0}(f)$ is a transfer of $f$.
So let $\delta\in G'(F)$ be a strongly $G$--regular element. Suppose $\delta$ belongs to $M'(F)$ for a semi--standard $F$--Levi factor $M$ in $G$. 
If the theorem of \ref{variant with two characters} is true for $M$, then the propositions of \ref{a result of descent} and \ref{about the hypothesis} imply that we have 
$$
{\bf SO}_\delta(\bs{\zeta}_\chi(f))= \omega'_\chi(\delta)\wt{\bf O}_\delta(f).
$$
Let $A_\delta$ be the maximal central $F$--split torus in $G'_\delta$ (recall that $G'_\delta= G'^{\delta,\circ}$ is the connected centralizer of $\delta$ in $G'$). 
There exists a $G'(F)$--conjugate of $g'^{-1}\delta g$ of $\delta$ such that $A_{g'^{-1}\delta g}= {\rm Int}_{g'^{-1}}(A_\delta)$ is contained 
in $T'=T$, so we can suppose (by replacing $\delta$ by $g'^{-1}\delta g'$) that $A_\delta\subset T'$. Let $M=Z_G(A_\delta)$. It is a semi--standard $F$--Levi factor in $G$, and the endoscopic group $M'$ of $M$ associated to $\chi_0$ as in \ref{some results of Roche} is given by $M'=Z_{G'}(A_\delta)$. Since $A_M = A_\delta = A_{M'}$, $M'$ is an {\it elliptic} endoscopic group of $M$, 
and $\delta$ is an {\it elliptic} strongly regular (in fact strongly $G$--regular) element in $M'(F)$. Hence if the theorem* of \ref{variant with two characters} is true for $M$, then the proposition of 
\ref{a result of descent} and the corollary of \ref{about the hypothesis} imply that we have
$$
{\bf SO}_\delta(\bs{\zeta}_\chi(f))= \omega'_\chi(\delta)\wt{\bf O}_\delta(f).
$$ 

So we are reduced to prove the theorem* of \ref{variant with two characters}.

\subsection{Reduction to $G'$ elliptic in the theorem of \ref{variant with two characters}}\label{reduction to G' elliptic}
Let us prove that if the theorem of \ref{variant with two characters} is true when the endoscopic group $G'$ is elliptic, then it is true in general.

Suppose the endoscopic $G'$ is not elliptic, and let $M= Z_G(A_{G'})$. Then $M'=G'$ is an elliptic endoscopic group of $M$, and (by hypothesis) 
the theorem of \ref{variant with two characters} is true for $M$. Let $f\in \ES{Z}(G(F), \rho_{\psi_0\chi_0})$ and $\delta\in G'(F)$ a strongly 
$G$--regular element. By the lemma 1 of \ref{a result of descent}, we have
$$
\wt{\bf O}_\delta(f)= \sum_{w\in W_M\backslash W} \wt{\bf O}_\delta^M(\zeta_M^G({^{\dot{w}}\!f})).\leqno{(1)}
$$
For $w\in W_M\backslash W$, the function $\zeta_M^G({^{\dot{w}}\!f})$ belongs to $\ES{Z}(M(F), \rho^M_{^{\dot{w}}\!(\psi_0\chi_0)})$ where $\rho^M_{^{\dot{w}}\!(\psi_0\chi_0)}$ 
is the inflation of the depth zero character ${^{\dot{w}}\!(\psi_0\chi_0)}$ of $\ES{T}$ to the Iwahori subgroup $\ES{I}_M$ of $M(F)$. Put $\xi_{\dot{w},0}= {^{\dot{w}}\!(\psi_0\chi_0)}\chi_0^{-1}$. 
For $w_M\in W_M$,  we have an injective 
morphism of ${\Bbb C}$--algebras
$$
\zeta_{w_M,\chi}^M: \ES{Z}(M(F), \xi_{\dot{w},0}\chi_0)\rightarrow \ES{Z}(G'(F),{^{w_M}\!(\xi_{\dot{w},0}\chi_0)}\chi_0^{-1}),
$$
and (by hypothesis) the function
$$
(\omega'_\chi)\vert_{M'(F)}^{-1}\vert W_M \vert^{-1} \sum_{w_M \in W_M}\zeta_{w_M,\chi}^M\circ \zeta^G_M({^{\dot{w}}\!f})\in C^\infty_{\rm c}(G'(F))
$$
is a transfer of $\zeta^G_M({^{\dot{w}}\!f})$. But for $w_M\in W_M$, 
$$
{^{w_M}\!(\xi_{\dot{w},0}\chi_0)}\chi_0^{-1}= {^{w_M\dot{w}}\!(\psi_0\chi_0)}\chi_0^{-1}
$$
and
$$
\zeta_{w_M,\chi}^M\circ \zeta^G_M\circ [\dot{w}]= \zeta_{w_M\dot{w},\chi}\circ [w_M\dot{w}],
$$
thus we have
$$
\vert W_M \vert^{-1} \sum_{w_M\in W_M} {\bf SO}_\delta(\zeta_{w_M\dot{w},\chi}({^{w_M\dot{w}}\!f}))= \omega'_\chi(\delta) \wt{\bf O}_\delta^M(\zeta^G_M({^{\dot{w}}\!f}).
$$
Summing over all the elements $w\in W_M\backslash W$, from (1) we obtain
$$
\vert W\vert^{-1} \sum_{w\in W}{\bf SO}_\delta(\zeta_{w,\chi}({^{w}\!f}))= \omega'_\chi(\delta) \wt{\bf O}_\delta(f).
$$
This proves that $\bs{\xi}_{\chi_0}(f) = (\omega'_\chi)^{-1} \vert W\vert^{-1} \sum_{w\in W}\zeta_{w,\chi}({^{w}\!f}))$ is a transfer of $f$.

\subsection{Further reductions (step 1)}\label{further reductions (step 1)}
As we explained in the introduction, a local--global argument will be necessary to finish the proof of the theorem* of \ref{variant with two characters}. 
We will use a simple trace formula, and it is possible to make it even simpler thanks to the following reductions. We proceed in two steps, as in \cite{Hai1, Hai2}:
\begin{enumerate}[leftmargin=17pt]
\item[(1)]We may assume that $G_{\rm der}$ is simply connected;
\item[(2)]We may assume that $G_{\rm der}$ is simply connected and $\mathfrak{Z}_G$ is a torus.
\end{enumerate}

Let us continue with the previous notations: $\chi_0$ is a depth zero character of $\ES{T}$, $\rho= \rho_{\chi_0}$ is the inflation of $\chi_0$ to 
$\ES{I}$, and $\chi$ is a $W_{\chi_0}$--invariant character of $T(F)$ extending $\chi_0$. For $w\in W$, $\rho'_w= \rho'_{\chi_{w,0}}$ is the inflation to $\ES{I}$ of the 
character $\chi_{w,0}= ({^w\!\chi}_0)\chi_0^{-1}$ of $\ES{T}$. 

We also fix another depth zero character $\psi_0$ of $\ES{T}$. Let $\eta = \rho_{\psi_0}$ be the inflation of $\psi_0$ to $\ES{I}$, and for $w\in W$, let 
$\xi_{w,0}$ be the character $({^w\!\psi_0})\chi_{w,0}= {^w\!(\psi_0\chi_0)} \chi_0^{-1}$ of $\ES{T}$, and 
$\tau'_w= \rho'_{\xi_{w,0}}$ the inflation of $\xi_{w,0}$ to $\ES{I}'$. Hence we have $\tau'_w = \rho'_{^w\!\psi_0}\rho'_w$.

\vskip1mm
For $w\in W$, we defined in \ref{variant with two characters} an injective morphism of ${\Bbb C}$--algebras
$$
\zeta_{w,\chi}: \ES{Z}(G'(F),{^w\!(\eta\rho)})\rightarrow \ES{Z}(G'(F), \tau'_w),
$$
and for $f\in \ES{Z}(G(F), \eta\rho)$, we put
$$
\bs{\zeta}_\chi(f) = \vert W'\vert^{-1}\sum_{w\in W} \zeta_{w,\chi}({^w\!f}).
$$
 
Let us prove (1). There exists a {\it $z$--extension} of $F$--split groups
$$
1\rightarrow Z_1\rightarrow G_1 \rightarrow G \rightarrow 1.
$$
By definition of a $z$--extension, $Z_1$ is a torus and $G_1$ is a connected reductive group such that $(G_1)_{\rm der}$ is simply connected. 
Let us give a characteristic--free short proof of the existence of such a $z$--extension. Let $\pi: G_{\rm SC}\rightarrow G$ be the simply connected covering of $G_{\rm der}$, 
and $T_{\rm sc}$ the inverse image of $T$ in $G_{\rm SC}$. It is easy to construct a short exact sequence of $F$--split tori $1 \rightarrow Z_1 \rightarrow T_1 \rightarrow T \rightarrow 1$ such that the morphism $T_{\rm sc} \rightarrow T$ lifts to an injective morphism $T_{\rm sc} \rightarrow T_1$ (cf. \cite[pp~298--299]{MS}). The group $T$ acts on $G$ via inner automorphisms: for $t\in T$, we choose an element $t_{\rm sc}\in T_{\rm sc}$ that lifts the projection of $t$ on $G_{\rm AD}$, and we make $t$ act on $G_{\rm sc}$ via ${\rm Int}_t$ (it does not depend on the choice of $t_{\rm sc}$). This defines a semidirect product $G_{\rm SC} \rtimes T$. We denote by $G_{\rm SC} *_{T_{\rm sc}} T $ the quotient of $G_{\rm SC} \rtimes T$ by the distinguished subgroup $\{(t_{\rm sc},\pi(t_{\rm sc})^{-1}):t_{\rm sc}\in T_{\rm sc}\}$. The morphism
$$
G_{\rm SC} \times T \rightarrow G,\, (g_{\rm sc},t) \mapsto \pi(g_{sc})t
$$
induces an identification $G= G_{\rm SC} *_{T_{\rm sc}}T$. We define in the same way the group $G_1 = G_{\rm SC}*_{T_{\rm sc}} T_1$. The morphism $T_1\rightarrow T$ induces a surjective morphism $G_1\rightarrow G$ with kernel $Z_1$, and the morphism $G_{\rm SC}\rightarrow  G_1$ induces an isomorphism $G_{\rm SC} \rightarrow (G_1)_{\rm der}$. Hence we are done.

So let us fix a $z$--extension of $F$--split groups $1 \rightarrow Z_1 \rightarrow G_1\rightarrow G \rightarrow 1$. The inverse image $T_1$ of $T$ in $G_1$ 
is a maximal $F$--split torus in $G_1$, and we have a short exact sequence of $F$--tori
$$
1\rightarrow Z_1\rightarrow T_1 \rightarrow T\rightarrow 1
$$
which induces a short exact sequence of commutative groups
$$
1\rightarrow Z_1(F)\rightarrow T_1(F)\rightarrow T(F)\rightarrow 1.
$$
Put $\ES{T}_1= T_1(\mathfrak{o})$ and $\ES{Z}_1= Z_1(\mathfrak{o})$. By applying the functor ${\rm Hom}(-,\mathfrak{o}^\times)$ to the 
short exact sequence of free ${\Bbb Z}$--modules
$$
0\rightarrow X=X(T)\rightarrow X(T_1)\rightarrow X(Z_1)\rightarrow 0,
$$
we obtain a short exact sequence of compact commutative groups
$$
1\rightarrow \ES{Z}_1\rightarrow \ES{T}_1 \rightarrow \ES{T}\rightarrow 1={\rm Ext}^1_{\Bbb Z}(X(Z_1),\mathfrak{o}^\times).
$$
It allows us to identify $\chi_0$ with a character of $\ES{T}_1$ trivial on $\ES{Z}_1$. It is also trivial on $\ES{T}_{1,+}$. 
Denote by $\ES{I}_1$ the Iwahori subgroup of $G_1(F)$ 
corresponding to $\ES{I}$ (for the identification of the building of $G_1(F)$ with the building of $G(F)$). Let $B_1$ be the Borel subgroup of 
$G_1$ corresponding to $B$ (i.e. the pre--image of $B$), $U_1$ the unipotent radical of $B_1$, and $\overline{U}_{\!1}$ the unipotent radical of the Borel subgroup 
$\overline{B}_1$ of $G$ opposite to $B_1$ with respect to $T_1$. The $F$--morphism $G_1\rightarrow G$ gives by restriction two isomorphisms
$$
\ES{I}_1\cap U_1\buildrel\simeq\over{\longrightarrow} \ES{I}\cap U,\quad \ES{I}_1\cap \overline{U}_{\!1}\buildrel\simeq\over{\longrightarrow} \ES{I}\cap \overline{U}.
$$
From the triangular decompositions for $\ES{I}_1$ and $\ES{I}$ (cf. \ref{some results of Roche}) and the previous short exact sequence, 
we obtain a short exact sequence of compact groups
$$
1\rightarrow \ES{Z}_1\rightarrow \ES{I}_1 \rightarrow \ES{I}\rightarrow 1
$$
which allows us to identify the character $\rho= \rho_{\chi_0}$ of $\ES{I}$ with a character of $\ES{I}_1$ trivial on $\ES{I}_{1,+}\ES{Z}_1=\ES{Z}_1\ES{I}_{1,+}$. The Weyl group $W_1= N_{G_1}(T_1)/T_1$ of $G_1$ is canonically identified with $W$, and the $W_{\chi_0}$--invariant character $\chi$ of $T(F)$ extending $\chi_0$ is identified with a $W_{\chi_0}$--invariant character of $T_1(F)$ trivial on $Z_1(F)$. Let $G'_1$ 
be the endoscopic group of $G_1$ defined by the character $\chi_0$ of $\ES{T}_1$. It is an extension of $G'$ by $Z_1$: we have a short exact sequence of $F$--split groups
$$
1\rightarrow Z_1\rightarrow G'_1\rightarrow G'\rightarrow 1.
$$
Let $\ES{I}'_1$ be the Iwahori subgroup of $G_1(F)$ corresponding to $\ES{I}'$. As in \ref{Bernstein center} we obtain an injective morphism of ${\Bbb C}$--algebras $\zeta_{1,\chi}: \ES{Z}(G_1(F),\rho)\rightarrow \ES{Z}(G'_1(F),1_{\ES{I}'_1})$. More generally, for all $w\in W$, as in \ref{variant with two characters} we obtain an injective morphism of ${\Bbb C}$--algebras 
$$
\zeta_{1,w,\chi}: \ES{Z}(G_1(F), {^w\!(\eta\rho)})\rightarrow \ES{Z}(G'_1(F),\tau'_w)
$$
where ${^w\!(\eta\rho)}$ is viewed as a character of ${^w\!(\eta\rho)}$ of ${^w(\ES{I}_1)}$ trivial on
$$
{^w(\ES{I}_{1,+})}\ES{Z}_1={^w(\ES{I}_{1,+}\ES{Z}_1)} = {^w(\ES{Z}_1\ES{I}_{1,+})}= \ES{Z}_1{^w(\ES{I}_{1,+})}
$$ and 
$\tau'_w$ is viewed as a character of $\ES{I}'_1$ trivial on 
$\ES{I}'_{1,+}\ES{Z}_1=\ES{Z}_1 \ES{I}'_{1,+}$.

For a function $\phi\in C^\infty_{\rm c}(G_1(F))$ and a character $\lambda$ of $Z_1(F)$, we put
$$
\phi_\lambda (g_1) = \int_{Z_1(F)}\phi(g_1z_1)\lambda(z_1)dz_1,\quad g_1\in G_1(F),
$$
where $dz_1$ is the Haar measure on $Z_1(F)$ which gives the volume $1$ to $\ES{Z}_1$. This defines a function $\phi_\lambda$ on $G_1(F)$ which belongs to the space 
$C^\infty_{\lambda}(G_1(F))$ of locally constant functions $f_1$ on $G_1(F)$ which are compactly supported modulo $Z_1(F)$ and verifie $f_1(z_1g_1)= \lambda(z_1)^{-1}f_1(g)$ for all $z_1\in Z_1(F)$ and all $g_1\in G_1(F)$. The function 
$\phi_{\lambda=1}$, denoted $\bar{\phi}$, is identified with an element of $C^\infty_{\rm c}(G(F))$. For a function 
$\phi'\in C^\infty_{\rm c}(G'_1(F))$ and a character $\lambda$ of $Z_1(F)$, we define the function $\phi'_\lambda\in C^\infty_{\lambda}(G'_1(F))$ in the same manner, and we put $\bar{\phi}'= 
\phi'_{\lambda=1}\in C^\infty_{\rm c}(G'(F))$. For all $w\in W$, the map
$$
C^\infty_{\rm c}(G_1(F))\rightarrow C^\infty_{\rm c}(G(F)),\,\phi\mapsto \bar{\phi}
$$
induces a surjective morphism of ${\Bbb C}$--algebras
$$
\ES{Z}(G_1(F),{^w\!(\eta\rho)})\rightarrow \ES{Z}(G(F),{^w\!(\eta\rho)}),
$$
and the map
$$
C^\infty_{\rm c}(G_1(F))\rightarrow C^\infty_{\rm c}(G(F)),\,\phi\mapsto \bar{\phi}
$$
induces a surjective morphism of ${\Bbb C}$--algebras
$$
\ES{Z}(G_1(F),\tau'_w)\rightarrow \ES{Z}(G(F),\tau'_w).
$$ 
By construction we have
$$
\zeta_{1,w,\chi} (\bar{\phi})= \overline{\zeta_{1,w,\chi}(\phi)},\quad \phi\in \ES{Z}(G_1(F),{^w\!(\eta\rho)}).\leqno{(3)}
$$
For $\phi\in \ES{Z}(G_1(F),\eta\rho)$, we put
$$
\bs{\zeta}_{1,\chi}(\phi)= \vert W'\vert^{-1} \sum_{w\in W} \zeta_{1,w,\chi}({^w\!\phi}).
$$
Note that the function $\bs{\zeta}_{1,\chi}(\phi)$ belongs to the subspace
$$
\bs{\ES{Z}}(G'_1(F))_{\psi_0,\chi_0}= \sum_{w\in W}\ES{Z}(G'_1(F),\tau'_w)
$$
of $\ES{Z}(G'_1(F),1_{\ES{Z}_1\ES{I}'_{1,+}})$. 
Since $\overline{^w\!\phi}= {^w\!\bar{\phi}}$ ($w\in W$), we obtain
$$
\overline{\bs{\zeta}_{1,\chi}(\phi)}= \bs{\zeta}_\chi(\bar{\phi}), \quad \phi\in \ES{Z}(G_1(F),\eta\rho). \leqno{(4)}
$$

The hyperspecial vertex of the building of $G(F)$ --- which is also the Bruhat--Tits building of $G_1(F)$ --- defining the hyperspecial maximal compact subgroup 
$\ES{K}$ of $G(F)$ defines also a hyperspecial maximal compact subgroup $\ES{K}_1$ of $G_1(F)$. From the 
decompositions $\ES{K}_1= \ES{I}_1W\ES{I}_1$ and $\ES{K}=\ES{I}W\ES{I}$ and the short exact sequence $1\rightarrow \ES{Z}_1\rightarrow \ES{I}_1\rightarrow \ES{I}\rightarrow 1$, 
we deduce a short exact sequence of 
compact groups
$$
1\rightarrow \ES{Z}_1\rightarrow \ES{K}_1 \rightarrow \ES{K}\rightarrow 1.
$$
For a pair $(\delta_1,\gamma_1)\in G'_1(F)\times G_1(F)$ of corresponding 
semisimple elements such that $\gamma_1$ is strongly regular in $G_1$, let $\Delta_1(\delta_1,\gamma_1)$ be the canonical transfer factor defined by the choice of $\ES{K}_1$. The image 
$(\delta,\gamma)\in G'(F)\times G(F)$ of $(\delta_1,\gamma_1)$ is a pair of corresponding semisimple elements such that $\gamma$ is strongly regular in $G(F)$, and from the definition of 
canonical transfer factors \cite[6.3]{MW}, we have the equality
$$
\Delta(z_1\delta_1,z_1\gamma_1)=\Delta_1(\delta_1,\gamma_1)= \Delta(\delta,\gamma),\qquad z_1\in Z_1(F). \leqno{(5)}
$$

The character $\omega'_\chi$ of $G'(F)$, viewed as a character of $G'_1(F)$ trivial on $Z_1(F)$, coincides with 
the character $\omega^{G'_1}_\chi$ of $G'_1(F)$ associated to $\chi$, 
viewed as a $W_{\chi_0}$--invariant character of $T_1(F)$ trivial on $Z_1(F)$, 
as in \ref{the theorem}. For $\phi\in \ES{Z}(G_1(F),\eta\rho)$, we put
$$
\bs{\xi}_{1,\chi_0}(\phi) = (\omega_{\chi}^{G'_1})^{-1}\bs{\zeta}_{1,\chi}(\phi).
$$
We have
$$
\overline{\bs{\xi}_{1,\chi_0}(\phi)}= \bs{\xi}_{\chi_0}(\bar{\phi}),\quad \phi\in \ES{Z}(G_1(F),\eta\rho).\leqno{(6)}
$$
\begin{prop}
Let $\phi\in C^\infty_{\rm c}(G_1(F))$ and $\phi'\in C^\infty_{\rm c}(G'_1(F))$.
\begin{enumerate}
\item[(i)] $\phi'$ is a transfer of $\phi$ if and only if $\phi'_\lambda$ is a transfer of $\phi_\lambda$ for every character $\lambda$ of $Z_1(F)$. 
\item[(ii)] if $\phi\in \ES{H}(G_1(F),1_{\ES{I}_{1,+}})$ and $\phi'\in \ES{H}(G'_1(F),1_{\ES{I}'_{1,+}})$, then in (i) we need only consider
characters $\lambda$ of $Z_1(F)$ that are trivial on $\ES{Z}_{1,+}$.
\item[(iii)] $\phi'_{\lambda=1}$ is a transfer of $\phi_{\lambda=1}$ (for the pair $(G_1,G'_1)$) if and only if $\bar{\phi}'$ is a transfer of $\bar{\phi}$ (for the pair $(G,G')$).
\end{enumerate}
\end{prop}

\begin{proof} 
It as a simple adaptation of \cite[5.3]{Hai1}, \cite[7.2]{Hai2}. Let $\phi\in C^\infty_{\rm c}(G_1(F))$ and $\phi'\in C^\infty_{\rm c}(G'_1(F))$. 
Let $\delta_1\in G'_1(F)$ be a strongly $G_1$--regular element. 
From (5), 
for a character $\lambda$ of $Z_1(F)$, we have
$$
\wt{\bf O}_{\delta_1}^{G_1}(\phi_\lambda)= \int_{Z_1(F)}\lambda(z_1)\wt{\bf O}_{z_1\delta_1}^{G_1}(\phi)dz_1
$$
and
$$
{\bf SO}_{\delta_1}^{G'_1}(\phi'_\lambda)=\int_{Z_1(F)}\lambda(z_1){\bf SO}_{z_1\delta_1}^{G'_1}(\phi')dz_1.
$$
This proves (i) and (ii). For $\lambda=1$, again from (5), we have
$$
\wt{\bf O}_{\delta_1}^{G_1}(\phi_{\lambda=1})= \wt{\bf O}_\delta(\bar{\phi})
$$
and
$$
{\bf SO}_{\delta_1}^{G'_1}(\phi'_{\lambda=1})= {\bf SO}_\delta(\bar{\phi}').
$$
This proves (iii).
\end{proof}

Note that the endoscopic group $G'$ of $G$ is elliptic if and only if the endoscopic group $G'_1$ of $G_1$ is elliptic. If $G'$ is elliptic and if 
$(\delta_1,\gamma_1)\in G'_1(F)\times G_1(F)$ is a pair of (semisimple) strongly $G_1$--regular elements which correspond, then the projection 
$(\delta,\gamma)$ of $(\delta_1,\gamma_1)$ on $(G'(F),G(F))$ is a pair of strongly $G$--regular elements which correspond, and $\delta_1$ is elliptic in $G'_1(F)$ 
if and only if $\delta$ is elliptic in $G'(F)$. Thus from (6) and the proposition, we may assume in the theorems of \ref{variant with two characters} 
that the derived group $G_{\rm der}$ of $G$ is simply connected.

\subsection{Further reductions (step 2)}\label{further reductions (step 2)}
Now let us prove the step (2) of \ref{further reductions (step 1)}. From the step (1) already proved in \ref{further reductions (step 1)}, we may assume that $G_{\rm der}$ is simply connected. From \cite[8]{R}, there exists an exact sequence of $F$--split groups
$$
1\rightarrow G \rightarrow G_1 \rightarrow C_1 \rightarrow 1
$$
where $C_1$ is a torus, and $G_1$ is a connected reductive group such that the scheme--theoretic center $\mathfrak{Z}_{G_1}$ of $G_1$ is a torus. 
We identify $G$ with a subgroup of $G_1$. Thus we have $(G_1)_{\rm der}=G_{\rm der}$. 
Let $T_1=Z_{G_1}(T)$. It is a maximal $F$--split torus in $G_1$. We have a short exact sequence of $F$--split tori
$$
1\rightarrow T \rightarrow T_1 \rightarrow C_1\rightarrow 1,
$$
which induces a short exact sequence of commutative groups
$$
1\rightarrow T(F) \rightarrow T_1(F) \rightarrow C_1(F)\rightarrow 1.
$$
Put $\ES{T}_1=T(\mathfrak{o})$ and $\ES{C}_1= C(\mathfrak{o})$. Let $U_F^1= 1+ \mathfrak{p}\subset \mathfrak{o}^\times$. 
By applying the functors ${\rm Hom}(-,\mathfrak{o}^\times)$, ${\rm Hom}(-,U_F^1)$ and ${\rm Hom}(-,\mathfrak{K})$ to the short exact sequence of free 
${\Bbb Z}$--modules
$$
0 \rightarrow X(C_1)\rightarrow X(T_1)\rightarrow X=X(T)\rightarrow 0, 
$$
we obtain two short exact sequences of compact commutative groups
$$
1\rightarrow \ES{T}\rightarrow \ES{T}_1\rightarrow \ES{C}_1\rightarrow 1,
$$
$$
1\rightarrow \ES{T}_+\rightarrow \ES{T}_{1,+}\rightarrow \ES{C}_{1,+}\rightarrow 1
$$
and a short exact sequence of finite commutative groups
$$
1\rightarrow T(\mathfrak{K})=\ES{T}/\ES{T}_+ \rightarrow \ES{T}_1/\ES{T}_{1,+} \rightarrow \ES{C}_1/\ES{C}_{1,+} \rightarrow 1.
$$
We also have a short exact sequence of locally compact commutative groups
$$
1 \rightarrow T(F)/\ES{T}_+ \rightarrow T_1(F)/\ES{T}_{1,+} \rightarrow C_1(F)/\ES{C}_{1,+} \rightarrow 1.
$$
Choose $\chi_{1,0}$ a depth zero character of $\ES{T}_1$ extending $\chi_0$, and let $\chi_1= \chi_{1,0}^\varpi$ be the character of $T_1(F)$ extending $\chi_{1,0}$ trivially on $\check{X}(T_1)$ for the identification $T_1(F)=\check{X}(T_1)\times \ES{T}_1$ given by $\varpi$. Then $\chi_1\vert_{T(F)}= \chi\;(= \chi_0^\varpi)$. 
The Weyl group $W_1= N_{G_1}(T_1)/T_1$ of $G_1$ is canonically identified with $W$, and it acts trivially on $C_1$. 
We have $W_{\chi_{1,0}}= W_{\chi_0}$ and $\chi_1$ is $W_{\chi_0}$--invariant. Let $G'_1$ be the endoscopic group of $G_1$ defined by the character $\chi_{1,0}$ of $\ES{T}_1$. We have a short exact sequence of $F$--split groups
$$
1\rightarrow G'\rightarrow G'_1\rightarrow C_1\rightarrow 1.
$$
Let $\ES{I}'_1$ be the Iwahori subgroup of $G_1(F)$ corresponding to $\ES{I}'$. Let $\psi_{1,0}$ be a depth zero character of $\ES{T}_1$ extending 
$\psi_0$, and let $\rho'_{\psi_{1,0}}$ be the inflation of $\psi_{1,0}$ to $\ES{I}'_1$. As in \ref{variant with two characters}, we define 
an injective morphism of ${\Bbb C}$--algebras
$$
\zeta_{\chi_1}: \ES{Z}(G_1(F),\eta_1\rho_1)\rightarrow \ES{Z}(G'_1(F),\rho'_{\psi_{1,0}}).
$$
We also have a natural morphism of ${\Bbb C}$--algebras
$$
\zeta^G_{G_1}: \ES{Z}(G(F),\eta\rho)\rightarrow \ES{Z}(G_1(F),\eta_1\rho_1)
$$
defined as follows. Let $\mathfrak{r}= \mathfrak{s}_{\psi_0\chi_0}$ be the $G(F)$--inertial class $[T,\psi\chi]_G$ of the cuspidal pair $(T,\psi\chi)$ of $G(F)$, and let $\mathfrak{r}_1=\mathfrak{s}_{\psi_{1,0}\chi_{1,0}}$ be the $G_1(F)$--inertial class $[T_1,\psi_1\chi_1]_{G_1}$ of the cuspidal pair 
$(T,\psi_1\chi_1)$ of $G_1(F)$, where $\psi$ is a character of $T(F)$ extending $\psi_0$ and $\psi_1$ is a character of $T_1(F)$ extending $\psi_{1,0}$. The choices of $\psi$ and $\psi_1$ 
do not matter (we can suppose $\psi_1\vert_{T(F)}=\psi$ but it is not necessary). The surjective morphism of 
algebraic varieties
$$
\mathfrak{X}(\mathfrak{r}_1) \rightarrow \mathfrak{X}(\mathfrak{r}),\, (T_1,\xi_1)_{G_1} \mapsto (T,\xi_1\vert_{T(F)})_G
$$
gives by duality an injective morphism of ${\Bbb C}$--algebras
$$
\eta^G_{G_1}: {\Bbb C}[\mathfrak{X}(\mathfrak{r})]\rightarrow {\Bbb C}[\mathfrak{X}(\mathfrak{r}_1)],
$$
and $\zeta^G_{G_1}$ is the unique morphism of ${\Bbb C}$--algebras which makes the following diagram
$$
\xymatrix{
{\Bbb C}[\mathfrak{X}(\mathfrak{r})] \ar[r]^(.45){\beta} \ar[d]_{\eta^G_{G_1}} &  \ES{Z}(G(F),\eta\rho)\ar[d]^{\zeta^G_{G_1}}\\
{\Bbb C}[\mathfrak{X}(\mathfrak{r}_1)] \ar[r]^(.45){\beta_1} & \ES{Z}(G_1(F),\eta_1\rho_1)
}\leqno{(1)}
$$
commutative. Here, $\beta_1$ is the isomorphism of ${\Bbb C}$--algebras given by the natural equivalence between $\mathfrak{R}^{\mathfrak{r}_1}(G_1(F))$ and the category of left $\ES{H}(G_1(F),\eta_1\rho_1)$--modules. We define in the same manner an injective morphism of ${\Bbb C}$--algebras
$$
\zeta^{G'}_{G'_1}: \ES{Z}(G'(F),\rho'_{\psi_0})\rightarrow \ES{Z}(G'_1(F),\rho'_{\psi_{1,0}}).
$$
By construction, the following diagram
$$
\xymatrix{\ES{Z}(G(F),\eta\rho) \ar[r]^{\zeta_\chi} \ar[d]_{\zeta^G_{G_1}} & \ES{Z}(G'(F),\rho'_{\psi_0}) \ar[d]^{\zeta^{G'}_{G'_1}} \\
\ES{Z}(G_1(F),\eta_1\rho_1) \ar[r]^{\zeta_{\chi_1}} & \ES{Z}(G'_1(F),\rho'_{\psi_{1,0}})
}\leqno{(2)}
$$
is commutative.

For $w\in W$, we define as in \ref{variant with two characters} an injective morphism of ${\Bbb C}$--algebras
$$
\zeta_{w,\chi_1}:\ES{Z}(G_1(F), {^w\!(\eta_1\rho_1)})\rightarrow \ES{Z}(G'_1(F), \tau'_{1,w})
$$
where $\tau'_{1,w}$ is the inflation of the character ${^w\!(\psi_{1,0}\chi_{1,0}}) \chi_{1,0}^{-1}$ of $\ES{T}_1$ to $\ES{I}_1$. We also define, as we have define 
$\zeta^G_{G_1}$ and $\zeta^{G'}_{G'_1}$ before, an injective morphism of ${\Bbb C}$-algebras (again denoted $\zeta^G_{G_1}$)
$$
\zeta^G_{G_1}: \ES{Z}(G(F), {^w\!(\eta\rho)})\rightarrow \ES{Z}(G_1(F), {^w\!(\eta_1\rho_1)})
$$
and an injective morphism of ${\Bbb C}$--algebras (again denoted $\zeta^{G'}_{G'_1}$)
$$
\zeta^{G'}_{G'_1} :\ES{Z}(G'(F), \tau'_w)\rightarrow \ES{Z}(G'_1(F), \tau'_{1,w}).
$$
By construction, the following diagram
$$
\xymatrix{\ES{Z}(G(F),{^w\!(\eta\rho)}) \ar[r]^{\zeta_{w,\chi}} \ar[d]_{\zeta^G_{G_1}} & \ES{Z}(G'(F),\tau'_w) \ar[d]^{\zeta^{G'}_{G'_1}} \\
\ES{Z}(G_1(F),{^w\!(\eta_1\rho_1)}) \ar[r]^{\zeta_{w,\chi_1}} & \ES{Z}(G'_1(F),\tau'_{1,w})
}\leqno{(3)}
$$
is commutative.

Let $\ES{K}_1$ be the hyperspecial maximal compact subgroup of $G_1(F)$ defined by the hyperspecial vertex of the building of $G(F)$ --- which is also the building of $G_1(F)$ --- 
defining $\ES{K}$. From the decompositions $\ES{K}=\ES{I}W\ES{I}$ and $\ES{K}_1= \ES{I}_1W\ES{I}_1$ and the short exact sequence of compact groups (cf. \ref{further reductions (step 1)})
$1\rightarrow \ES{I}\rightarrow \ES{I}_1\rightarrow \ES{C}_1\rightarrow 1$, we deduce a short exact sequence of compact groups
$$
1 \rightarrow \ES{K}\rightarrow \ES{K}_1 \rightarrow \ES{C}_1 \rightarrow 1.
$$
For a pair $(\delta_1,\gamma_1)\in G'_1(F)\times G_1(F)$ of corresponding semisimple elements such that $\gamma_1$ is strongly regular in $G_1$, 
let $\Delta_1(\delta_1,\gamma_1)$ be the canonical transfer factor defined by the choice of $\ES{K}_1$.

\begin{lem}
Let $(\delta,\gamma)\in G'(F)\times G(F)$ be a pair of corresponding semisimple elements such that $\gamma$ is strongly regular in $G$. Then 
$\gamma$ is strongly regular in $G_1$, 
and we have the equality
$$
\Delta_1(\delta,\gamma)= \Delta(\delta,\gamma).
$$
\end{lem}

\begin{proof} The centralizer $(G_1)^{\gamma}$ of $\gamma$ in $G_1$ is an extension of $C_1$ by $G^\gamma$: we have a short exact sequence of $F$--groups
$$
1\rightarrow G^\gamma \rightarrow (G_1)^\gamma \rightarrow C_1 \rightarrow 1.
$$
Since $C_1$ and $G^\gamma$ are tori, $(G_1)^\gamma$ is a torus too. Hence $\gamma$ is strongly regular in $G_1$. For the equality of transfer factors, we calculate 
$\Delta(\delta,\gamma)$ and $\Delta_1(\delta,\gamma)$. Put $S=G_\gamma$, $S_1= (G_1)_\gamma$, $S'= G'_\delta$ and $S'_1 = (G'_1)_\delta$. The $F$--torus $S_1$, resp. $S'_1$, is an extension of the $F$--split torus $C_1$ by $S$, resp. $S'$. We identify the dual torus $\check{S}$ to $\check{T}$ endowed with a Galois action defined by 
a $1$--cocycle of $\Gamma_F$ with values in $W= N_G(T)/T\;(=N_{\check{G}}(\check{T}))$. Hence $\check{S}_1$ identifies to $\check{T}_1$ with a Galois action defined by the same $1$--cocycle. 
The dual torus $\check{S}'$, resp. $\check{S}'_1$, identifies to $\check{S}$, resp. $\check{S}_1$, in a $\Gamma_F$--equivariant way. We have to fix some $a$--data and some $\chi$--data for the action 
of $\Gamma_F$ on $S$ (or $\check{S}$). They identifie to $a$--data and $\chi$--data for the action of $\Gamma_F$ on $S_1$ (or $\check{S}_1$). The formula in \cite[I, 6.3]{MW} 
gives\footnote{Recall that in loc.~cit., the transfer factor is the normalized one, that is our transfer factor divided 
by $\Delta_{\rm IV}(\delta,\gamma)= \vert D_G(\gamma)\vert_F^{1/2}\vert D_{G'}(\delta)\vert_F^{-1/2}$.}
$$
\Delta(\delta,\gamma)= \lambda_{\zeta}(\delta)^{-1}\lambda_z(\gamma) \langle (V_S,\nu_{\rm ad}),(t_{S,{\rm sc}},s_{\rm ad})\rangle^{-1} \Delta_{\rm II}(\delta,\gamma)\Delta_{\rm IV}(\delta,\gamma).
$$
Here $\lambda_\zeta$ is a character of $G'(F)$ and $\lambda_z$ is a character of $G(F)$. Since the Galois actions on the dual groups $\check{G}$ and $\check{G}'$ are trivial, these two characters are trivial. For each of the other three factors of $\Delta(\delta,\gamma)$, it is easy to verify that it coincides with the corresponding factor defined by replacing $G$ by $G_1$. 
\end{proof}

Let $\omega'_{\chi_1}= \omega_{\chi_1}^{G'_1}$ be the character of $G'_1(F)$ associated to $\chi_1$ as in \ref{the theorem}. Since $\chi_1\vert_{T(F)}= \chi$, the Langlands parameter 
$\varphi_\chi:W_F \rightarrow \check{T}$ of $\chi$ coincides with the composition of the Langlands parameter $\varphi_{\chi_1}: W_F \rightarrow \check{T}_1$ 
with the natural morphism $\check{T}_1 \rightarrow \check{T}$ (i.e. the one dual to the inclusion $T\subset T_1$), and we 
have
$$
\omega'_{\chi_1}\vert_{G'(F)}= \omega'_{\chi}.
$$

It is convenient to introduce the following normalisation of the elliptic regular orbital integrals. Let $A_G=Z(G)^\circ$ be the maximal central $F$--split torus in $G$, and $da$ the Haar measure on $A_G(F)$ which gives the 
volume $1$ to the maximal compact subgroup $A_G(\mathfrak{o})$ of $A_G(F)$. For a (semisimple) elliptic regular element $\gamma$ in $G(F)$, we can normalize the orbital integral ${\bf O}_\gamma$ on $G(F)$ by taking the Haar measure 
$dg_\gamma$ on $G_\gamma(F)$ such that
$$
{\rm vol}(A_G(F)\backslash G_\gamma(F), \textstyle{dg_\gamma \over da})=1.
$$
Recall that $dg$ is the Haar measure on $G(F)$ which gives the volume $1$ to $\ES{I}$. Hence we have
$$
{\bf O}_\gamma(f)= \int_{A_G(F)\backslash G(F)}f(g^{-1}\gamma g)\textstyle{dg\over da},\quad f\in C^\infty_{\rm c}(G(F)).
$$
We define in the same manner the normalized elliptic regular orbital integrals on $G_1(F)$, $G'(F)$ and $G'_1(F)$. 
Note that the endoscopic group $G'$ of $G$ is elliptic if and only if $A_{G'}=A_G$, in which case for a pair $(\delta,\gamma)\in G'(F)\times G(F)$ 
of corresponding semisimple elements such that $\delta$ is elliptic $G$--regular, the Haar measures $dg_\delta$ on $G'_\delta(F)$ and 
$dg_\gamma$ on $G_\gamma(F)$ normalized as above are compatible. On the other hand, the endoscopic group $G'$ of $G$ is elliptic 
if and only if the endoscopic group $G'_1$ of $G_1$ is elliptic. 

For $\phi\in \ES{Z}(G_1(F),\eta_1\rho_1)$, we put
$$
\bs{\zeta}_{\chi_1}(\phi)= \vert W'\vert^{-1}\sum_{w\in W} \zeta_{w,\chi_1}({^w\!\phi})\leqno{(4)}
$$
and
$$
\bs{\xi}_{\chi_{1,0}}(\phi) = (\omega'_{\chi_1})^{-1}\bs{\zeta}_{\chi_1}(\phi).
$$

\begin{prop}
Assume that the endoscopic group $G'$ of $G$ is elliptic, and that all the elliptic regular orbital integrals are the normalized ones.
\begin{enumerate}
\item[(i)]
For all $f\in \ES{Z}(G(F),\eta\rho)$ and all elliptic strongly $G$--regular $\delta\in G'(F)$, we have
$$
\wt{\bf O}_\delta(f) =\wt{\bf O}_\delta^{G_1}(\zeta^G_{G_1}(f)).
$$
\item[(ii)]For all $f\in \ES{Z}(G(F),\eta\rho)$ and all elliptic strongly $G$--regular $\delta\in G'(F)$, 
we have
$$
{\bf SO}_\delta(\bs{\zeta}_\chi(f)) = {\bf SO}_\delta^{G'_1}(\bs{\zeta}_{\chi_1}\circ \zeta^G_{G_1}(f)).
$$
\end{enumerate}
\end{prop}

\begin{proof}We need to compare orbital integrals on $G_1$, resp. $G'_1$, to orbital integrals on $G$, resp. $G'$. For this, let $\theta_0$ be 
a depth zero character of $\ES{T}$, and let $\theta_{1,0}$ be a depth zero character of $\ES{T}_1$ extending $\theta_0$. 
Since the Weyl group $W=W_1$ acts trivially on $C_1$, we have the equality $W_{\theta_{1,0}}= W_{\theta_0}$. 
We define as before (diagram (1)) a morphism of ${\Bbb C}$--algebras
$$
\zeta^G_{G_1}: \ES{Z}(G(F),\rho_{\theta_0}) \rightarrow \ES{Z}(G_1(F), \rho_{\theta_{1,0}}).
$$
Put $N_1= N_G(T_1)$ and $\ES{N}_1=N_1(\mathfrak{o})$. Let us choose a character $\tilde{\theta}_{1,0}$ of 
$(\ES{N}_1)_{\theta_{1,0}}$ extending $\theta_{1,0}$ \cite[6.11]{HL}, and 
let $\tilde{\theta}_1$ be the character 
of $N_1(F)_{\theta_{1,0}}= \check{X}_1\rtimes (\ES{N}_1)_{\theta_{1,0}}$
 extending $\tilde{\theta}_{1,0}$ trivially on $\check{X}_1= \check{X}(T_1)$. 
Then $\theta_1= \tilde{\theta}_1\vert_{T_1(F)}$ is a $W_{\theta_0}$--invariant character of 
$T_1(F)$ extending $\theta_{1,0}$. By restriction, we obtain also a character $\tilde{\theta}= \tilde{\theta}_1\vert_{N(F)_{\theta_0}}$ of $N(F)_{\theta_0}= N\cap N_1(F)_{\theta_{1,0}}$ extending $\theta_0$, and a 
$W_{\theta_0}$--invariant character $\theta = \tilde{\theta}\vert_{T(F)} \;(= \tilde{\theta}_1\vert_{T(F)})$ of $T(F)$. For $w_1\in (\ES{W}_1)_{\theta_{1,0}}= N_1(F)_{\theta_{1,0}}/\ES{T}_1$, we define as in 
\ref{Hecke algebra isomorphism} the element $f_{1,w_1,\tilde{\theta}_1}$ of the Hecke algebra $\ES{H}(G_1(F),\rho_{\theta_{1,0}})$: it is 
the unique element of $\ES{H}(G_1(F),\rho_{\theta_{1,0}})$ supported on $\ES{I}_1 w_1 \ES{I}_1= \ES{I}_1 n_{w_1} \ES{I}_1$ and having value 
$q^{-l_1(w)/2}\tilde{\theta}_1(n_{w_1})^{-1}$ at $n_{w_1}$ for any representative element $n_{w_1}$ of $w_1$ in $N_1(F)_{\theta_{1,0}}$. Here $l_1$ is the length function on 
$(\ES{W}_1)_{\theta_{1,0}}$ defined by the lemma of \ref{Hecke algebra isomorphism}.
For $w\in \ES{W}_{\theta_0}= N(F)_{\theta_0}/\ES{T}$, we define the element 
$f_{w,{\tilde{\theta}}}\in \ES{H}(G(F),\rho_{\theta_0})$ in the same manner. We have $\ES{W}_{\theta_0} \subset (\ES{W}_1)_{\theta_{1,0}}$, and the length function $l$ on 
$ \ES{W}_{\theta_0}$ coincides with the restriction $l_1\vert_{\ES{W}_{\theta_0}}$. The map
$$
f_{w,\tilde{\theta}} \mapsto f_{1,w,\tilde{\theta}_1},\quad w\in \ES{W}_{\theta_0},
$$
defines by linearity an injective morphism of ${\Bbb C}$--vector spaces
$$
\Psi^G_{G_1}: \ES{H}(G(F),\rho_{\theta_0})\rightarrow \ES{H}(G_1(F),\rho_{\theta_{1,0}}).
$$
From \cite[7.3]{Hai2} (cf. also \ref{Hecke algebra isomorphism} and \ref{Bernstein center}), $\Psi^G_{G_1}$ is a morphism of ${\Bbb C}$--algebras 
such that 
$$
\Psi^G_{G_1}\vert_{\ES{Z}(G(F),\rho_{\theta_0})}= \zeta^G_{G_1}.\leqno{(5)}
$$
Note that for $f\in \ES{H}(G(F),\rho_{\theta_0})$, we have
$$
{\rm Supp}(\Psi^G_{G_1}(f))= \ES{I}_1{\rm Supp}(f)\ES{I}_1.
$$
Let $w\in \ES{W}_{\theta_0}$. From \cite[7.4.5]{Hai2}, for an elliptic regular element $\gamma\in G(F)$, we have the equality
$$
{\bf O}_\gamma^{G_1}(f_{1,w,\tilde{\theta}_1}) =\sum_{\lambda_1} {\bf O}_{\gamma_{\lambda_1}}(f_{w,\tilde{\theta}})\leqno{(6)}
$$
where $\lambda_1$ runs over the elements of a finite group $\Lambda_1/\Lambda$ and $\gamma_{\lambda_1}= g_{\lambda_1}^{-1} \gamma g_{\lambda_1} $ for an element 
$g_{\lambda_1}\in G_1(F)$. More precisely, $\Lambda$ and $\Lambda_1$ are the finite groups defined by 
$\Lambda= \check{X}/ ({\Bbb Z}\check{\Phi}+ \check{X}(A_G))$ and $\Lambda_1= 
\check{X}_1/ ({\Bbb Z}\check{\Phi}+ \check{X}(A_{G_1})$ --- note that since $G_{\rm der}=D(G_1)$ is simply connected, the ${\Bbb Z}$--modules $\check{X}/{\Bbb Z}\check{\Phi}$ and $\check{X}_1/{\Bbb Z}\check{\Phi}$ are free ---, and $\Lambda$ is identified with a subgroup of $\Lambda_1$ via 
the natural injective morphism $\Lambda \rightarrow \Lambda_1$. There is also a surjective morphism $G_1(F)\rightarrow \Lambda_1$, 
which is the composition of the Kottwitz morphism $G_1(F)\rightarrow \check{X}_1/{\Bbb Z}\check{\Phi}$ with the natural projection 
$\check{X}_1/{\Bbb Z}\check{\Phi}\rightarrow \Lambda_1$. For $\lambda_1\in \Lambda_1$, we choose an element $g_{\lambda_1}\in G_1(F)$ that projects on $\lambda_1$. The correspondence 
$\Lambda_1 \rightarrow G(F),\, \lambda_1\mapsto \gamma_{\lambda_1}$ is not really defined (as a map), but it satisfies the following properties \cite[7.4.3]{Hai2}:
\begin{enumerate}
\item[(${\rm P}_1$)]two elements of the form $\gamma_{\lambda_1}$ and $\gamma_{\lambda'_1}$ are in the same $G(F)$--conjugacy class if and only if 
$\lambda_1$ and $\lambda'_1$ belong 
to the same coset in $\Lambda_1/\Lambda$;
\item[(${\rm P}_2$)]any element $\gamma'\in G(F) \cap \ES{I}_1 w \ES{I}_1$ belonging to the $G_1(F)$--conjugacy class of $\gamma$ is $G(F)$--conjugate to an element 
of the form $\gamma_{\lambda_1}$.
\end{enumerate}
Moreover, from \cite[7.4.2]{Hai2} we have:
\begin{enumerate}
\item[(${\rm P}_3$)] if an element $\gamma_1\in \ES{I}_1 w \ES{I}_1$ is stably conjugate to $\gamma$, then it is $G_1(F)$--conjugate to an element in 
$G(F)\cap \ES{I}_1 w \ES{I}_1$.
\end{enumerate}
From (5) and (6), for all functions $f\in \ES{Z}(G(F),\rho_{\theta_0})$, we have the equality
$$
{\bf O}_\gamma^{G_1}(\zeta^G_{G_1}(f)) =\sum_{\lambda_1\in \Lambda_1/\Lambda} {\bf O}_{\gamma_{\lambda_1}}(f).\leqno{(7)}
$$

Now we can apply these results  to the points (i) and (ii) of the proposition. Let $f\in \ES{Z}(G(F),\eta\rho)$ and $\delta\in G'(F)$ be an elliptic strongly $G$--regular element. 
For (i), we take $\theta_0= \psi_0\chi_0$ and $\theta_{1,0}= \psi_{1,0}\chi_{1,0}$. We have
$$
\wt{\bf O}_\delta^{G_1}(\zeta^G_{G_1}(f))= \sum_{\gamma_1}\Delta_1(\delta,\gamma_1) {\bf O}_\gamma^{G_1}(\zeta^G_{G_1}(f))\leqno{(8)}
$$
where $\gamma_1$ runs over the elements of $G_1(F)$ which correspond (in $G_1$) to $\delta$ modulo conjugation by $G_1(F)$. Let $\gamma\in G(F)$ be 
an elliptic strongly regular element which correspond to $\delta$ (in $G$). From (${\rm P}_3$), all the elements $\gamma_1\in G_1(F)$ giving a non--trivial contribution to the right hand side of  
(8) can be choosen in $G(F)$, and from (${\rm P}_1$), (${\rm P}_2$) and (7), the right hand side of (8) is equal to
$$
\sum_{\gamma'}\Delta_1(\delta,\gamma') {\bf O}_{\gamma'}(f)
$$
where $\gamma'$ runs over the elements of $G(F)$ which correspond (in $G$) to $\delta$ modulo conjugation by $G(F)$. Using the lemma, we obtain that the right hand side of 
(8) is equal to $\wt{\bf O}_\delta(f)$. The point (i) is proved. For (ii), we have
$$
\bs{\zeta}_{\chi_1}\circ \zeta^G_{G_1}(f) = \vert W'\vert^{-1}\sum_{w\in W} \zeta_{w,\chi_1} ({^w[\zeta^G_{G_1}(f)]}).
$$
Since we have
$$
{^w[\zeta^G_{G_1}(f)]}= \zeta^G_{G_1}({^w\!f}),\quad w\in W, 
$$
from (3) we obtain
$$
\bs{\zeta}_{\chi_1}\circ \zeta^G_{G_1}(f)= \vert W'\vert^{-1}\sum_{w\in W} \zeta^{G'}_{G'_1} \circ \zeta_{w,\chi}({^w\!f}).\leqno{(9)}
$$
Let $w\in W$. Replacing $G$ by $G'$, $f$ by $f'_w=\zeta_{w,\chi}({^w\!f})$, and $(\psi_0\chi_0, \psi_{1,0}\chi_{1,0})$ by the depth zero characters $\theta_0={^w\!(\psi_0\chi_0}) \chi_0^{-1}$ of $\ES{T}$ 
and $\theta_{1,0}={^w\!(\psi_{1,0}\chi_{1,0}}) \chi_{1,0}^{-1}$ of $\ES{T}_1$, as in the proof of (i), we obtain the equality
$$
{\bf SO}_\delta^{G'_1}(\zeta^{G'}_{G'_1}(f'_w))= {\bf SO}_\delta(f'_w).\leqno{(10)}
$$
From (9) and (10), we deduce the point (ii). 
\end{proof}

\begin{cor}If the theorem, resp. theorem*, of \ref{variant with two characters} is true for $G_1$, 
then it is true for $G$.
\end{cor}

\begin{proof}
If the theorem of \ref{variant with two characters} is true for $G_1$, then for all $f\in \ES{Z}(G(F), \rho_{\psi_0\chi_0})$, the function 
$\bs{\xi}_{\chi_{1,0}}(\zeta^{G}_{G_1}(f))$ is a transfer of $\zeta^G_{G_1}(f)$: for all strongly $G_1$--regular elements $\delta_1\in G_1(F)$, we have
$$
{\bf SO}_{\delta_1}(\bs{\zeta}_{\chi_1}\circ \zeta^G_{G_1}(f))= \omega'_{\chi_1}(\delta_1)\wt{\bf O}_{\delta_1}(\zeta^G_{G_1}(f)).
$$
In particular for $\delta_1=\delta\in G'(F)$, since $\omega'_{\chi_1}(\delta) = \omega'_\chi(\delta)$, we obtain
$$
{\bf SO}_{\delta}(\bs{\zeta}_{\chi}(f)= \omega'_{\chi}(\delta)\wt{\bf O}_{\delta}(f),
$$
which means that $\bs{\xi}_{\chi_0}(f)$ is a transfer of $f$. This also proves that if the theorem* of \ref{variant with two characters} is true for $G_1$, then it is true for $G$.
\end{proof}

We have proved that in the theorems of \ref{variant with two characters}, we may assume that the derived group $G_{\rm der}$ of $G$ is simply connected 
and the scheme--theoretic center $\mathfrak{Z}_G$ of $G$ is a torus (i.e. $\mathfrak{Z}_G=Z(G)=Z(G)^\circ$). 

\begin{rem}
{\rm 
Suppose $\mathfrak{Z}_G$ is a torus. Hence it is an $F$--split torus, and since $\mathfrak{Z}_G$ is the scheme--theoretic kernel of the natural morphism 
$G\rightarrow G_{\rm AD}$, the morphism $G(F)\rightarrow G_{\rm AD}(F)$ is surjective. If moreover the derived 
group $G_{\rm der}$ is simply connected\footnote{This condition is not really necessary. In fact the Step 1 runs for all short exact sequences of connected reductive $F$--split groups 
$1 \rightarrow Z_1 \rightarrow G_1 \rightarrow G \rightarrow 1$ such that $Z_1$ is a torus.}, then we may apply the Step 1 
to the $z$--extension
$$
1 \rightarrow \mathfrak{Z}_G \rightarrow G \rightarrow G_{\rm AD} \rightarrow 1.
$$
Thus we may also assume in the theorem* of \ref{variant with two characters} that $G=G_{\rm AD}$.
}\end{rem}


\section{Proof of the theorem of \ref{variant with two characters} (assuming the existence of local data)}\label{proof of the theorem assuming local data}

Let us continue with the notations of \ref{some results of Roche}---\ref{variant with two characters}: $G$ is a connected 
reductive $F$--split group, $T$ is a maximal $F$--split torus in $G$, ($\chi_0,\psi_0)$ is a pair of depth zero characters 
of the maximal compact subgroup $\ES{T}= T(\mathfrak{o})$ of $T(F)$, and $G'$ is the ($F$--split) endoscopic group of $G$ defined by $\chi_0$. 
We {\it do not} assume that the endoscopic group $G'$ of $G$ is elliptic. 

\subsection{Elementary functions}\label{elementary functions}
The elementary functions defined here are simple variants of those introduced by Labesse in \cite{Lab1}. 
Recall we have choosen in \ref{some results of Roche} a uniformizer element $\varpi$ of 
$F$, and $\chi$ is the $W_{\chi_0}$--invariant character $\chi^\varpi_0$ of $T(F)$ extending $\chi_0$ trivially on $\check{X}$, 
for the identification $T(F)= \check{X}\rtimes \ES{T}$ given by $\varpi$. 

\begin{nota}
{\rm For $\nu \in \check{X}= \check{X}(T)$, 
let $t_\nu = t_\nu^\varpi$ the element $\nu(\varpi)\in T(F)$. In other words, $t_\nu =\nu $ for the identification $T(F)= \check{X}\times \ES{T}$ given by $\varpi$. 
}
\end{nota}

Let $\nu \in \check{X}$ be a $B$--dominant regular co--character of $T$. Here {\it $B$--dominant} (or {\it dominant with respect to $B$}) means that $\langle \alpha, \nu \rangle \geq 0$ for all $\alpha \in \Delta$, and {\it regular} means that $\langle \alpha,\nu\rangle \neq 0$ for all $w\in \Phi$; thus $\nu\in \check{X}$ is $B$--dominant regular if and only if $\langle \alpha,\nu\rangle >0$ for all $\alpha\in \Delta$. 
To any element $g\in G(F)$, on can associate a $F$--parabolic subgroup $P_g$ of $G$ and a $F$--Levi component $M_g$ of $P_g$: $P_g$ (resp. $M_g$) is the set of elements 
$x\in G$ such that $x^n g x^{-n}$ remains {\it bounded} as $n$ ranges over all positive integers (resp. all integers). For $g= t_\nu$, since $\nu$ is $B$--dominant regular, we have $P_{t_\nu}= B$ and $M_{t_\nu}=T$.

The map
$$
\delta_\nu:\ES{T}\backslash \ES{I}\times t_\nu \ES{T} \rightarrow G(F),\, (k,y)\mapsto k^{-1} yk
$$
is injective and everywhere submersive. Thus its image is a compact open subset of $G(F)$ which we denote by $\ES{L}_\nu$. Moreover, the absolute value value of the Jacobian 
of $\delta_\nu$ is constant on $\ES{T}\backslash \ES{I}\times t_\nu \ES{T}$ and equal to the non--zero number
$$
\delta_B(t_\nu)^{-1}= q^{\langle 2\mu, \nu \rangle}\leqno{(1)}
$$
where $\mu$ denotes the half sum of positive roots of $G$. In \cite[prop. IV.1.1]{Lab2}, Labesse proved that
$$
\ES{L}_\nu = \ES{I}t_\nu \ES{I}.\leqno{(2)}
$$

\vskip1mm
As in \cite[8.1]{Hai2}, we define the elementary function $\mathfrak{f}_\nu=\mathfrak{F}_{\nu,\psi_0\chi_0}$ to be the function on $G(F)$ supported on $\ES{L}_\nu$
and verifying
$$
\mathfrak{f}_\nu(k^{-1}t_\nu t_0k)= q^{-l(t_\nu)/2}(\psi_0\chi_0)(t_0)^{-1},\quad k\in \ES{I},\, t_0\in \ES{T}.
$$
From \cite[lemma 8.1.3]{Hai2}, $\mathfrak{f}_\nu$ belongs to $\ES{H}(G(F),\rho_{\psi_0\chi_0})$, where $\rho_{\psi_0\chi_0}$ is the inflation 
of the depth zero character $\psi_0\chi_0$ of $\ES{T}$ to $\ES{I}$. 

\vskip1mm
Now let $\nu \in \check{X}$ be a $B'$--dominant $G'$--regular co--character of $T$, i.e. such that $w'(\nu)\neq \nu$ for all $w'\in W'$ and $\langle \alpha , \nu \rangle >0$ for all $\alpha \in \Delta'$. 
Similarly, we define the elementary function $\mathfrak{f}'_\nu = \mathfrak{f}'_{\nu,\psi_0}$ to be the fonction on $G'(F)$ supported on the compact open subset
$$
\ES{L}'_\nu= \{k'^{-1}t_\nu t_0k': k'\in \ES{I}',\,t_0\in \ES{T}\}\subset G'(F)
$$
and verifying
$$
\mathfrak{f}'_\nu(k'^{-1}t_\nu t_0k')= q^{-l'(t_\nu)/2}\psi_0(t_0)^{-1},\quad k'\in \ES{I}',\, t_0\in \ES{T}.
$$
From \cite[lemma 8.1.3]{Hai2}, $\mathfrak{f}'_\nu$ belongs to $\ES{H}(G'(F),\rho'_{\psi_0})$ where $\rho'_{\psi_0}$ is the inflation of the depth zero character $\psi_0$ 
of $\ES{T}$ to $\ES{I}'$. 

We will also need to consider the following variant of the elementary functions $\mathfrak{f}'_\nu$. Let $\nu\in \check{X}$ be a $B$--dominant regular co--character of $T$, and let $w\in W$. 
The co--character $w(\nu)$ of $T$ is still regular, but no longer $B$--dominant; in fact it is ${^w\!B}$--dominant, with ${^w\!B}= wB w^{-1}$. As in \ref{some results of Roche}, 
the Borel subgroup ${^w\!B}$ of $G$ defines a Borel subgroup $({^w\!B})'$ of $G'$, and the Iwahori subgroup ${^w\ES{I}}$ of $G(F)$ defines a Iwahori subgroup $({^w\ES{I}})'$ of $G'(F)$. Since 
$w(\nu)$ is a $({^w\!B})'$--dominant ($G$--)regular co--character of $T$, it is a fortiori $({^w\!B})'$--dominant and $G'$--regular. Since the Weyl group $W'=N_{G'}(T)/T$ acts simply transitively on the set of Borel subroup of $G'$ containing $T$, there exists a {\it unique} element $w'\in W'$ such that ${^{w'}\!({^w\!B}')}= B'$. Let $\tilde{w}= w'w$. Then $\tilde{w}(\nu)$ is a $B'$--dominant $G'$--regular co--character of $T$. 
We define the elementary function $\mathfrak{f}'_{w,\nu}=\mathfrak{f}'_{\tilde{w}(\nu),\xi_{\tilde{w},0}}$ to be the function on $G'(F)$ supported on the compact open subset
$$
\ES{L}'_{\tilde{w}(\nu)} = \{k'^{-1}t_{\tilde{w}(\nu)} t_0k': k'\in \ES{I}',\,t_0\in \ES{T}\}\subset G'(F)
$$
and verifying
$$
\mathfrak{f}'_{w,\nu}(k'^{-1}t_{\tilde{w}(\nu)}t_0 k') = q^{-l'(t_{\tilde{w}(\nu)})/2}\xi_{\tilde{w},0}(t_0)^{-1},\quad k'\in \ES{I}',\, t_0\in \ES{T}.
$$
From loc.~cit., $\mathfrak{f}'_{w,\nu}$ belongs to $\ES{H}(G'(F),\rho'_{\xi_{\tilde{w},0}})$ where $\rho'_{\xi_{\tilde{w},0}}$ is the inflation of the character 
$\xi_{\tilde{w},0}= {^{\tilde{w}}\!(\psi_0\chi_0)} \chi_0^{-1}$ of $\ES{T}$ to $\ES{I}'$. Note that for all $w_1\in W'w$, we have $\tilde{w}_1= \tilde{w}$ and (by definition) 
$\mathfrak{f}'_{w_1,\nu}= \mathfrak{f}'_{w,\nu}= \mathfrak{f}'_{\tilde{w},\nu}$. In particular, if $w\in W'$ then $\tilde{w}=1$ and 
$\xi_{\tilde{w}=1,0}= \psi_0$, thus $\mathfrak{f}'_{w,\nu}= \mathfrak{f}'_\nu$. 

\begin{lem}
For $\nu\in \check{X}$ and $w\in W$, we have $l'(t_{w(\nu)})= l'(t_\nu)$.
\end{lem}

\begin{proof}
Recall that $l'$ is the length function on $\ES{W}_{\chi_0}$ extending the one associated to the Coxeter system $(\ES{W}'\!,\ES{S}')$ --- cf. 
\ref{Hecke algebra isomorphism}. The character ${^{w}\!\chi_0}$ of $\ES{T}$ defines a $F$--group $\wt{G}'_{w} =G'_w \rtimes C_{^w\!\chi_0}$ with connected 
component $G'_w$, and $w$ induces by transport of structure an $F$--isomorphism 
$\wt{G}' \buildrel\simeq\over{\longrightarrow} \wt{G}'_w$ (cf. the remark 2 of \ref{conjugation by W}). The stabilizer $\ES{W}_{^w\!\chi_0}$ of ${^w\!\chi_0}$ in $\ES{W}= \check{X}\rtimes W$ 
is equal to
$$
{^w\!(\ES{W}_{\chi_0})}= \check{X}\rtimes {^w(W_{\chi_0})},
$$ and $l''=l'\circ w^{-1}$ is the length function on $\ES{W}_{^w\!\chi_0}$ extending the one associated 
to the Coxeter system $({^w\ES{W}'},{^w\ES{S}'})$ deduced from $(\ES{W}'\!,\ES{S}')$ by transport of structure. This implies the lemma.
\end{proof}

For a dominant $G'$--regular co--character $\nu \in \check{X}$, the function
$$
\bs{\mathfrak{f}}'_\nu = \vert W'\vert^{-1} \sum_{w\in W}\mathfrak{f}'_{w,\nu}= \sum_{w\in W'\backslash W} \mathfrak{f}'_{\tilde{w},\nu}
$$
belongs to the subspace $\sum_{w\in W'\backslash W} \ES{H}(G'(F),\rho'_{\xi_{\tilde{w},0}})$ of $\ES{H}(G'(F),1_{\ES{I}'_+})$, and the function
$$
\bs{\mathfrak{h}}'_\nu = (\omega'_\chi)^{-1} \bs{\mathfrak{f}}'_\nu
$$
belongs to the subspace $\sum_{w\in W'\backslash W} \ES{H}(G'(F),\rho'_{^{\tilde{w}}\!(\psi_0\chi_0)})$ of $\ES{H}(G'(F),1_{\ES{I}'_+})$

\subsection{Matching of elementary functions}\label{matching of elementary functions}
Let $\nu\in \check{X}$ be a $B$--dominant regular co--character of $T$. Put
$$
\mathfrak{U}_\nu = \bigcup_{w\in W} t_{w(\nu)} \ES{T}.
$$
Since $\nu$ is regular, the union is disjoint: 
$t_{w_1(\nu)}\ES{T}\cap t_{w_2(\nu)}\ES{T}=\emptyset $ if $w_1\neq w_2$. 
From the proposition 8.3.1 of \cite{Hai2}, for a strongly regular element $\gamma\in G(F)$, we have ${\bf O}_\gamma(\mathfrak{f}_\nu)=0$ if $\gamma$ is not 
$G(F)$--conjugate to an element in $t_\nu\ES{T}$, and for all $t_0\in \ES{T}$, we have
$$
{\bf O}_{t_\nu t_0}(\mathfrak{f}_\nu)=q^{-l(t_\nu)/2}(\psi_0\chi_0)(t_0)^{-1}.\leqno{(1)}
$$
In particular, the function $T(F)\rightarrow {\Bbb C},\, t\mapsto {\bf O}_t(\mathfrak{f}_\nu)$ 
is supported on the compact open $W$--invariant set $\mathfrak{U}_\nu$. 

Similarly, for a strongly regular element 
$\delta\in G'(F)$, we have ${\bf O}_\delta(\mathfrak{f}'_\nu)=0$ if $\delta$ is not $G'(F)$--conjugate to an element of 
$t_\nu\ES{T}$, and for all $t_0\in  \ES{T}$, we have
$$
{\bf O}_{t_\nu t_0}(\mathfrak{f}'_\nu)=q^{-l'(t_\nu)/2}\psi_0(t_0)^{-1}.\leqno{(2)}
$$
In particular, the function $T(F)\rightarrow {\Bbb C},\, t\mapsto {\bf O}_t(\mathfrak{f}'_\nu)$ 
is supported on the compact open $W'$--invariant 
subset 
$$
\mathfrak{U}'_\nu =\coprod_{w'\in W'}t_{w'(\nu)}\ES{T} \subset \mathfrak{U}_\nu.
$$ 

Let $w\in W$. Recall we put $\tilde{w}= w' w$, where $w'$ is the unique element of $W'$ such that ${^{w'}\!({^w\!B})'}= B'$. For a strongly regular $\delta\in G'(F)$, we have ${\bf O}_\delta(\mathfrak{f}'_{w,\nu})=0$ if $\delta$ is not $G'(F)$--conjugate to an element 
in $t_{\tilde{w}(\nu)}\ES{T}= {^{w'}\!(t_{w(\nu)}\ES{T})}$, i.e. if $\delta$ is not $G'(F)$--conjugate to an 
element in $t_{w(\nu)}\ES{T} = {^w\!(t_\nu \ES{T})}$, and for $t_0\in \ES{T}$, we have
$$
{\bf O}_{^w\!(t_\nu t_0)}(\mathfrak{f}'_{w,\nu})= {\bf O}_{^{\tilde{w}}\!(t_\nu t_0)} (\mathfrak{f}'_{w,\nu})
= q^{-l'(t_{\tilde{w}(\nu)})/2}\xi_{\tilde{w},0}(^{\tilde{w}}t_0)^{-1}
$$
hence
$$
{\bf O}_{^w\!(t_\nu t_0)}(\mathfrak{f}'_{w,\nu})=\xi_{\tilde{w},0}(^{\tilde{w}}t_0)^{-1}\psi_0(^{\tilde{w}}t_0) {\bf O}_{^w\!(t_\nu t_0)}(\mathfrak{f}'_{\tilde{w}(\nu)})\leqno{(3)}
$$
with
$$
 \xi_{\tilde{w},0}(^{\tilde{w}}t_0)^{-1}\psi_0(^{\tilde{w}}t_0)= (\psi_0\chi_0)(({^{\tilde{w}}t_0})t_0^{-1}).
$$
But we also have (using the lemma of \ref{elementary functions})
$$
{\bf O}_{^w\!(t_\nu t_0)}(\mathfrak{f}'_{\tilde{w}(\nu)})= q^{-l'(t_{\tilde{w}(\nu)})/2}\psi_0({^{\tilde{w}}t_0})^{-1}=\psi_0(({^{\tilde{w}}t_0})t_0^{-1})^{-1} {\bf O}_{t_\nu t_0}(\mathfrak{f}'_\nu),
$$
which gives
$$
{\bf O}_{^w\!(t_\nu t_0)}(\mathfrak{f}'_{w,\nu})= \chi_0(({^wt_0})t_0^{-1}){\bf O}_{t_\nu t_0}(\mathfrak{f}'_\nu)= \chi_{w^{-1}}(t_\nu t_0){\bf O}_{t_\nu t_0}(\mathfrak{f}'_\nu).\leqno{(4)}
$$
with $\chi_{w^{-1}} = ({^{w^{-1}}\!\chi})\chi^{-1}= ({^{\tilde{w}^{-1}}\!\chi})\chi^{-1}$ (since $\chi= \chi_0^\varpi$, we have $\chi_{w^{-1}}(t_\nu)=1$).

For a strongly $G$--regular element $t\in T(F)$, we have (by definition of the canonical transfer factor) 
$$
\Delta(t,t) =\Delta_T(t,t) \vert D_G(t)\vert_F^{1/2} \vert D_{G'}(t)\vert_F^{1/2}\leqno{(5)}
$$
with $\Delta_T=1$. 
On the other hand, for $t\in \mathfrak{U}_\nu$, we have $\vert D_G(t)\vert_F = q^{\langle 2\mu , \nu \rangle }$ where $\mu$ is the half sum of positive roots of $G$. 
 From \cite[2.4.1]{MD}, we have
 $$
 l(t_\nu) = \sum_{\alpha\in \Phi^+} \langle \alpha, \nu \rangle = \langle 2\mu, \nu \rangle.
 $$
 Thus $\vert D_G(t)\vert_F^{1/2}= q^{l(t_\nu)/2}$. We also have $\vert D_{G'}(t)\vert_F^{1/2}= q^{l'(t_\nu)/2}$. From (5) we obtain
 $$
 \Delta(t,t)q^{-l(t_\nu)/2}= q^{-l'(t_\nu)/2}.\leqno{(6)}
 $$
For $w\in W$, since $\Delta({^{w^{-1}}\!t},{^{w^{-1}}\!t})=\Delta(t,t)$, we have
 $$
 \chi({^{w^{-1}}\!t})\Delta({^{w^{-1}}\!t},{^{w^{-1}}\!t})= \chi_w(t)\chi(t) \Delta(t,t). \leqno{(7)}
 $$
 
Now let $t\in \mathfrak{U}_\nu$. There exists a unique element $w\in W$ such that $t\in t_{w(\nu)}\ES{T}$. 
Write $t_1= {^{w^{-1}}\!t}\in t_\nu\ES{T}$. From (1), (2) and (6), we have
$$
\chi(t_1)\Delta(t_1,t_1){\bf O}_{t_1}(\mathfrak{f}_\nu)= {\bf O}_{t_1}(\mathfrak{f}'_\nu),
$$
and from (7) we deduce
$$
\chi(t)\Delta(t,t) {\bf O}_t(\mathfrak{f}_\nu)= 
\chi_w(t)^{-1}{\bf O}_{t_1}(\mathfrak{f}'_\nu).
$$
Since $\chi_w(t)^{-1}=\chi_{w^{-1}}(t_1)$, from (4) we obtain
$$
\chi(t)\Delta(t,t) {\bf O}_t(\mathfrak{f}_\nu)=
{\bf O}_t(\mathfrak{f}'_{w,\nu})= {\bf O}_t(\mathfrak{f}'_{\tilde{w},\nu}).\leqno{(8)}
$$

Recall we have
$$
\omega'_\chi\vert_{T(F)}=\chi.
$$
Since
$$
\mathfrak{U}_\nu = \coprod_{w\in W'\backslash W}\mathfrak{U}'_{\tilde{w}(\nu)},
$$
the equality (8) says that $\bs{\mathfrak{h}}'_{\nu}= (\omega'_\chi)^{-1} \bs{\mathfrak{f}}'_\nu$ is a transfer of $\mathfrak{f}_\nu$.

\begin{rem}
{\rm 
Let us take the example of \ref{an example}: the depth zero character $\chi_0$ of $\ES{T}$ is {\it regular} (i.e. $G'=T$; in particular, $G'$ is not elliptic if $G\neq T$), and $\psi_0=1$. 
Suppose moreover that the center $Z(G)$ is connected. In that case, since the Hecke--algebra $\ES{H}(G(F),\rho)\simeq \ES{H}(G'(F),1_{\ES{I}'})={\Bbb C}[T(F)/\ES{T}]$ is commutative, we have $\ES{H}(G(F),\rho)=\ES{Z}(G(F),\rho)$ 
and we can apply the result of \ref{an example} 
to the elementary functions. Let $\nu\in \check{X}$ be a 
dominant regular co--character of $T$. We have $\zeta_\chi(\mathfrak{f}_\nu)= \mathfrak{f}'_\nu =q^{-l'(t_\nu)/2}1_{t_\nu\ES{T}}$ and 
$$
\zeta_{w,\chi}({^w\!(\mathfrak{f}_\nu)})= \chi^{-1}_w\cdot {^w\!(\zeta_\chi(\mathfrak{f}_\nu))}= \chi^{-1}_w\cdot {^w\!(\mathfrak{f}'_\nu)}
$$ 
is the unique element of $\ES{H}(T(F),\chi_{w,0})$ supported on $t_{w(\nu)}\ES{T}$ which takes the value 
$\chi_{w,0}(t_0)^{-1}q^{-l'(t_\nu)/2}$ at $t_{w(\nu)}t_0$ for all $t_0\in \ES{T}$. 
Hence we have
$$
\zeta_{w,\chi}({^w\!(\mathfrak{f}_\nu))}=  \mathfrak{f}'_{w,\nu},
$$
and from \ref{an example}, the function
$$
\bs{\xi}_{\chi_0}(\mathfrak{f}_\nu)=(\omega'_\chi)^{-1}\vert W'\vert^{-1}\sum_{w\in W}\mathfrak{f}'_{w,\nu}=\bs{\mathfrak{h}}'_{\nu}
$$
is a transfer of $\mathfrak{f}_\nu$, which is the result proved above.
}
\end{rem}

\subsection{Traces of elementary functions}\label{traces of elementary functions}
For a depth zero character $\theta_0$ of $\ES{T}$, let $\wt{\Xi}_{\theta_0}$ be the set of characters of $T(F)$ which extend some $W$--conjugate of $\theta_0$, and let 
$\Xi_{\theta_0}\subset \wt{\Xi}_{\theta_0}$ be the subset formed of those characters which extend $\theta_0$. Let $\mathfrak{s}_{\theta_0}$ be the 
$G(F)$--inertial class $[T,\theta]_G$ of the cuspidal pair $(T,\theta)$ of $G(F)$ for some character $\theta\in \wt{\Xi}_{\theta_0}$. We denote by $i_B^G: \mathfrak{R}(T(F))\rightarrow \mathfrak{R}(G(F))$ and $r_G^B: \mathfrak{R}(G)\rightarrow \mathfrak{R}(T(F))$ the normalized parabolic induction and Jacquet restriction functors (with respect to $U(F)$).
Let $\pi$ an irreducible representation in $\mathfrak{R}^{\mathfrak{s}_{\theta_0}}(G(F))$. Then $\pi$ is a subquotient of $i_B^G(\xi)$ for some $\xi\in \wt{\Xi}_{\theta_0}$, and by the geometric lemma, 
the Jacquet module $\pi_U = \delta_B^{1/2}r^G_B(\pi)$ is a subquotient of
$$
i_B^G(\xi)_U= \delta_B^{1/2} \bigoplus_{w\in W} {\Bbb C}_{^w\!\xi}.
$$
Here ${\Bbb C}_{^w\!\xi}$ denotes the $1$--dimensional representation of $T(F)$ corresponding to the character ${^w\!\xi}$. Thus there is a well defined subset 
$\wt{\Xi}_{\theta_0}(\pi)\subset \wt{\Xi}_{\theta_0}$ such that
$$
\pi_U = \delta_B^{1/2} \bigoplus_{\mu\in \wt{\Xi}_{\theta_0}(\pi)}{\Bbb C}_{\mu}^{a_{\mu}(\pi)},\quad a_{\mu}(\pi)>0.
$$
Let
$$
\Xi_{\theta_0}(\pi) = \Xi_{\theta_0}\cap \wt{\Xi}_{\theta_0}(\pi).
$$
For convenience of writing, if $\pi$ is an irreducible representation of $G(F)$ which is not in $\mathfrak{R}^{\mathfrak{s}_{\theta_0}}(G(F))$, we put
$$
\Xi_{\theta_0}(\pi)= \emptyset.
$$

Recall that we have an identification $T(F)= \check{X}\times \ES{T}$ given by the choice of the uniformizer element $\varpi$ of $F$. For a character $\mu$ of $T(F)$, 
we denote by $\mu^{\rm nr}: T(F)/\ES{T} \rightarrow {\Bbb C}^\times$ the unramified character of $T(F)$ defined by $\mu^{\rm un}= \mu\vert_{\check{X}}$. Thus, writing $\mu_0= \mu\vert_{\ES{T}}$, 
we have the decomposition
$$
\mu = \mu^{\rm nr} \mu_0.
$$
For a co--character $\nu\in \check{X}$, we thus have $\mu(t_\nu)= \mu^{\rm nr}(t_\nu)\buildrel {\rm def}\over{=} \langle \nu ,  \mu^{\rm nr} \rangle$.

We fixed in \ref{the theorem} the Haar measure $dg$ on $G(F)$: it is the one that gives the volume $1$ to the Iwahori subgroup $\ES{I}$ of $G(F)$. 
For a function $f\in C^\infty_{\rm c}(G(F))$, we denote by ${\rm trace}(\pi(f))$ the trace of the operator $\int_{G(F)}f(g)\pi(g)dg$ on the space of $\pi$ 
(it depends only on the isomorphism class of $\pi$).

Let take $\theta_0= \psi_0\chi_0$. To a $B$--dominant regular co-character $\nu \in \check{X}$, we have associated in \ref{elementary functions} a function 
$\mathfrak{f}_\nu \in \ES{H}(G(F),\rho_{\theta_0})$ with support in $\ES{L}_\nu$. From \cite[8.5.2]{Hai2}, we have the

\begin{prop1}
Let $\nu \in \check{X}$ be a $B$--dominant regular co--character of $T$, and let $\pi$ be an irreducible representation of $G(F)$. We have
$$
{\rm trace}(\pi(\mathfrak{f}_\nu))= q^{l(t_\nu)/2}\sum_{\mu\in \Xi_{\psi_0\chi_0}(\pi)}\mu(t_\nu) a_{\mu}(\pi)
$$
\end{prop1}

Now let take $\theta'_0= {^{\tilde{w}}\!(\psi_0\chi_0)}$ for an element $w\in W'\backslash W$. By replacing $G$ by $G'$ and $\theta_0$ by $\theta'_0$ 
in the previous construction, we 
define, for each irreducible representation $\pi'$ of $G'(F)$, a subset $\wt{\Xi}'_{\theta'_0}(\pi')$ of the set $\wt{\Xi}'_{\theta'_0}\;(\subset \wt{\Xi}_{\theta'_0})$ 
of characters of $T(F)$ extending a $W'$--conjugate of $\theta'_0$, and we put $\Xi_{\theta'_0}(\pi')= \Xi_{\theta'_0}\cap \wt{\Xi}'_{\theta'_0}(\pi')$. 
To a $B$--dominant ($G$--)regular co--character $\nu \in \check{X}$, we associated in \ref{elementary functions} a function 
$\mathfrak{f}'_{\tilde{w},\nu} \in \ES{H}(G(F),\rho'_{\xi_{\tilde{w},0}})$ with support in $\ES{L}'_{\tilde{w}(\nu)}$. The function 
$(\omega'_\chi)^{-1}\mathfrak{f}'_{\tilde{w},\chi}$ belongs to $\ES{H}(G'(F),\rho'_{\theta'_0})$ and coincides with $\mathfrak{f}'_{\!\tilde{w}(\nu),\theta'_0}$, thus 
the proposition 1 allows us to calculate ${\rm trace}((\omega'_\chi)^{-1}\mathfrak{f}'_{\tilde{w},\nu}))$ in terms of the exponents $\mu'\in \Xi_{\theta'_0}(\pi')$. 

Recall we put $\bs{\mathfrak{h}}'_{\nu} = (\omega'_\chi)^{-1}\sum_{w\in W'\backslash W} \mathfrak{f}'_{\tilde{w},\nu}$. Since 
$l'(t_{w(\nu)})= l'(t_\nu)$ for all $w\in W$ (lemma of \ref{elementary functions}), we obtain the

\begin{prop2}
Let $\nu \in \check{X}$ be a $B$--dominant regular co--character of $T$, and let $\pi'$ be an irreducible representation of $G'(F)$. We have
$$
{\rm trace}(\pi'(\bs{\mathfrak{h}}'_{\nu}))= q^{-l'(t_\nu)/2}\sum_{w\in W'\backslash W}\sum_{\mu'\in \Xi_{^{\tilde{w}}\!(\psi_0\chi_0)}(\pi')}\mu'(t_{\tilde{w}(\nu)}) a_{\mu'}(\pi')
$$
\end{prop2}

\begin{rem}
{\rm 
For a $B$--dominant regular co--character $\nu$ of $T$, we can also calculate the expression ${\rm trace}(\pi'(\bs{\mathfrak{f}}'_\nu))$ for any irreducible representation $\pi'$ of $G'(F)$: 
since $\bs{\mathfrak{f}}'_\nu = (\omega'_\chi)\bs{\mathfrak{h}}'_\nu$, we have
$$
{\rm trace}(\pi'(\bs{\mathfrak{f}}'_\nu))={\rm trace}((\omega'_\chi\otimes \pi')(\bs{\mathfrak{h}}'_\nu)).
$$
}
\end{rem}

\subsection{Local data}\label{local data}
Recall that we have choosen a $W_{\chi_0}$--invariant character $\chi$ of $T(F)$ extending $\chi_0$, and a character 
$\psi$ of $T(F)$ extending $\psi_0$. We denote by $\mathfrak{s}_0= \mathfrak{s}_{\psi_0\chi_0}$ the $G(F)$--inertial class 
$[T,\psi\chi]_G$ of the cuspidal 
pair $(T,\psi\chi)$ of $G(F)$. Let ${\rm Irr}^{\mathfrak{s}_0}(G(F))$ be the set of isomorphism classes of irreducible representations in 
$\mathfrak{R}^{\mathfrak{s}_0}(G(F))$. For $w\in W$, let $\mathfrak{s}'_{^w\!(\psi_0\chi_0)}$ be the $G'(F)$--inertial class $[T, {^w\!(\psi\chi)}]_{G'}$ of the cuspidal 
pair $(T,{^w\!(\psi\chi)})$ of $G'(F)$. Let
$$
\mathfrak{S}'_0 = \bigcup_{w\in W}\mathfrak{s}'_{^w\!(\psi_0\chi_0)}= \bigcup_{w\in W'\backslash W}\mathfrak{s}'_{^{\tilde{w}}\!(\psi_0\chi_0)}
$$ and 
${\rm Irr}^{\mathfrak{S}'_0}(G'(F))$ be the set of isomorphism classes of irreducible representations in 
$\mathfrak{R}^{\mathfrak{S}'_0}(G'(F))$. 

\vskip1mm
Let $(I, \{\alpha_i(\pi)\},\{\beta_i(\pi')\})$ be a triple where:
\begin{itemize}
\item $I$ is an indexing set (possibly infinite);
\item $\{\alpha_i(\pi)\}$ is a collection of complex numbers for $i\in I$ and $\pi \in {\rm Irr}^{\mathfrak{s}_0}(G(F))$;
\item $\{\beta_i(\pi')\}$ is a collection of complex numbers for $i\in I$ and $\pi' \in {\rm Irr}^{\mathfrak{S}'_0}(G'(F))$;
\item for a fixed $i\in I$, the complex numbers $\alpha_i(\pi)$ and $\beta_i(\pi')$ are zero for all but finitely many $\pi$ and $\pi'$. 
\end{itemize}

\begin{defi}
{\rm We say that $(I, \{\alpha_i(\pi)\},\{\beta_i(\pi')\})$ are some {\it local data for $(G,G')$ with respect to $(\ES{I},\ES{I}')$} if for any pair of functions $(f,f')\in C^\infty_{\rm c}(G(F))\times 
C^\infty_{\rm c}(G'(F))$ such that $f\in \ES{H}_0=\ES{H}(G(F),\rho_{\psi_0\chi_0})$ and $f'\in \bs{\ES{H}}'_0=\sum_{w\in W} \ES{H}(G'(F), \rho'_{^w\!(\psi_0\chi_0)})$, the following two conditions are equivalent:
\begin{enumerate}
\item[(A)]for all $i\in I$, we have
$$
\sum_\pi \alpha_i(\pi){\rm trace}(\pi(f))= \sum_{\pi'}\beta_i(\pi'){\rm trace}(\pi'(f'))
$$
where $\pi$, resp. $\pi'$, runs over the elements (or a set of representatives for the elements) of ${\rm Irr}^{\mathfrak{s}_0}(G(F))$, resp. ${\rm Irr}^{\mathfrak{S}'_0}(G'(F))$;
\item[(B)]for all strongly $G$--regular elements $\delta\in G'(F)$, we have
$$
{\bf SO}_\delta(f') = \wt{\bf O}_\delta(f).
$$
\end{enumerate}
}
\end{defi}

\subsection{Proof of the theorem of \ref{variant with two characters}}\label{proof of the theorem}
We assume there exists some local data $(I, \{\alpha_i(\pi)\},\{\beta_i(\pi')\})$ for $(G,G')$ as in 
\ref{local data}. The proof is a simple adaptation of \cite[9.2]{Hai2}. Put $\theta_0= \psi_0\chi_0$.

Let $f\in \ES{Z}(G(F),\rho_{\theta_0})$ and $f'= \bs{\xi}_{\chi_0}(f)$. With the notations of \ref{local data}, 
the pair $(f,f')$ belongs to $\ES{H}_0\times \bs{\ES{H}}'_0$. To prove that the functions 
$f$ and $f'$ satisfy \ref{local data}.(B) --- i.e. $f'$ is a transfer of $f$ ---, we have to prove they satisfy 
\ref{local data}.(A):
$$
\sum_\pi \alpha_i(\pi){\rm trace}(\pi(f))= \sum_{\pi'}\beta_i(\pi'){\rm trace}(\pi'(f')),\quad i\in I,\leqno{(1)}
$$
where $\pi$, resp. $\pi'$, runs over the elements of ${\rm Irr}^{\mathfrak{s}_0}(G(F))$, resp. ${\rm Irr}^{\mathfrak{S}'_0}(G'(F))$.

From \ref{elementary functions} and \ref{matching of elementary functions}, for all $B$--dominant regular characters $\nu\in \check{X}$, the pair  
$(\mathfrak{f}_\nu, \bs{\mathfrak{h}}'_{\nu})$ belongs to $\ES{H}_0 \times \bs{\ES{H}}'_0$ and $\bs{\mathfrak{h}}'_\nu$ is a 
transfer of $\mathfrak{f}_\nu$, thus we have
$$
\sum_\pi \alpha_i(\pi){\rm trace}(\pi(\mathfrak{f}_\nu))= \sum_{\pi'}\beta_i(\pi'){\rm trace}(\pi'(\bs{\mathfrak{h}}'_\nu)),\quad i\in I,\leqno{(2)}
$$
where $\pi$, resp. $\pi'$, runs over the elements of ${\rm Irr}^{\mathfrak{s}_0}(G(F))$, resp. ${\rm Irr}^{\mathfrak{S}'_0}(G'(F))$. 

Let fix an index $i\in I$. Let $\overline{\Xi}_{\theta_0}=\Xi_{\theta_0}/W_{\theta_0}$ be a set of representative elements of the $W_{\theta_0}$--orbits in $\Xi_{\theta_0}$. For 
$\theta\in \overline{\Xi}_{\theta_0}$, let ${\rm Irr}(\theta)$ be the subset of ${\rm Irr}^{\mathfrak{s}_0}(G(F))$ formed of those classes of 
irreducible representations of $G(F)$ which appear as a subquotient of $i_B^G(\theta)$. 
From the proposition 1 of \ref{traces of elementary functions}, the left hand side of (2) rewrites as
$$
q^{l(t_\nu)/2}\sum_{\theta \in \overline{\Xi}_{\theta_0}}\sum_{\pi\in {\rm Irr}(\theta)}\alpha_i(\pi)\sum_{\mu\in \Xi_{\theta_0}(\pi)}\mu(t_\nu) a_{\mu}(\pi).\leqno{(3)}
$$
In the same manner, for $w\in W$, let $\overline{\Xi}'_{^w\!(\theta_0)} = \Xi_{^w\!(\theta_0)}/W'_{^w\!(\theta_0)}$ be a set of 
representative elements for the $W'_{^w\!(\theta_0)}$--orbits in $\Xi_{^w\!(\theta_0)}$, and for $\theta'\in \overline{\Xi}'_{^w\!(\theta_0)}$, let 
${\rm Irr}'(\theta')$ be the subset of ${\rm Irr}^{\mathfrak{S}'_0}(G'(F))$ formed of those classes of 
irreducible representations of $G'(F)$ which appear as a subquotient of $i_{B'}^{G'}(\theta')$. 
From the proposition 2 of \ref{traces of elementary functions}, the right hand side of (2) rewrites as
$$
q^{l'(t_\nu)/2}\sum_{w\in W'\backslash W}\sum_{\theta' \in \overline{\Xi}'_{^{\tilde{w}}\!(\theta_0)}}\sum_{\pi'\in {\rm Irr}'(\theta')}\beta_i(\pi')
\sum_{\mu'\in \Xi_{^{\tilde{w}}\!(\theta_0)}(\pi')}\mu'(t_{\tilde{w}(\nu)}) a_{\mu'}(\pi').\leqno{(4)}
$$
Recall for a character $\mu$ of $T(F)$, we put $\mu(t_\nu)= \langle \nu, \mu^{\rm un}\rangle$; thus for $w\in W'\backslash W$, we have
$$
\mu(t_{\tilde{w}(\nu)})= \langle \tilde{w}(\nu), \mu^{\rm un} \rangle = \langle \nu, ({^{\tilde{w}^{-1}}\! \mu})^{\rm un}\rangle.
$$
We can view the equality $(3)=(4) $ as a linear combination of characters on the subset $\check{X}^{\natural}\subset \check{X}$ of $B$--dominant $G$--regular co--characters of $T$. 
By linear independence statement, we obtain that the equality $(3)=(4)$ remains true for $\nu=0$ (cf. \cite[9.2]{Hai2}). 

Let $e_{\theta_0}= e_{\rho_{\theta_0}}$ be the unit element of the Hecke algebra $\ES{H}(G(F),\rho_{\theta_0})$, i.e. the function on $G(F)$ with support $\ES{I}$ such that 
$e_{\theta_0}(k)= \rho_{\theta_0}(k)^{-1}$ for all $k\in \ES{I}$. For $w\in W$, let $e'_{^w\!(\theta_0)}= e_{\rho'_{^w\!(\theta_0)}}$ be the unit element of the 
Hecke algebra $\ES{H}(G'(F), \rho'_{^w\!(\theta_0)})$. Put $\tilde{\bs{e}}'_{\theta_0}= \sum_{w\in W'\backslash W}e'_{^{\tilde{w}}\!(\theta_0)}$. 
As in \cite[8.5 (corollary 8.5.3)]{Hai2}, we obtain
$$
\sum_\pi \alpha_i(\pi){\rm trace}(\pi(e_{\theta_0}))= \sum_{\pi'}\beta_i(\pi'){\rm trace}(\pi'(\tilde{\bs{e}}'_{\theta_0})),\quad i\in I,
$$
where $\pi$, resp. $\pi'$, runs over the elements of ${\rm Irr}^{\mathfrak{s}_0}(G(F))$, resp. ${\rm Irr}^{\mathfrak{S}'_0}(G'(F))$. 
In other words, $\tilde{\bs{e}}'_{\theta_0}$ is a transfer of $e_{\theta_0}$. Put $\bs{e}'_{\theta_0}= \vert W'\vert^{-1} \sum_{w\in W}e'_{^w\!(\theta_0)}$. Since for $w\in W'\backslash W$ and 
$w'\in W$, we have
$$
{\rm trace}(\pi'(e'_{^{w'\tilde{w}}\!(\theta_0)}))= {\rm trace}(\pi'(e'_{^{\tilde{w}}\!(\theta_0)}))
$$ for all irreducible representations $\pi'$ of $G'(F)$, we deduce that 
$\bs{e}'_{\theta_0}$ is a transfer of $e_{\theta_0}$. Note that $\bs{e}'_{\theta_0}= (\omega'_\chi)^{-1}\vert W'\vert^{-1} \sum_{w\in W} e_{\rho'_{\xi_{w,0}}}$ where $e_{\rho'_{\xi_{w,0}}}$ is the 
unit element of the Hecke algebra $\ES{H}(G'(F),\rho'_{\xi_{w,0}})$. 

For $w\in W$ and $\theta'\in \Xi_{^w\!(\theta_0)}$, 
the character ${^{w^{-1}}\!\theta'}$ of $T(F)$ belongs to $\Xi_{\theta_0}$, and the map $\theta'\mapsto {^{w^{-1}}\!\theta'}$ defines a map  
$\overline{\Xi}'_{^w\!(\theta_0)}\rightarrow \overline{\Xi}_{\theta_0}$. For $\theta\in \overline{\Xi}_{\theta_0}$, we denote by $\overline{\Xi}'_{^w\!(\theta_0)}(\theta)\subset 
\overline{\Xi}'_{^w\!(\theta_0)}$ the pre--image of $\theta$ 
by this map. The linear independence statement mentioned above gives in fact a stronger result (cf. \cite[lemma 9.2.2]{Hai2}): for all $i\in I$ and all $\theta\in \overline{\Xi}_{\theta_0}$, we have
$$
\sum_{\pi\in {\rm Irr}(\theta)}\alpha_i(\pi){\rm trace}(\pi(e_{\theta_0}))= \sum_{(w,\theta'\!,\pi')}\beta_i(\pi'){\rm trace}(\pi'(e'_{^{\tilde{w}}\!(\theta_0)})), \leqno{(5)}
$$
where $(w,\theta',\pi')$ runs over the triples such that $w\in W'\backslash W$, $\theta'\in \overline{\Xi}'_{^{\tilde{w}}\!(\theta_0)}(\theta)$ and 
$\pi'\in {\rm Irr}'(\theta')$. Equivalently, for all $i\in I$ and all $\theta\in \overline{\Xi}_{\theta_0}$, writing $\pi'_\chi= (\omega'_\chi)^{-1}\otimes \pi'$ ($\pi'\in {\rm Irr}(G'(F))$), we have
$$
\sum_{\pi\in {\rm Irr}(\theta)}\alpha_i(\pi){\rm trace}(\pi(e_{\theta_0}))= \vert W'\vert^{-1}\sum_{(w,\theta'\!,\pi')}\beta_i(\pi'){\rm trace}(\pi'_\chi(e_{\rho'_{\xi_{w,0}}})), \leqno{(6)}
$$
where $(w,\theta',\pi')$ runs over the triples such that $w\in W$, $\theta'\in \overline{\Xi}'_{^w\!(\theta_0)}(\theta)$ and 
$\pi'\in {\rm Irr}'(\theta')$. Let $\theta\in \overline{\Xi}_{\theta_0}$. Recall that the function $f\in \ES{Z}(G(F),\rho_{\theta_0})$ acts on $i_B^G(\theta)^{\rho_{\theta_0}}$ by multiplication by the scalar $\lambda_\theta(f)$, 
and for $w\in W$ and $\theta'\in \overline{\Xi}'_{^w\!(\theta_0)}(\theta)$, the function $\zeta_{w,\chi}({^w\!f})$ acts on $i_{B'}^{G'}(\chi^{-1}\theta')^{\rho'_{\xi_{w,0}}}$ by multiplication by the scalar
$$
\lambda'_{\chi^{-1}\theta'}(\zeta_{w,\chi}({^w\!f}))=\lambda_{\theta'}(f)=\lambda_\theta(f).
$$
Thus multiplying both sides of the equality (6) by $\lambda_\theta(f)$, we obtain
$$
\sum_{\pi\in {\rm Irr}(\theta)}\alpha_i(\pi){\rm trace}(\pi(f))= \vert W'\vert^{-1}\sum_{(w,\theta'\!,\pi')}\beta_i(\pi'){\rm trace}(\pi'_\chi(\zeta_{w,\chi}({^w\!f})))\leqno{(7)}
$$
with $\pi'_\chi(\zeta_{w,\chi}({^w\!f}))= \pi'((\omega'_\chi)^{-1}\zeta_{w,\chi}({^w\!f})))$. Then summing (7) over all $\theta \in \overline{\Xi}_{\theta_0}$ gives (1). 
This completes the proof of the theorem of \ref{variant with two characters}, assuming the existence of a local datum for $(G,G')$.

\subsection{Existence of local data by Waldspurger (${\rm char}(F)=0$)}\label{walds}

When ${\rm char}(F)=0$, Waldspurger gave us the following demonstration of the existence of local data. 

Let denote by $\bs{D}(G(F))$ the space of finite linear combinations of traces of (complex, smooth) irreducible tempered representations of $G(F)$, viewed as invariant distributions on $G(F)$. 
Let $\bs{SD}(G'(F))\subset \bs{D}(G'(F))$ be the subspace formed of those distributions that are stable. For a function $f'\in C^\infty_{\rm c}(G'(F))$, let us consider the two following conditions:
\begin{enumerate}
\item[(a)]for all strongly $G$--regular elements $\delta\in G'(F)$, we have ${\bf SO}_\delta(f')=0$;
\item[(b)]for all distributions $D'\in \bs{SD}(G'(F))$, we have $D'(f')=0$.
\end{enumerate}
From the lemma 6.3 of \cite{A}, we have\footnote{This result has been extended to the twisted case in \cite[XI.5.2]{MW}.}:
\begin{enumerate}
\item[(1)] conditions (a) and (b) are equivalent.
\end{enumerate}
From loc.~cit., we know the existence of spectral transfer: for all distribution $D'\in \bs{SD}(G'(F))$, there exists a unique element $D= {\rm T}(D')\in \bs{D}(G(F))$ verifying
$D(f)= D'(f')$ for all pairs of functions $(f,f')\in C^\infty_{\rm c}(G(F))\times C^\infty_{\rm c}(G'(F))$ such that $f'$ is a transfer of $f$. Let us fix a basis $\{D'_i\}_{i\in I}$ of $\bs{SD}(G'(F))$, 
and for each $i\in I$, let $D_i={\rm T}(D'_i)\in \bs{D}(G(F))$. We write $D'_i$ and $D_i$ explicitely as a linear combination of traces of irreducible tempered representations:
$$
D'_i= \sum_{\pi'} b_i(\pi') {\rm trace}(\pi'),\quad D_i = \sum_\pi a_i(\pi){\rm trace}(\pi).
$$
For a pair of functions $(f,f')\in C^\infty_{\rm c}(G(F))\times C^\infty_{\rm c}(G'(F))$, Let us consider the two following conditions:
\begin{enumerate}
\item[(A)]for all $i\in I$, we have $\sum_\pi a_i(\pi){\rm trace}(\pi(f))= \sum_{\pi'}b_i(\pi'){\rm trace}(\pi'(f'))$;
\item[(B)]for all strongly $G$--regular elements $\delta\in G'(F))$, we have ${\bf SO}_\delta(f')= \wt{\bf O}_\delta(f)$.
\end{enumerate}
Then we have:
\begin{enumerate}
\item[(2)] conditions (A) and (B) are equivalent.
\end{enumerate}
Let us prove (2). Let ${\rm T}(f)\in C^\infty_{\rm c}(G'(F))$ be a transfer of $f$. 
Condition (A) rewrites ${\rm T}(D'_i)(f)= D'_i(f')$ for all $i\in I$, which is equivalent to $D'_i({\rm T}(f)-f')=0$ for all $i\in I$, i.e. to 
$D'({\rm T}(f')-f')=0$ for all $D'\in \bs{SD}(G'(F))$. Applying (1), this is equivalent to $f'= {\rm T}(f)$ modulo --- as usual --- the functions whose strongly $G$--regular 
orbital integrals are zero, i.e. to (B). 

Now for all $i$, $\pi$, $\pi'$ as above, we put $\alpha_i(\pi)= a_i(\pi)$ if $\pi$ is in ${\rm Irr}^{\mathfrak{s}_0}(G(F))$, and $\alpha_i(\pi)=0$ if not, 
$\beta_i(\pi')=b_i(\pi')$ if $\pi'\in {\rm Irr}^{\mathfrak{S}'_0}(G'(F))$, and $\beta_i(\pi')=0$ if not. For a pair of functions $(f,f')\in \ES{H}_0\times \bs{\ES{H}}_0$, conditions (A) and (B) above coincide with conditions (A) and (B) of \ref{local data}. This proves the existence of local data when ${\rm char}(F)=0$.


\section{Existence of local data (${\rm char}(F)\geq 0$)}\label{existence of local data}
In this section \ref{existence of local data}, we change notations: $F$ is a global field, and $G$ is a connected reductive group defined over $F$. We fixe an algebraic closure 
$\overline{F}$ of $F$, and denote by $\Gamma_F$, resp. $W_F$, the Galois group, resp. Weil group, of $F^{\rm sep}/F$, where $F^{\rm sep}$ is the separable closure of $F$ in $\overline{F}$.

\subsection{The local--global argument}\label{the local-global argument}
In characteristic zero, Hales \cite{Hal} obtained local data to deduce the fundamental lemma for the whole spherical Hecke algebra at a given finite place $v_0$ of $F$ from the fundamental lemma for the unit element at almost all unramified places. A variant of this method has been used by Waldspurger \cite{W1} to deduce the geometric transfer conjecture at $v_0$ from the fundamental lemma for the unit at almost all unramified places, and by Arthur \cite{A} to prove the spectral transfer conjecture at $v_0$ (see \ref{walds}). 

In this section 3, we prove the existence of the local data (\ref{local data}) used in the proof of the theorem of \ref{variant with two characters} (\ref{proof of the theorem}), when ${\rm char}(F)=p >0$ and $G'$ is elliptic, assuming local hypothesis that are satisfied if $p$ is large enough. We first prove the existence of local data in characteristic zero, following \cite{Hal}. Then we show this proof extends to the positive characteristic. 

Let us start with an unspecified local problem in characteristic zero, involving a finite extension $F_0$ of a $p$--adic field ${\Bbb Q}_p$, a connected reductive group $G_0$ defined over $F_0$, an {\it elliptic}  endoscopic datum $\underline{G}'_0=(G'_0,\ES{G}'_0, s_0)$ of $G$, etc. We compare two simple trace formulas, one for $G$ and the other one for the underlying group $G'$ of 
an elliptic endoscopic datum $\underline{G}'$ of $G$. Both groups $G$ and $G'$ are defined over a number field $F$. The global situation is choosen in such a way that at a given finite place $v_0$ of $F$, we recover the objects of our local problem: the completion $F_{v_0}$ of $F$ at $v_0$ is isomorphic to $F_0$, the local group $G_{v_0}= G\times_F F_{v_0}$ is isomorphic to $G_0$, the local endoscopic datum $\underline{G}'_{v_0}$ is isomorphic to $\underline{G}'_0$, etc. 
Such a comparison appears in many papers, most of the time when $\underline{G}'$ is the {\it principal endoscopic datum}\footnote{In the non twisted case, this means that the underlying endoscopic group 
$G'$ is the quasi--split inner form of $G$. Note in that case, $\underline{G}'$ is elliptic {\it at all} places of $F$.} (e.g. \cite{K2}, or \cite{Lab3} in the twisted case of base change). In the principal endoscopic datum case,  
the use of a cuspidal stabilizing function --- e.g. a Kottwitz function\footnote{Kottwitz functions are generalizations of Euler--Poincar\'e functions. If the center of $G$ is $F$--anisotropic, a pseudo--coefficient of the Steinberg representation is, up to a sign, an Euler--Poincar\'e function; such a pseudo--coefficient is a Kottwitz function.} \cite[3.9]{Lab3} --- at {\it one} finite place of $F$ ensures that the elliptic part of the geometric side of the trace formula for $G$ is already stabilized: the only endoscopic contribution comes from $\underline{G}'$. In the 
case of a non--principal endoscopic datum $\underline{G}'$, it is possible to play a similar game, at least when $\underline{G}'$ is elliptic {\it at almost all places\footnote{This hypothesis is necessary to separate the elliptic endoscopic data of $G$. Note that since $\underline{G}'$ is unramified at almost all finite places of $F$, by Chebotarev density theorem, $\underline{G}'$ is elliptic and unramified at {\it one} finite place, if and only if it is elliptic at {\it infinitely many} places. But this does not implies that $\underline{G}'$ is elliptic at almost all places.} of $F$} (see \ref{an application}): we can isolate $\underline{G}'$ in the endoscopic decomposition of the elliptic part of the trace formula. Once we have isolated $\underline{G}'$, it is easy to produce local data following the method used by Clozel \cite{Cl} and Labesse \cite{Lab1} to prove the fundamental lemma for base change for the whole spherical Hecke algebra --- cf. Hales \cite{Hal}. 

Let us emphasize that the case of a non--principal endoscopic datum $\underline{G}'$ is diametrically opposite to the principal endoscopic datum case: not only is it in general impossible to isolate $\underline{G}'$ with a single finite place, but if $\underline{G}'$ is not elliptic at almost all places of $F$, it may be impossible to isolate $\underline{G}'$ by using all the finite places where it is elliptic --- even if they form an infinite set. In fact, Waldspurger\footnote{Personnal communication, 08--12--2016.} has constructed an example, when $G$ is the simply connected cover of the quasi--split group $U(4n)$ with respect to the 
quadratic extension ${\Bbb Q}(i)/{\Bbb Q}$, of two {\it non--isomorphic} (global) elliptic endoscopic data of $G$, that are elliptic at infinitely many places, and are isomorphic at all the places where they are elliptic.

\subsection{Some highlights on the trace formula}
From now to \ref{end of the proof}, we suppose $F$ is a number field. 

The stabilization of the elliptic part of the trace formula for $G$ has been done by Langlands \cite{Lan} for the elliptic regular terms, and then extended by Kottwitz \cite{K1} to all the elliptic elements. 
Note that the result of \cite{K1} implies quite directly the variant with a character $\omega$ of $G(F)$ --- which could be  
used in \cite{Hal} to avoid the reduction to the ``basic case''. This variant is contained in the more general twisted case, 
done by Labesse \cite{Lab5} for all the elliptic elements, after Kottwitz and Shelstad did it for the elliptic strongly regular ones \cite{KS}. 
In the meantime, Kottwitz \cite{K2} proved the Weil conjecture on the Tamagawa numbers for all semisimple simply connected groups satisfying the Hasse principle, and Chernousov \cite{Ch} proved the Hasse principle for the resistant $E_8$ case (thereby eliminating restrictive assumptions in 
\cite{K1}). Finally, thanks to the work of Ng\^o Bao Ch\^au and Waldspurger (and Shelstad for the archimedean case), the local transfer assumption, used in the last step of the stabilization \cite{Lan,K1,KS,Lab5}, is now a theorem.

In \ref{elliptic terms of the trace formula}--\ref{stabilization 2}, we briefly recall the result of Kottwitz \cite{K1}\footnote{We will only use the result of Langlands \cite{Lan}, whose statement is not really simpler.}, but we will mainly 
refer to the non--abelian cohomological constructions of Labesse \cite{Lab3,Lab5}.

\subsection{Elliptic terms of the trace formula}\label{elliptic terms of the trace formula}
Let ${\Bbb A}_F$ be the ring of ad\`eles of $F$. 
For a place $v$ of $F$, let $F_v$ be the completion of $F$ with respect to $v$, and put $G_v= G\times_F F_v$.

For the question we are interested with, {\bf we may assume that the maximal central $F$--split torus $A_G$ in $G$ is trivial.} 

Let $C^\infty_{\rm c}(G({\Bbb A}_F))$ be the space of smooth compactly supported functions on $G({\Bbb A}_F)$. For a semisimple element $\gamma\in G(F)$ and a function 
$f\in C^\infty_{\rm c}(G({\Bbb A}_F))$, we define 
the orbital integral ${\bf O}_\gamma(\phi)={\bf O}_\gamma^G(f)$ by
$$
{\bf O}_\gamma(f)= \int_{I({\Bbb A}_F)\backslash G({\Bbb A}_F)} f(g^{-1}\gamma g) d\dot{g}_\gamma,\leqno{(1)}
$$
where $I$ is the ``stable centralizer'' of $\gamma$ in $G$, which coincides with the connected centralizer $G_\gamma=G^{\gamma,\circ}$. 
Hence $I$ is a connected reductive group defined over $F$. The integral (1) is absolutely convergent. The groups $G({\Bbb A}_F)$, $I({\Bbb A}_F)$, etc., are endowed with the Tamagawa measures. 
Thus $d\dot{g}_\gamma$ in (1) 
is the quotient of the Tamagawa measure on $G({\Bbb A}_F)$ by the Tamagawa measure on $I({\Bbb A}_F)$.

We denote by $R$ the right regular representation of $G({\Bbb A}_F)$ into the space 
$\bs{\mathfrak{X}}_G={\rm L}^2(G(F)\backslash G({\Bbb A}_F))$. 
For $f\in C^\infty_{\rm c}(G({\Bbb A}_F))$, let $K_{\rm e}(f)$ be the contribution of the elliptic elements to the kernel of the operator $R(f)$ on 
$\bs{\mathfrak{X}}_G$:
$$
K_{\rm e}(f;x,y)= \sum_{\gamma\in G(F)_{\rm e}}f^0(x^{-1}\gamma y),\leqno{(2)}
$$
where $G(F)_{\rm e}\subset G(F)$ denotes the subset of (semisimple) elliptic elements\footnote{Since $A_G=\{1\}$, a semisimple element of $G(F)$ is elliptic if and only if it belongs to an $F$--anistropic torus.}. 
We denote by ${\bf T}_{\rm e}$ the linear form on $C^\infty_{\rm c}(G({\Bbb A}_F))$ obtained by integration along the diagonal of $K_{\rm e}$:
$$
{\bf T}_{\rm e}(f)= \int_{G(F)\backslash G({\Bbb A}_F)}K_{\rm e}(f;x,x)d\dot{x}.\leqno{(3)}
$$
It is the contribution of the elliptic elements to the trace formula. For $\gamma\in G(F)_{\rm e}$, let
$$
a^G(\gamma) = {\tau(G_\gamma)\over \iota_G(\gamma)}
$$
where $\iota_G(\gamma)$ is the index of $G_\gamma(F)$ in $G^\gamma(F)$, and $\tau(G_\gamma)={\rm vol}(G_\gamma(F)\backslash G_\gamma({\Bbb A}_F))$ the Tamagawa number of $G_\gamma$.  
The integral (3) is absolutely convergent, and we have
$$
{\bf T}_{\rm e}(f)= \sum_\gamma a^G(\gamma){\bf O}_\gamma(f^0)\leqno{(4)}
$$
where $\gamma$ runs over the $G(F)$--conjugacy classes in $G(F)_{\rm e}$.

\subsection{Stabilization (step 1)}\label{stabilization 1}
The stabilization of the elliptic part of the trace formula is done in two steps. The first step, called pre--stabilization, consists in rewriting 
\ref{elliptic terms of the trace formula}.(4) in terms of stable conjugacy classes instead of ordinary conjugacy classes. Recall that two semisimple elements $\gamma,\,\gamma'\in G(F)$ are {\it stably conjugate} if there exists an element $x\in G(\overline{F})$ such that $x^{-1}\gamma x= \gamma'$ and $x\sigma(x)^{-1}\in G_\gamma(\overline{F})$ for all $\sigma\in \Gamma_F$. Thus, the set of $G(F)$--conjugacy class within the stable conjugacy class of $\gamma$ is parametrized by the pointed set
$$
\ES{D}(G_\gamma,G;F)={\rm coker}[{\bf H}^0(F,G)\rightarrow {\bf H}^0(F,G_\gamma\backslash G)].
$$
This leads to some variant of the orbital integral 
\ref{elliptic terms of the trace formula}.(1), the {\it $\kappa$--orbital integrals}, we define bellow. 

Let $\gamma\in G(F)_{\rm e}$ and $I=G_\gamma$. The {\it endoscopic characters} for $\gamma$ are the elements of the Pontryagin dual $\mathfrak{K}(I,G;F)_1$ of
$$
\mathfrak{E}(I,G;{\Bbb A}_F/F)={\rm coker}[{\bf H}^0_{\rm ab}({\Bbb A}_F,G) \rightarrow {\bf H}^0_{\rm ab}({\Bbb A}_F/F, I\backslash G)].
$$
Here we adopt the notations of Labesse (cf. \cite[1.8]{Lab3} or \cite[IV.1]{Lab5}), who defines groups of abelianized cohomology in the sense of Borovoi considering 
complex or bi--complex of reductive groups in {\it homological} degrees $0$ and $1$, and taking the $0$--cohomological group. 
For the definition and properties of adelic cohomology, we refer to \cite[1.4]{Lab3}. Put $\bar{\Bbb A}_F = {\Bbb A}_F \otimes_F\overline{F}$. Let $x\in G(\bar{\Bbb A}_F)$ such that $x\sigma(x)^{-1}\in I(\bar{\Bbb A}_F)$ for all $\sigma\in \Gamma_F$. We put $\gamma_x=x^{-1}\gamma x$ and $I_x = x^{-1}Ix$. Then $x$ defines a class $\dot{x}\in {\bf H}^0({\Bbb A}_F,I\backslash G)$, and for 
$\kappa\in \mathfrak{K}(I,G;F)_1$, we denote by $\kappa(x)$ the evaluation of $\kappa$ on the image of $\dot{x}$ by the natural map
$$
{\bf H}^0({\Bbb A}_F,I\backslash G) \rightarrow {\bf H}^0_{\rm ab}({\Bbb A}_F,I\backslash G).
$$
Let
$$
\mathfrak{D}(I,G;{\Bbb A}_F)={\rm coker}[{\bf H}^0({\Bbb A}_F,G) \rightarrow {\bf H}^0({\Bbb A}_F, I\backslash G)]
$$
The sets ${\bf H}^0({\Bbb A}_F,I\backslash G)$ and $\mathfrak{D}(I,G;{\Bbb A}_F)$ are endowed with a Tamagawa measure \cite{Lab4}. 
 For a function $\phi\in C^\infty_{\rm c}(\mathfrak{A}_G\backslash G({\Bbb A}_F))$, we define the $\kappa$--orbital integral
$$
{\bf O}^\kappa_\gamma(\phi)= \int_{{\bf H}^0({\Bbb A}_F,I\backslash G)}\kappa(x)\phi(\gamma_x)dx
$$
and the stable orbital integral
$$
{\bf SO}_\gamma(\phi)= {\bf O}_\gamma^{\kappa =1}(\phi)=\int_{\mathfrak{D}(I,G;{\Bbb A}_F)}{\bf O}_{\gamma_x}(\phi)dx.
$$

\begin{rem}
{\rm If $\phi=\prod_v\phi_v$, then the $\kappa$--orbital integral ${\bf O}^\kappa_\gamma(\phi)$ decomposes as a product of local $\kappa_v$--orbital integrals
$$
{\bf O}^\kappa_\gamma(\phi)= \prod_v {\bf O}^{\kappa_v}_\gamma(\phi_v)
$$
with
$$
{\bf O}^{\kappa_v}_\gamma(\phi_v)=\int_{{\bf H}^0(F_v;I_v\backslash G_v)}e((I_v)_{x_v})\kappa_v(x_v)\phi_v(\gamma_{x_v})dx_v.
$$
Here $e((I_v)_{x_v})$ is the local Kottwitz sign (this sign is $1$ if $(I_v)_{x_v}$ is quasi--split), and 
$\kappa_v$ is the image of $\kappa$ by the localization map
$$
\mathfrak{K}(I,G;F)_1\rightarrow \mathfrak{K}(I_v,G_v;F_v)_1
$$
where $\mathfrak{K}(I_v,G_v;F_v)_1$ is the Pontryagin dual of the abelian group
$$
\mathfrak{E}(I_v,G_v;F_v)= {\rm coker}[{\bf H}^0_{\rm ab}(F_v,G_v) \rightarrow {\bf H}^0_{\rm ab}(F_v,I_v\backslash G_v)].
$$
The Tamagawa measure $dx$ on ${\bf H}^0({\Bbb A}_F,I\backslash G)$ gives local invariant 
measures $dx_v$ on ${\bf H}^0(F_v, I_v\backslash G_v)$ for each place $v$ of $F$ (these local measures are not really defined but their product is well defined). 
Since $I$ is connected, we know by the Kottwitz reciprocity law (cf. \cite[1.7.2]{Lab3}) that for $x\in {\bf H}^0({\Bbb A}_F,I\backslash G)$, we have
$\prod_v e((I_x)_v)=1$. On the other hand (cf. \cite[2.7.1]{Lab3}), we have ${\bf O}_\gamma^{\kappa_v}(\phi_v)=1$ for almost all $v$.
}
\end{rem} 

For $f\in C^\infty_{\rm c}(G({\Bbb A}_F))$, we have (cf. \cite[V.2.1]{Lab5})
$$
{\bf T}_{\rm e}(f)= \tau(G)\sum_\gamma {1\over \tilde{\iota}_G(\gamma)}\sum_{\kappa} {\bf O}^\kappa_\gamma(f^0)\leqno{(1)}
$$
where $\gamma$ runs over a set of representative elements of the stable conjugacy classes in $G(F)_{\rm e}$, 
$\tilde{\iota}_G(\gamma)$ is the number of elements of the finite set ${\bf H}^0(F,G_\gamma \backslash G^\gamma)$ --- this number is constant on the stable conjugacy classes (cf. 
\cite[III.3]{Lab5}) ---, and $\kappa$ runs over 
the elements of $\mathfrak{K}(G_\gamma,G;F)_1$. Note that $\tilde{\iota}_G(\gamma)= \iota_G(\gamma)j(\gamma)$ whith
$$
j(\gamma) = \#  \ker[{\bf H}^1(F,G_\gamma)\rightarrow {\bf H}^1(F,G^\gamma)].
$$
If $G$ is quasi--split (over $F$), we define the stable analogue ${\bf ST}_{\rm e}$ of ${\bf T}_{\rm e}$ by
$$
{\bf ST}_{\rm e}(f)= \tau(G)\sum_\gamma {1\over \tilde{\iota}_G(\gamma)}{\bf SO}_\gamma(f^0)\leqno{(2)}
$$

\subsection{Stabilization (step 2)}\label{stabilization 2}
The second step of the stabilization consists in rewriting \ref{stabilization 1}.(1) in terms of some stable formulas \ref{stabilization 1}.(2) for some groups $G'$ associated 
to $G$ --- the elliptic endoscopic groups of $G$. There is a ``correspondence'' $(\gamma,\kappa)\rightarrow (\underline{G}',\delta)$, described in \cite[9.7]{K1} (see also \cite[IV.3]{Lab5}), 
between:
\begin{itemize}
\item pairs $(\gamma,\kappa)$ where $\gamma\in G(F)_{\rm e}$ and $\kappa \in \mathfrak{K}(G_\gamma,G;F)_1$;
\item pairs $(\underline{G}', \delta)$ where $\underline{G}'=(G',\ES{G}',s)$ is an endoscopic datum of $G$ which is elliptic and relevant, and $\delta\in G'(F)_{\rm e}$ is a 
$(G,G')$--regular element which corresponds to $\gamma$.
\end{itemize}
Here we use the notion of endoscopic datum given in \cite[I]{MW}. 
A semisimple element $\delta\in G'(F)$ is called {\it $(G,G')$--regular} if 
for any semisimple element $\gamma\in G(F)$ 
which corresponds to $\delta$, the connected centralizers $G'_\delta$ and 
$G_\gamma$ are isomorphic. For the other notions attached to endoscopic data (ellipticity, equivalence, relevance, etc.), 
we refer to loc.~cit. 

Let $\underline{G}'=(G',\ES{G}',s)$ be an elliptic endoscopic datum of $G$. For a function $f'\in C^\infty_{{\rm c}}(G'({\Bbb A}_F))$, we put
$$
{\bf T}_{\rm e}^{\underline{G}'}(f')= \sum_{\delta} a^{G'}(\delta) {\bf O}_\delta(f')\leqno{(1)}
$$
where $\delta$ runs over a set of representative elements of the conjugacy classes in $G'(F)_{\rm e}$ which are $(G,G')$--regular. Note that (1) is just a variant of \ref{elliptic terms of the trace formula}.(4) for $G'$, where we restrict to the elements in $G'(F)_{\rm e}$ 
which are $(G,G')$--regular. We define the stable version of (1) by
$$
{\bf ST}_{\rm e}^{\underline{G}'}(f')= \tau(G')\sum_{\delta}{1\over \tilde{\iota}_{G'}(\delta)}{\bf SO}_{\delta}(f')\leqno{(2)}
$$
where $\delta$ runs over a set of representative elements of the stable conjugacy classes in $G'(F)_{\rm e}$ which are 
$(G,G')$--regular.

Let $\mathfrak{E}_{\rm e}= \mathfrak{E}_{\rm e}(G)$ be a set of representatives for the equivalence classes of endoscopic data of $G$ which are elliptic and relevant. For 
each $\underline{G}'\in \mathfrak{E}_{\rm e}$, we fix some unitary auxiliary data $G'_1$, $C_1$, $\check{\xi}_1$ --- cf. \cite[VI.3.3]{MW}. Recall that $G'_1$ is a central extension of $G$ by the quasi--trivial torus $C_1$, and that the embedding $\check{\xi}_1: \ES{G}'\rightarrow {^LG_1}$ defines a character $\lambda_1$ of $C_1({\Bbb A}_F)$ trivial on $C_1(F)$. Let denote by $C^\infty_{{\rm c},\lambda_1}(G_1({\Bbb A}_F))$ the space of smooth, compactly supported modulo $C_1({\Bbb A}_F)$, functions on $G_1({\Bbb A}_F)$ that transform according to $\lambda_1^{-1}$ on $C_1({\Bbb A}_F)$. For a function $f'_1\in C^\infty_{{\rm c},\lambda_1}(G_1({\Bbb A}_F))$, one can define some ``variants with central character'' of the formulas (1) and (2) above\footnote{This slight complication does not exist if the derived group $G_{\rm der}$ of $G$ is simply connected, but this condition is hard to achieve while preserving the hypothesis $A_G=\{1\}$. However, to work around the problem, we may assume that the only $\underline{G}'=(G',\ES{G}',s)\in \mathfrak{G}_{\rm e}$ 
that can give a non--trivial contribution to the right hand side of (3) are those which satisfy $\ES{G}'\simeq {^LG'}$ (e.g. if ${\bf SO}_\delta(f^{\underline{G}'})=0$ for all the (semisimple) strongly regular elements $\delta\in G'(F)$ when $\ES{G}'\not\simeq {^LG'}$). This is what will be done later.}. 
Let $f\in C^\infty_{\rm c}(G({\Bbb A}_F))$ and for each $\underline{G}'\in \mathfrak{G}_{\rm e}$, let $f^{\underline{G}'}\in C^\infty_{{\rm c},\lambda_1}(G'_1({\Bbb A}_F))$, such that:
\begin{itemize}
\item $f$ decomposes as $f=\prod_vf_v$;
\item $f^{\underline{G}'}$ decomposes as $f^{\underline{G}'}= \prod_v f^{\underline{G}'}_v$;
\item for each place $v$ of $F$, $f^{\underline{G}'}_v$ is a local transfer of $f_v$ (see \ref{the local transfer map}).
\end{itemize}
The stabilization of the elliptic part of the trace formula \ref{elliptic terms of the trace formula}.(4)\ is the following formula \cite[9.6]{K1} (see also \cite[V.4.2]{Lab5}), for $f$ and $\{f^{\underline{G}'}\}_{\underline{G}'\in \mathfrak{E}_{\rm e}}$ as above:
$$
{\bf T}_{\rm e}(f)= \sum_{\underline{G}'\in \mathfrak{E}_{\rm e}}\iota(G,\underline{G}'){\bf ST}_{\rm e}^{\underline{G}'}(f^{\underline{G}'}).
\leqno{(3)}
$$
Here $\iota(G,\underline{G}')$ is a constant which depends only on the equivalence class of $\underline{G}'$. It is given by
$$
\iota(G,\underline{G}')= {\tau(G)\over \lambda(\underline{G}')\tau(G')}
$$
where $\lambda(\underline{G}')$ is the number of elements of the finite group ${\rm Out}(\underline{G}')$ --- for the definition of ${\rm Out}(\underline{G}')$, 
cf. \cite[I.1.5]{MW}.

\subsection{On the local transfer map}\label{the local transfer map}
Let $\underline{G}'=(G',\ES{G}',s)\in \mathfrak{G}_{\rm e}$.

For each place $v$ of $F$, $\underline{G}'$ gives by localization a relevant endoscopic datum $\underline{G}'_v= (G'_v,\ES{G}'_v,s)$ of $G_v$, which may not be elliptic. 
Let $\ES{D}(\underline{G}'_v)$ be the set of pairs of corresponding semisimple elements $(\delta_v,\gamma_v)\in G'(F_v)\times G(F_v)$ 
such that $\delta_v$ is strongly $G_v$--regular (i.e. $\gamma_v$ is strongly regular). Since $\underline{G}'$ is relevant --- which means that for any place $v$ of $F$, the set $\ES{D}(\underline{G}'_v)$ is nonempty ---, we know there exists a strongly $G$--regular element of $G'(F)$ which corresponds, at each place $v$ of $F$, to a strongly regular element of $G(F_v)$. This property allows us to define, for $\delta=(\delta_v)\in G'({\Bbb A}_F)$ and $\gamma=(\gamma_v)\in G({\Bbb A}_F)$ such that $(\delta_v,\gamma_v)\in \ES{D}(\underline{G}'_v)$ for any $v$, a canonical global transfer factor $\Delta(\delta,\gamma)$ which decomposes as
$$
\Delta(\delta,\gamma)= \prod_v \Delta_{v}(\delta_v,\gamma_v),\leqno{(1)}
$$
where $\Delta_v(\delta_v,\gamma_v)$ is a local transfer factor such that $\Delta_v(\delta_v,\gamma_v)=1$ for almost all $v$. 

For each place $v$ of $F$, the auxiliary data $G'_1$, $C_1$, $\check{\xi}_1$ for $\underline{G}'$ give by localization auxiliary data $G'_{1,v}$, $C_{1,v}$, $\check{\xi}_{1,v}$ for $\underline{G}'_v$. The character of $C_1(F_v)$ defined by $\check{\xi}_{1,v}$ is the localization $\lambda_{1,v}$ of the character $\lambda_1$ of $C_1(F)$ defined by $\check{\xi}_1$. Let $\ES{D}_1(\underline{G}'_v)$ be the set of pairs of semisimple elements $(\delta_1,\gamma)\in G'_1(F_v)\times G(F_v)$ such that $\delta_{1,v}$ projects on a strongly $G_v$--regular element in $G'(F_v)$ corresponding to $\gamma_v$ (we say that $\delta_{1,v}$ is strongly $G_v$--regular). For $\delta_1=(\delta_{1,v})\in G'_1({\Bbb A}_F)$ and $\gamma=(\gamma_v)\in G({\Bbb A}_F)$ such that $(\delta_{1,v},\gamma_v)\in \ES{D}_1(\underline{G}'_v)$ for any $v$, we also have a canonical global transfer factor $\Delta_1(\delta_1,\gamma)$, which decomposes as (1).

Let $v$ be a place of $F$, and let $f_v\in C^\infty_{\rm c}(G(F_v))$ and $f'_{1,v}\in C^\infty_{{\rm c},\lambda_{1,v}}(G'_1(F_v))$. For a (semisimple) strongly $G$--regular element 
$\delta_{1,v}\in G'_1(F_v)$, we define the stable orbital integral ${\bf SO}_{\delta_{1,v}}(f_{1,v})$ as in \ref{the theorem} (cf. \cite[I.2.4]{MW}). 
We say that $f'_{1,v}$ is a {\it transfer} of $f_v$ if for any strongly $G$--regular element $\delta_{1,v}\in G'_1(F_v)$, we have
$$
{\bf SO}_{\delta_{1,v}}(f'_{1,v})= \sum_{\gamma_v} \Delta(\delta_{1,v},\gamma_v) {\bf O}_{\gamma_v}(f_v)\leqno{(2)}
$$
where $\gamma_v$ runs over the elements of $G(F_v)$ such that $(\delta_{1,v},\gamma_v)\in \ES{D}_1(\underline{G}'_v)$, modulo conjugation by $G(F_v)$. 
Here the distributions ${\bf O}_{\gamma_v}$ and ${\bf SO}_{\delta_{1,v}}$ are defined using the measures coming by localization from the Tamagawa measures on the groups involved. 
The local transfer assumption is now a theorem\footnote{for $v$ archimedean, it is due to Shelstad, and for $v$ non archimedean, it is due to Ng\^o Bao Ch\^au and  Waldspurger.}: for any function $f_v\in C^\infty_{\rm c}(G(F_v))$, there exists a transfer $f'_{1,v}\in C^\infty_{{\rm c},\lambda_{1,v}}(G'_1(F_v))$ of $f_v$. This function $f'_{1,v}$ is not unique. Let $\bs{I}(G(F_v))$ be the quotient of $C^\infty_{\rm c}(G(F_v))$ by the subspace annulated by all the distributions ${\bf O}_{\gamma_v}$ for $\gamma_v\in G(F_v)$ semisimple strongly regular, and let $\bs{SI}(\underline{G}'_v)=\bs{SI}_{\lambda_{1,v}}(G'_1(F_v))$ be the quotient of $C^\infty_{{\rm c},\lambda_{1,v}}(G'_1(F_v))$ by the subspace annulated by all the distributions ${\bf SO}_{\delta_{1,v}}$ for $\delta_{1,v}\in G'_1(F_v)$ semisimple strongly regular (in $G'_1$). We can see the transfer as a linear map
$$
\bs{I}(G(F_v))\rightarrow \bs{SI}(\underline{G}'_v),\, f\mapsto f'= f^{\underline{G}'_v}.\leqno{(3)}
$$

Let $k$ be a finite extension of the $p$--adic field ${\Bbb Q}_p$, and $H$ a connected reductive $k$--group. We fix a set $\mathfrak{E}_{\rm e}(H)$ of representatives for the equivalence classes of elliptic endoscopic data $\underline{H}'=(H',\ES{H}',s)$ of $H$ which are relevant, and for each $\underline{H}'\in \mathfrak{E}_{\rm e}(H)$, we fix some auxiliary data $H'_1$, $C_1$, $\check{\xi}_1$, and a transfer factor $\Delta_1: \ES{D}_1(\underline{H}')\rightarrow {\Bbb C}^\times$. We also fix a Haar measure $dh$ on $H(k)$ and (for each $\underline{H}'$) a Haar measure $dh'$ on $H'(k)$. These choices allow to define as before a morphism
$$
\bs{I}(H(k))\rightarrow \bigoplus_{\underline{H}'\in \mathfrak{E}_{\rm e}(H)} \bs{SI}(\underline{H}'),\, f\mapsto \oplus_{\underline{H}'}f^{\underline{H}'}\leqno{(4)}
$$
Here $\bs{SI}(\underline{H}')= \bs{SI}_{\lambda_1}(H'_1(k))$ where $\lambda_1$ is the character of $C_1(k)$ defined by $\check{\xi}_1$. In \cite[I.3.1]{MW} is defined a subspace 
$\bs{I}_{\rm cusp}(H(k)) \subset \bs{I}(H(k)$ and, for each $\underline{H}'$, a subspace $\bs{SI}_{\rm cusp}(\underline{H}')\subset \bs{SI}(\underline{H}'(k))$. A function $f\in C^\infty_{\rm c}(H(k))$ is called {\it cuspidal} if its image in $\bs{I}(H(k))$ belongs to $\bs{I}_{\rm cusp}(H(k))$. Let $C^\infty_{\rm cusp}(H(k))\subset C^\infty_{\rm c}(H(k))$ the subspace of cuspidal functions. 
From \cite[I, 4.11]{MW}, the morphism (4) is injective --- its image is precisely described --- and it restricts to an isomorphism 
$$
\bs{I}_{\rm cusp}(H(k))\buildrel\simeq\over{\longrightarrow} \bigoplus_{\underline{H}'\in \mathfrak{E}_{\rm e}(H)} \bs{SI}_{\rm cusp}(\underline{H}')^{{\rm Aut}(\underline{H}')},\, f\mapsto \oplus_{\underline{H}'}f^{\underline{H}'},\leqno{(5)}
$$
where $X^{{\rm Aut}(\underline{H}')}\subset X$ is the subspace of invariants under ${\rm Aut}(\underline{H}')$ for the action defined in \cite[I, 2.6]{MW}.

\begin{defi}
{\rm Let $\underline{H}'\in \mathfrak{E}_{\rm e}(H)$. An element $f\in \bs{I}(H(k))$ is called {\it $\underline{H}'$--stabilizing} if $f^{\underline{H}'}\neq 0$ and $f^{\underline{H}''}=0$ for any $\underline{H}''\in \mathfrak{E}_{\rm e}(H)\smallsetminus \{\underline{H}'\}$. A function $f\in C^\infty_{\rm c}(H(k))$ is called {\it $\underline{H}'$--stabilizing} if its image in $\bs{I}(H(k))$ is $\underline{H}'$--stabilizing.
}
\end{defi}

\begin{rem}
{\rm For $k={\Bbb R}$, and more generally for any locally compact {\it Archimedean} field $k$, 
there is a similar injective morphism (4) which restricts to a similar isomorphism (5) \cite[I, 4.11]{MW}, and 
gives rise to a similar definition.
}
\end{rem}

\subsection{An application}\label{an application}
Let $\gamma\in G(F)_{\rm e}$, and put $I= G_\gamma$. Let $\mathfrak{K}(I,G;F)_0$ be the Pontryagin dual of the abelian 
group
$$
\ker^1_{\rm ab}(F,I\backslash G) = \ker[{\rm H}^1_{\rm ab}(F,I\backslash G)\rightarrow {\rm H}^1_{\rm ab}({\Bbb A}_F,I\backslash G)].
$$
The localization map (cf. the remark of \ref{stabilization 1}) gives an exact sequence
$$
1\rightarrow \mathfrak{K}(I,G;F)_0\rightarrow \mathfrak{K}(I,G;F)_1\rightarrow \prod_v \mathfrak{K}(I_v,G_v;F_v)_1
$$
where $v$ runs over the places of $F$. 
A finite nonempty subset $\ES{V}$ of places of $F$ is called {\it $(G,I)$--essential} \cite[1.9.5]{Lab3} if the 
localization map
$$
\mathfrak{K}(I,G;F)_1\rightarrow \mathfrak{K}(I,G;F_\ES{V})_1=\prod_{v\in \ES{V}} \mathfrak{K}(I_v,G_v; F_v)_1.
$$
is injective. This is possible only if $\mathfrak{K}(I,G;F)_0=\{1\}$. 

\begin{rem1}
{\rm 
A sufficient condition for a singleton $\ES{V}=\{v_1\}$ to be $(G,I)$--essential is that the group $I_{v_1}$ is {\it $G_{v_1}$--elliptic}, i.e. 
there exists a maximal $F_{v_1}$--torus in $I_{v_1}$ which is $F_{v_1}$--elliptic in $G_{v_1}$. If $I_{v_1}$ is $G_{v_1}$--elliptic, then the co--localization map
${\bf H}^0_{\rm ab}(F_{v_1},I_{v_1}\backslash G_{v_1})\rightarrow {\bf H}^0_{\rm ab}({\Bbb A}_F/F,I\backslash G)$
is surjective (cf. \cite[1.9.2]{Lab3}). In that case, the dual localization map 
is injective, and it restricts to an injective map
$
\mathfrak{K}(I,G;F)_1\rightarrow \mathfrak{K}(I_{v_1},G_{v_1};F_{v_1})_1.
$
}
\end{rem1}

\begin{rem2}
{\rm 
Let $\underline{G}'$ be the principal endoscopic datum of $G$. It is possible to isolate the contribution of $\underline{G}'$ in the decomposition \ref{stabilization 2}.(3) by taking a function $f\in C^\infty_{\rm c}(G({\Bbb A}_F))$ which decomposes as a product of local functions $f=\prod_vf_v$ and such that at {\it one} finite place $v_1$ of $F$, the component $f_{v_1}\in C^\infty_{\rm c}(G(F_{v_1}))$ is a cuspidal $\underline{G}'_{v_1}$--stabilizing function --- it is the trick used by Kottwitz in \cite{K2}. For such a function $f_{v_1}$, the $\kappa_{v_1}$--orbital integrals ${\bf O}_{\gamma_{v_1}}^{\kappa_{v_1}}(f_{v_1})$ are zero if $\kappa_{v_1}\neq 1$, for all pairs $(\gamma_{v_1},\kappa_{v_1})$ such that $\gamma_{v_1}\in G(F_{v_1})_{\rm e}$ and $\kappa_{v_1}\in \mathfrak{K}((G_{v_1})_{\gamma_{v_1}},G_{v_1};F_{v_1})_1$. This implies that the global $\kappa$--orbital integrals ${\bf O}_\gamma^\kappa(f)$ are zero if $\kappa \neq 1$, for all pairs $(\gamma,\kappa)$ such that $\gamma\in G(F)_{\rm e}$ and $\kappa\in \mathfrak{K}(G_\gamma,G;F)_1$. More precisely: if $\gamma$ is not elliptic in $G(F_{v_1})$, then ${\bf O}_{\gamma}^{\kappa_{v_1}}(f_{v_1})=0$ because $f_{v_1}$ is cuspidal; if $\gamma$ is elliptic in $G(F_{v_1})$, then $(G_\gamma)_{v_1}$ is a $G_{v_1}$--elliptic, and the singleton $\{v_1\}$ is $(G,G_\gamma)$--essential (remark 1).
}
\end{rem2}

Let $v$ be a finite place of $F$ such that the $F_v$--group $G_v$ is unramified, 
i.e. quasi--split and split over an unramified (finite) extension of $F_v$. 
An endoscopic datum $\underline{G}'= (G'\!,\ES{G}',s)$ of $G_v$ is said to be {\it unramified} if the 
inertia subgroup $I_{F_v}\subset W_{F_v}$ of the Weil group $W_{F_v}$ of $F_v$ is contained in $\ES{G}'$. This implies that the group $G'$ itself 
is unramified \cite[I.6.2]{MW} and the subgroup $\ES{G}'\subset {^LG}$ is isomorphic to the $L$--group of $G'$ \cite[I.6.3]{MW}. 
For a finite set $\ES{V}$ of places of $F$ containing all the 
archimedean places and all the places where $G$ is ramified, we denote by $\mathfrak{E}_{\rm e}^\ES{V}\subset \mathfrak{E}_{\rm e}$ the subset formed of those endoscopic data 
that are unramified outside $\ES{V}$. The set $\mathfrak{E}_{\rm e}^\ES{V}$ is finite \cite[VIII, 5, lemme 8.2]{Lan}.

\begin{defi}
{\rm 
Let $\underline{G}'\in \mathfrak{G}_{\rm e}$. Let $\ES{V}$ be a nonempty finite set of places of $F$ containing 
all the archimedean places, all the finite places where $G$ is ramified, and at least one finite place where $\underline{G}'$ is elliptic, and let $\ES{V}^*\subset \ES{V}$ 
be a nonempty subset of finite places where $\underline{G}'$ is elliptic. We say that the pair $(\ES{V} , \ES{V}^*)$
is {\it $(G,\underline{G}')$--special} if for any $\underline{G}''\in \mathfrak{G}_{\rm e}^\ES{V}$ such that $\underline{G}''_v$ is equivalent to $\underline{G}'_v$ at any place 
$v\in \ES{V}^*$, we have $\underline{G}''=\underline{G}'$. 
}
\end{defi}

\begin{lem}
Let $\underline{G}'\in \mathfrak{G}_{\rm e}$ such that $\underline{G}'$ is elliptic at almost all places of $F$. 
Let $\ES{V}'$ be a finite set of places of $F$ containing all the archimedean places, all the finite places where $G$ or $\underline{G}'$ is ramified, and at least 
one finite place $v_1$ such that $\underline{G}'_{v_1}$ is elliptic. Then there exists a finite set $\{v_2,\dots ,v_r\}$ of places of $F$ such that $v_i\notin \ES{V}'$, 
$\underline{G}'_{v_i}$ is elliptic, and writing $\ES{V}= \ES{V}'\cup \{v_2,\ldots ,v_r\}$ and $\ES{V}^*= \{v_1,\ldots ,v_r\}$, the pair 
$(\ES{V},\ES{V}^*)$ is $(G,\underline{G}')$--special.

\end{lem}

\begin{proof}
Since $\underline{G}'$ is elliptic at almost all place of $F$, it is elliptic {\it and} unramified at almost all finite places of $F$. 
Let $\{\underline{G}^{(1)}=\underline{G}', \underline{G}^{(2)},\ldots ,\underline{G}^{(r)}\}$ be the subset of $\mathfrak{E}_{\rm e}^{\ES{V}'}(G)$ formed of those elements 
such that $\underline{G}^{(i)}_{v_1}$ is equivalent to $\underline{G}^{(1)}_{v_1}=\underline{G}'_{v_1}$. Note that the underlying 
group $G^{(i)}$ of $\underline{G}^{(i)}$ is a (quasi--split) form of $\underline{G}'$. We know\footnote{Unfortunately, there is no published proof of this result today. An unpublished proof 
is due to Waldspurger, 
who checked the result case by case for each irreducible root system. This proof will be written for publication shortly.} that two (global) endoscopic data of $G$ which are equivalent almost everywhere, are globally equivalent. 
For $i=2,\ldots ,r$, pick a finite place $v_i$ of $F$ such that $v_i\notin \ES{V}'$ and $\underline{G}^{(i)}_{v_i}$ is {\it not} equivalent to $\underline{G}'_{v_i}$. The Tchebotarev density 
theorem gives infinitely many choices for $v_i$, hence we can choose $v_i$ such that $\underline{G}'_{v_i}$ is elliptic. Put $\ES{V}= \ES{V}'\cup 
\{v_2,\ldots, v_r\}$ and $\ES{V}^* =\{v_1,\ldots ,v_r\}$. By construction, the pair $(\ES{V},\ES{V}^*)$ is $(G,G')$--special.
\end{proof}

Let $\underline{G}'=(G',\ES{G}',s)\in \mathfrak{E}_{\rm e}$ such that $\underline{G}'$ is elliptic at almost all places of $F$. To simplify the statement, and because it is the only case we will use later, 
we also assume that the subgroup $\ES{G}'\subset {^LG}$ is isomorphic to ${^LG'}$. So we may (and do) that the auxiliary data $G'_1$, $C_1$, $\check{\xi}_1$ for $\underline{G}'$ are the trivial ones 
($G'_1=G'$ and $C_1=\{1\}$). From the lemma, there exists a $(G,\underline{G}')$--special pair $(\ES{V},\ES{V}^*)$ of subsets of places of $F$. 
Let $f\in C^\infty_{\rm c}(G({\Bbb A}_F))$ and $f'\in C^\infty_{\rm c}(G'({\Bbb A}_F))$ be two functions which decompose as $f=\prod_vf_v$ and $f'=\prod_vf'_{v}$, 
such that:
\begin{itemize}
\item at any place $v$ of $F$, $f'_v$ is a transfer of $f_v$;
\item at any place $v\in \ES{V}^*$, $f_{v}$ is a (cuspidal or not) $\underline{G}'_{v}$--stabilizing function;
\item at any place $v\notin \ES{V}$, $f_v$ is the characteristic function of a hyperspecial maximal compact subgroup $K_v$ of $G(F_v)$, 
and $f'_v$ is the characteristic function of a hyperspecial maximal compact subgroup $K'_v$ of $G'(F_v)$. 
\end{itemize}
Then from \ref{stabilization 2}.(3), we have
$$
{\bf T}_{\rm e}(f)= \iota(G,\underline{G}'){\bf ST}_{\rm e}^{\underline{G}'}(f').\leqno{(1)}
$$
Suppose moreover:
\begin{itemize}
\item at {\it one} place $v\in \ES{V}^*$, the function $f'_v$ is cuspidal --- this implies that the function $f_v$ is cuspidal too, because it is $\underline{G}'_v$--stabilizing (cf. \ref{the local transfer map}.(5)) 
--- 
and stabilizing, i.e. there exists a $\delta_v\in G(F_v)_{\rm e}$ such that ${\bf SO}_{\delta_v}(f'_v)\neq 0$, and ${\bf O}_{\delta_v}^{\kappa'_v}(f'_v)=0$ for all $\delta_v\in G'(F_v)_{\rm e}$ and all $\kappa'_v \neq 1$ (e.g. $f'_v$ is a Kottwitz function \cite[3.9]{Lab3}).
\end{itemize}
Since $\underline{G}'$ can be identified with the principal endoscopic datum of $G'$, 
we have (cf. remark 2)
$$
{\bf T}_{\rm e}^{\underline{G}'}(f')= {\bf ST}_{\rm e}^{\underline{G}'}(f').\leqno{(2)}
$$
Combining equalities (1) and (2), we obtain
$$
{\bf T}_{\rm e}(f)= \iota(G,\underline{G}') {\bf T}_{\rm e}^{\underline{G}'}(f').\leqno{(3)}
$$

\begin{rem3}
{\rm If moreover at one finite place $w$ of $F$, the support of the function $f'_w$ is contained in the set of elliptic strongly $G$--regular 
elements of $G'(F_w)$, then we have
$$
{\bf T}_{\rm e}^{\underline{G}'}(f')= {\bf T}_{\rm e}(f')
$$
where ${\bf T}_{\rm e}(f')$ is the whole elliptic part of the trace formula for $G'$ (defined as in \ref{elliptic terms of the trace formula} replacing $G$ by $G'$).
}
\end{rem3}

\subsection{Spectral decomposition (a simple trace formula)}\label{spectral decomposition}
Let $R_{\rm cusp}$ be the restriction of the operator $R$ to the cuspidal tempered spectrum $\mathfrak{X}_{G,{\rm cusp}}$ of $G({\Bbb A}_F)$ in $\mathfrak{X}_G={\rm L}^2(G(F)\backslash G({\Bbb A}_F))$, and ${\rm Irr}_{\rm cusp}= {\rm Irr}_{\rm cusp}(G({\Bbb A}_F))$ the set of isomorphism classes of irreducible admissible representations of $G({\Bbb A}_F)$ in $\mathfrak{X}_{G,{\rm cusp}}$. For $\pi\in {\rm Irr}_{\rm cusp}$, we denote by $m(\pi)$ the multiplicity of $\pi$ in $\mathfrak{X}_{G,{\rm cusp}}$. For $f\in C^\infty_{\rm c}(G({\Bbb A}_F))$, the operator $R_{\rm cusp}(f)$ is of trace class, and we have
$$
{\rm trace}(R_{\rm cusp}(f))= \sum_{\pi\in {\rm Irr}_{\rm cusp}}m(\pi){\rm trace}(\pi(f)).\leqno{(1)}
$$

By imposing more or less strong conditions to the test function $f$, we obtain different versions of the trace formula. If the center $Z(G_v)$ of $G_v$ is not $F_v$--anisotrope, we cannot put at a finite place $v$ of $F$ a coefficient of an irreducible supercuspidal representation of $G(F_v)$ or a pseudo--coefficient of a discrete series representation of $G(F_v)$, because these functions are {\it not} compactly supported, but only compactly supported modulo $Z(G;F_v)$. To avoid this problem, {\bf from now to the end of \ref{spectral decomposition}, we assume that the group $G$ is semi--simple}:
$$
G=G_{\rm der}.
$$

The simplest version of the trace formula is the following (it was first proved by Deligne and Kazhdan). Suppose $f\in C^\infty_{\rm c}(G({\Bbb A}_F))$ decomposes as $f= \prod_vf_v$, and there are finite places $v_1,\,v_2$ of $F$ ($v_1=v_2$ is allowed) such that:
\begin{itemize}
\item $f_{v_1}$ is a finite linear combination of coefficients of irreducible supercuspidal representations of 
$G(F_{v_1})$;
\item $f_{v_2}$ is supported on the set of elliptic regular elements in $G(F_{v_2})$. 
\end{itemize}
In that case, we have \cite[4]{Hen}
$$
{\rm trace}(R_{\rm cusp}(f)) = {\bf T}_{\rm e}(f).\leqno{(2)}
$$
This version may be too simple for some applications, because it does not see the (semisimple) elliptic but not regular elements --- in particular, it does not see the term $\tau(G)f(1)$. Another simple trace formula is obtained by replacing the condition at $v_2$  by the less strong following condition:
\begin{itemize}
\item all the semisimple orbital integrals of $f_{v_2}$ are zero except (perhaps) for the elliptic elements $\gamma\in G(F_{v_2})$ --- in particular $f_{v_2}$ is cuspidal.
\end{itemize}
In that case, we still have \cite[5]{K2}
$$
{\rm trace}(R_{\rm cusp}(f)) = {\bf T}_{\rm e}(f))\leqno{(3)}
$$
Note that the Euler--Poincar\'e function on $G(F_{v_2})$, or a pseudo--coefficient of the Steinberg representation of $G(F_{v_2})$, are examples 
of functions $f_{v_2}$ verifying the above condition at $v_2$.

\subsection{End of the proof in characteristic zero \cite{Cl, Lab1,Hal}}\label{end of the proof}
Now let's go back to the local situation of section \ref{main result} and \ref{proof of the theorem assuming local data}: $F_0$ is a finite extension of a $p$--adic field ${\Bbb Q}_p$, 
$G_0$ is a $F_0$--split connected reductive group, and $G'_0$ is a $F_0$--split elliptic endoscopic group of $G_0$ defined by a depth--zero character of the maximal compact subgroup 
of $T_0(F)$, where $T_0$ is a maximal $F_0$--split torus of $G_0$. 
Note that $G_0$ and $G'_0$ are $F_0$--isomorphic to Chevalley groups, i.e. algebraic groups 
defined over ${\Bbb Z}$. We may assume $G_0$ is adjoint (remark of \ref{further reductions (step 2)}).

{\bf From now to the end of \ref{end of the proof}, we suppose that $G_0$ is adjoint}:
$$
G_0= G_{0,{\rm AD}}.
$$ 

There exists a totally imaginary number field $F$, a $F$--split connected reductive group $G$, and a $F$--split endoscopic group $G'$ of $G$, such that at a finite place $v_0$ of $F$, we have
$$
F_{v_0}\simeq F_0,\quad G_{v_0}\simeq G_0,\quad G'_{v_0}\simeq G'_0. 
$$
Then $G=G_{\rm AD}$, and $G'$ defines an elliptic endoscopic datum $\underline{G}'=(G',\ES{G}',s)$ of $G$, where $\ES{G}'= \check{G}' \times W_F$  
and $s$ is a semisimple element in the dual group $\check{G}$ such that $\check{G}_s = \check{G}'$. Note that $\underline{G}'$ is elliptic at all places of $F$, 
and unramified at all finite places. We fix a Borel pair $(B,T)$ of $G$, such that $T$ is a maximal $F$--split torus in $G$. It defines, for each place $v$ of $F$, a Borel pair $(B_v,T_v)$ 
of $G_v$. We identifie $G_0(F_0)$ with $G(F_{v_0})$ and $G'_0(F_0)$ with $G'(F_{v_0})$. 

Let us fix three finite places $w_1,\,w_2,\,w_3$, different of each other, and such that $v_0\notin \{w_1,w_2,w_3\}$. 
Let $v_1$ be a finite place of $F$ not in $\{v_0,w_1,w_2,w_3\}$, and 
$\ES{V}'$ a finite set of places of $F$ containing $\{v_0,v_1,w_1,w_2,w_3\}$ and all the archimedean places. From the lemma of \ref{an application}, we can complete $\ES{V}'$ in $\ES{V}= \ES{V}' \cup \{v_2,\ldots ,v_r\}$ with $v_i\notin \ES{V}'$, such that the pair $(\ES{V}, \ES{V}^* =\{v_1,\ldots ,v_r\})$ is $(G,\underline{G}')$--special. 
Let $f\in C^\infty_{\rm c}(G({\Bbb A}_F))$ be a function which decomposes as 
$f= \prod_v f_v$, and $f'\in C^\infty_{\rm c}(G'({\Bbb A}_F))$ a function which decomposes as $f'= \prod_v f'_v$, verifying the following conditions:
\begin{itemize}
\item at any place $v$ of $F$, $f'_v$ is a transfer of $f_v$;
\item $f_{w_1}$ is a coefficient of a supercuspidal representation of $G(F_{w_1})$;
\item $f'_{w_2}$ is a coefficient of a supercuspidal representations of $G'(F_{w_2})$;
\item the support of $f_{w_3}$ is contained in the set of elliptic strongly regular element in $G(F_{w_3})$, and the support of $f'_{w_3}$ is contained in the set of elliptic strongly $G_{w_3}$--regular elements in $G'(F_{w_3})$;
\item at any place $v\in \ES{V}^*$, $f_{v}$ is a $\underline{G}'_{v}$--stabilizing function;
\item $f'_{v_1}$ is a cuspidal stabilizing function --- e.g. a pseudo--coefficient of the Steinberg representation of $G'(F_{v_1})$;
\item at any place $v\notin \ES{V}$, $f_v$ is the characteristic function of a hyperspecial maximal compact subgroup $K_v$ of $G(F_v)$, 
and $f'_v$ is the characteristic function of a hyperspecial maximal compact subgroup $K'_v$ of $G'(F_v)$. 
\end{itemize}
From \ref{an application} and 
\ref{spectral decomposition}, we have
$$
{\rm trace}(R_{\rm cusp}(f))=\iota(G, \underline{G}') {\rm trace}(R_{\rm cusp}(f')).\leqno{(1)}
$$
Let $F_\infty$ be the product of the completions $F_v$ for $v$ archimedean (i.e. for $v\vert \infty$ where $\infty$ is the archimedean place of ${\Bbb Q}$):
$$
F_\infty= \prod_{v\vert \infty} F_v= F\otimes_{\Bbb Q}{\Bbb R}.
$$
By varying the functions $f_v$ at the archimedean places $v$, we obtain \cite[6]{Hal} that (1) is equivalent to a family of identities, indexed by the 
irreducible unitary representations $\pi_\infty= \otimes_{v\vert \infty} \pi_v$ of $G(F_\infty)=\prod_{v\vert \infty}G(F_v)$. For such a 
$\pi_\infty$, we obtain an identity (cf. loc.~cit.)
$$
\sum_{\pi} a_{\pi_\infty}(\pi_{v_0}){\rm trace}(\pi_{v_0}(f_{v_0}))= \sum_{\pi'} b_{\pi_\infty}(\pi'_{v_0}){\rm trace}(\pi'_{v_0}(f'_{v_0}))\leqno{(2)}
$$
where $\pi$, resp. $\pi'$, runs over the elements of ${\rm Irr}_{\rm cusp}$, resp. ${\rm Irr}_{\rm cusp}(G'({\Bbb A}_F))$. More precisely, 
we can choose the functions $f_\infty = \prod_{v \vert \infty}f_v$ and $f'_\infty = \prod_{v\vert \infty}f'_v$ in such a way that the archimedean 
components $\pi_\infty= \prod_{v\vert\infty}\pi_v$, resp. $\pi'_\infty = \prod_{v\vert \infty}\pi'_v$ of the representations 
$\pi$, resp. $\pi'$, giving a non--trivial contribution to
(1) are irreducible tempered principal series representations of $G(F_\infty)$, resp. $G'(F_\infty)$. For such a $\pi'_\infty$, 
the character of $\pi'_\infty$ 
lifts to a tempered invariant distribution on 
$G(F_\infty)$ which is, up to a sign, the character of a principal series representation. 
By definition of the lift of a tempered distribution, each non--zero term ${\rm trace}(\pi'_\infty(f'_\infty))$ may be replaced 
by a term ${\rm trace}(\pi_\infty(f_\infty))$ for some tempered representation $\pi_\infty$ of $G(F_\infty)$, which is in fact (by the choice of $f_\infty$) 
an irreducible principal series. Now suppose 
the functions $f_v$, resp. $f'_v$, are fixed as above at all places $v\neq v_0$. The identity (2) is true for each pair of functions $(f_{v_0},f'_{v_0})$ 
such that $f'_{v_0}$ is a transfer of $f_{v_0}$. If moreover $(f_{v_0},f'_{v_0})\in \ES{H}_0\times \bs{\ES{H}}'_0$ (with the notations \ref{local data}), the cuspidal representations $\pi$, resp. $\pi'$, giving a non--trivial contribution to (2) have a non--zero fixed vector under the action of an open compact subgroup of the group of points of $G$, resp. $G'$, on the finite ad\`eles, which depends 
on the functions $f_v$, resp. $f'_v$, we have fixed 
at each finite place $v\neq v_0$ of $F$, and contains the pro--unipotent radical of an Iwahori subgroup at the place $v_0$. Since the archimedean component $\pi_\infty$ of $\pi$ 
is fixed, and is up to a sign a lift of the archimedean component $\pi'_\infty$ of $\pi'$ (by the choice of $f'_\infty$), 
these cuspidal representations $\pi$, resp. $\pi'$, belong to a finite set (by Harish--Chandra finitness theorem) which do not depend on $(f_{v_0},f'_{v_0})\in \ES{H}_0\times \bs{\ES{H}}'_0$. This prove the implication $(B) \Rightarrow (A)$ in the definition of local data (\ref{local data}). However, to prove the inverse implication $(A)\Rightarrow (B)$, we have to break the character identities (2) into a collection of character identities, one for each strongly $G_{v_0}$--regular element $\delta_0\in G'_0(F_0)=G'(F_{v_0})$.

So let $\delta_0\in G'(F_{v_0})$ be a strongly $G_{v_0}$--regular element. 
Choose a strongly $G$--regular element $\delta\in G'(F)$ being elliptic at a finite place $w_3\notin \{v_0,v_1,w_1,w_2\}$ of $F$, 
and approximating $\delta_0$ at $v_0$. More precisely, we demand that
$$
\wt{\bf O}_{\delta}(f_{v_0}) - {\bf SO}_{\delta}(f'_{v_0})= \wt{\bf O}_{\delta_0}(f_{v_0})- {\bf SO}_{\delta_0}(f'_{v_0})\leqno{(3)}
$$
for all pairs of functions $(f_{v_0},f'_{v_0})\in \ES{H}_0\times \bs{\ES{H}}'_0$ (cf. \ref{local data}). Such a global element $\delta$ exists by weak approximation and the Howe conjecture. Note that since $\delta$ is elliptic at a finite place, it is a fortiori elliptic in $G'(F)$. 
The test functions $f_\infty=\prod_{v\vert \infty}f_v$ and $f'_\infty= \prod_{v\vert \infty}f'_v$ are fixed as before. At each archimedean place $v$ of $F$, $f_v$ 
determines an open subset $U_v$ of $G(F_v)$ contained in the set of semisimple regular elements in $G(F_v)$, such that ${\bf O}_{\gamma_v}(f_v)\neq 0$ for all $\gamma_v\in U_v$ 
(cf. \cite[6, p.~991]{Hal}). By weak approximation, we may assume that at all archimedean places $v$ of $F$, $\delta$ corresponds to a strongly regular element $\gamma_v\in G(F_v)$ which belongs to $U_v$. We choose the finite set $\ES{V}'$ containing $\{v_0,v_1,w_1,w_2,w_3\}\cup \{v: v\vert\infty\}$
such that $\delta\in K'_v$ for all places $v\notin \ES{V}'$. We define $\ES{V}^*=\{v_1,\ldots ,v_r\}$ and $\ES{V}= \ES{V'}\cup\{v_2,\ldots ,v_r\}$ as before. Again by weak approximation, we may arrange that $\wt{\bf O}_\delta(f_{w_1})\neq 0$ and ${\bf SO}_\delta(f'_{w_2})\neq 0$, and ${\bf SO}_\delta(f'_v)\neq 0$ at each finite place $v\in \ES{V} \smallsetminus \{v_0,w_1,w_2\}$. At the finite places $v\notin \ES{V}$, the function $f_v$ and $f'_v$ are choosen as before. Finally, shrinking the support of $f'_{w_3}$, we can also arrange that the only 
$G'({\Bbb A}_F)$--conjugacy classes in $G'({\Bbb A}_F)$ meeting the support of $f'_{w_3}$ come from $\delta$. 
Now let $(h_{v_0},h'_{v_0})\in C^\infty_{\rm c}(G(F_{v_0}))\times C^\infty_{\rm c}(G'(F_{v_0}))$ be a pair of functions such that $h'_{v_0}$ is a transfer of $h_{v_0}$ and ${\bf SO}_\delta(h'_{v_0})\neq 0$. 
Then $h_{v_0}$, resp. $h'_{v_0}$, is biinvariant by a compact open subgroup $J_{v_0}$, resp. $J'_{v_0}$, of $G(F_{v_0})$, resp. $G'(F_{v_0})$, contained in the pro--unipotent radical of the Iwahori subgroup underlying the definition of $\ES{H}_0$, resp. $\bs{\ES{H}}'_0$. At the place $v_0$, we consider only pairs of test functions $(f_{v_0},f'_{v_0})$ such that $f_{v_0}$ is $J_{v_0}$--biinvariant and $f'_{v_0}$ is $J'_{v_0}$--biinvariant 
(e.g. $(f_{v_0},f'_{v_0})\in \ES{H}_0\times \bs{\ES{H}}'_0$). If $f'_{v_0}$ is a transfer of $f_{v_0}$, then we obtain as before the identity (2), the cuspidal representations $\pi$, resp. $\pi'$, 
running over a finite set which depends only on $J_{v_0}$, resp. $J'_{v_0}$, and on the test functions fixed at the other places of $F$. 
Conversely, suppose we have the identity (2) for all irreducible unitary representations $\pi_\infty$ 
of $G(F_\infty)$. Hence we have
$$
{\bf T}_{\rm e}(f)= \iota(G,\underline{G}'){\bf T}_{\rm e}(f').
$$
From the choice of the test functions $f_v$ at $v\neq v_0$, there are two non--zero constants $c,\,c'$ such that
$$
{\bf T}_{\rm e}(f) = \wt{\bf O}_\delta(f) = c\, \wt{\bf O}_\delta(f_{v_0})
$$
and
$$
{\bf T}_{\rm e}(f')= {\bf SO}_\delta(f')= c' {\bf SO}_\delta(f'_{v_0}).
$$
Thus we have
$$
\wt{\bf O}_\delta(f_{v_0})= \iota(G,\underline{G}') c^{-1}c' {\bf SO}_\delta(f'_{v_0}).
$$
Replacing $(f_{v_0},f'_{v_0})$ by $(h_{v_0},h'_{v_0})$ in the last equality, we obtain
$$
\iota(G,\underline{G}') c^{-1}c'=1.
$$ 
Hence we have
$$
\wt{\bf O}_\delta(f_{v_0})= {\bf SO}_\delta(f'_{v_0}).
$$
If moreover $(f_{v_0},f'_{v_0})$ belongs to $\ES{H}_0\times \bs{\ES{H}}'_0$, from (3) we obtain
$$
\wt{\bf O}_{\delta_0}(f_{v_0})= {\bf SO}_{\delta_0}(f'_{v_0}).
$$

\subsection{Existence of local data (${\rm char}(F)>0$)}\label{local data in char p}
In this subsection \ref{local data in char p}, we suppose $F$ is a function field of characteristic $p$.

Note that the Hasse principle for all semisimple simply connected groups 
is known by Harder \cite{Har2}, and the Weil conjecture on Tamagawa number has been proven recently by Gaitsgory and Lurie \cite{GL}. 

The proof of the stabilization of the elliptic regular part of the trace formula given by Langlands in \cite{Lan} is still valid in this context. In other words, replacing in 
\ref{elliptic terms of the trace formula}--\ref{stabilization 2} the set 
$G(F)_{\rm e}$ by the subset $G(F)_{\rm e,reg}\subset G(F)_{\rm e}$ of regular elements, we obtain the analogous formula of \ref{stabilization 2}.(3):
$$
{\bf T}_{\rm e,reg}(f) = \sum_{\underline{G}'\in \mathfrak{E}_{\rm e}} \iota(G,\underline{G}') {\bf ST}_{\rm e,reg}^{\underline{G}'}(f^{\underline{G}'}).\leqno{(1)}
$$
Note that if $\gamma,\,\gamma'\in G(F)_{\rm e,reg}$ are stably conjugate, the groups $G_\gamma$ and $G_{\gamma'}$ are $F$--isomorphic, 
and $\tau(G_\gamma)=\tau(G_{\gamma'})$. The only point on which we should give some details is the {\it Langlands obstruction}: denoting by $G^*$ the quasi--split form of $G$, 
we seek a cohomological condition so that an element of $G^*(F)$ which transfers to $G$ locally everywhere, transfers also globally\footnote{For number fields, 
the question has been solved by Langlands \cite{Lan} and Kottwitz \cite{K1} using global 
Poitou--Tate--Nakayama duality, and, in a more general context, by Labesse \cite{Lab5} using non--abelian 
cohomological constructions.}. In order to avoid additional complications, we suppose in (1) that $G$ is quasi--split (in that case, there is no Langlands obstruction). 

Of course on the spectral side, for $f_{v_1}$ and $f_{v_2}$ verifying the conditions of \ref{spectral decomposition}.(2), we have\footnote{Note that the simple trace formula \ref{spectral decomposition}.(2) is proved in \cite[4]{Hen} for a global field $F$ of characteristic $\geq 0$.}
$$
{\rm trace}(R_{\rm cusp}(f)) = {\bf T}_{\rm e,reg}(f).\leqno{(2)}
$$

Since the local transfer assumption is not proved in positive characteristic, we have to modify the definition of \ref{the local transfer map}: let $k$ be a finite extension of ${\Bbb F}_p((t))$, and $H$ a connected reductive $k$--group. We fix $\mathfrak{E}_{\rm e}(H)$ as in \ref{the local transfer map}. A pair 
$(\gamma,\kappa)$ formed of an elliptic element $\gamma \in H(k)$ and an endoscopic character $\kappa \in \mathfrak{K}(G_\gamma,G,k)_1$ is called an 
{\it elliptic endoscopic pair}. To such a pair is associated an element $\underline{H}_\kappa \in \mathfrak{E}_{\rm e}(H)$, cf. \ref{stabilization 2}. Let $\underline{H}'=(H'\!,\ES{H}',s)\in \mathfrak{E}_{\rm e}(H)$. 
A function $f\in C^\infty_{\rm c}(H(k))$ is said to be {\it $\underline{H}'$--stabilizing} if there exists an elliptic endoscopic pair $(\gamma, \kappa)$ such that $\underline{H}_\kappa = \underline{H}'$ and 
${\bf O}^\kappa_\gamma(f)\neq 0$, and ${\bf O}_\gamma^\kappa(f)=0$ for any elliptic endoscopic pair $(\gamma,\kappa)$ such that $\underline{H}_{\kappa} \neq \underline{H}'$. With this definition, the results of \ref{an application} remain true. Moreover, from \cite[6, pp.~989--990]{Hal}, we have the 

\begin{lem1}
Suppose $H$ is quasi--split. Let $\underline{H}'=(H'\!,\ES{H}'\!,s)\in \mathfrak{E}_{\rm e}(H)$, and $\delta\in H'(k)$ an elliptic strongly $H$--regular element. There exist a function $f\in C^\infty_{\rm c}(H(k))$ supported on the set of elliptic strongly regular elements in $H(k)$, and a function $f'\in C^\infty_{\rm c}(H'(k))$ supported on the set of elliptic strongly $H$--regular elements in $H'(k)$, such that $f'$ is a transfer of $f$ and ${\bf SO}_\delta(f')\neq 0$. Moreover, we can choose $(f,f')$ such that $f$ is $\underline{H}'$--stabilizing, resp. $f'$ is stabilizing.   
\end{lem1}
 
Now we can replace ${\Bbb Q}_p$ by ${\Bbb F}_p((t))$ in \ref{end of the proof}. We always suppose $G_0$ is adjoint. There exist a global field $F$ of characteristic $p$, a $F$--split connected reductive group $G$, and a $F$--split endoscopic group $G'$ of $G$, such that at a finite place $v_0$ of $F$, we have $F_{v_0}\simeq F_0$, $G_{v_0}\simeq G_0$ and $G'_{v_0}\simeq G'_0$. From \cite{Har1}, 
we know that $G$ satisfies the weak approximation property. All the considerations at archimedean places are useless here. We fix a $(G,\underline{G}')$--special pair $(\ES{V},\ES{V}')$ of sets of places of $F$ as before, and a strongly $G_{v_0}$--regular element $\delta_0\in G'(F_{v_0})$. By weak approximation property and the Howe conjecture (proved by Barbasch--Moy \cite{BM} in characteristic $p$), we can choose a strongly $G$--regular element $\delta\in G'(F)$ elliptic at $w_3$ and verifying \ref{end of the proof}.(3). We may also assume that $\delta$ is elliptic at all places 
$v\in \ES{V}\smallsetminus \{v_0\}$. Thanks to the lemma 1, the proof of \ref{end of the proof} works here, provided that the following results are available:
\begin{enumerate}
\item[(H1)] at $w_1$, there exists a non--zero (finite) linear combination $f_{w_1}$ of matrix coefficients of irreducible supercuspidal representations of $G(F_{w_1})$ that transfers to $G'(F_{w_1})$;
\item[(H2)] at $w_2$, there exists a function $f_{w_2}\in C^\infty_{\rm c}(G(F_{w_2})$ that transfers to a non--zero (finite) linear combination of matrix coefficients of supercuspidal representations of $G'(F_{w_2})$;
\item[(H3)] (fundamental lemma for the unit) at any place $v\notin \ES{V}$, the characteristic function of $K'_{v}$ is a transfer of the characteristic function of $K_v$.
\end{enumerate}

We claim that (H1), (H2), (H3) are true if $p$ is large enough. For (H3), it is a consequence of Ng\^o's result on the Lie algebra \cite{N}, cf. also the remark below. For (H1) and (H2), let us take again the connected reductive 
$k$--group $H$ introduced above.

\begin{lem2} 
Suppose $H$ is quasi--split and unramified. Let $\underline{H}'=(H'\!,\ES{H}',s)\in \mathfrak{E}_{\rm e}(H)$, and $\delta\in H'(k)$ an elliptic strongly $H$--regular element. Suppose the endoscopic 
data $\underline{H}'$ is unramified, the $k$--torus $T'=H^\delta$ is unramified, and $\delta$ belongs to the maximal compact subgroup $T'(\mathfrak{o}_k)$ of $T'(k)$. If $p= {\rm char}(k)$ is ``large enough'', there exist a (finite) linear combination $f$ of matrix coefficients of irreducible depth zero supercuspidal representations of $H(k)$, and a (finite) linear combination $f'$ of matrix coefficients of irreducible depth zero supercuspidal representations of $H'(k)$, such that $f$ transfers to $f'$ and ${\bf SO}_\delta(f')\neq 0$. 
\end{lem2}

\begin{proof}
Cf. \cite[5]{Hal}. In loc~cit., the base field $k$ is supposed to be of characteristic $0$ and of large enough residual characteristic $p$. All the arguments remain valid if ${\rm char}(k)=p>0$ is 
large enough.
\end{proof}

{\bf Suppose $p$ is ``large enough''}\footnote{Looking at the proof of \cite[lemma 5.1]{Hal}, it is possible to make the counting argument explicit, hence to 
get an explicit bound for the $p$ which work. More precisely, let ${\Bbb F}_q$ be the residue field of $k$. The torus $T'$ transfers to $H$. One has to compare the dimension of various spaces, for $H'$ and for $H$: the space of Deligne--Lusztig characters coming from $T'({\Bbb F}_q)$; the space of invariant functions  supported on the orbits of regular elements in $T'({\Bbb F}_q)$; the space of invariant functions supported on the orbits of elements in $T'({\Bbb F}_q)$. The asymptotic values of these dimensions are known --- this makes the lemma true for $p$ ``large enough'' ---, and we need to calculate these dimensions.  }. 
Taking $w_1=w_2$, we can put at $w_1$ a pair of matching functions $(f_{w_1},f'_{w_1})$ given by the lemma 2. Moreover, by weak approximation theory, we can choose $\delta$ elliptic {\it and} unramified at $w_1$, i.e. such that the maximal $F$--torus $T'= G'^\delta$ is unramified at $w_1$. Since $T'_{w_1}$ is $F_{w_1}$--anistropic (recall $G=G_{\rm AD}$), we have $T'(\mathfrak{o}_{F_{w_1}})= T'(F_{w_1})$. So we are done.
\begin{rem}
{\rm From \cite{N,W1,Hal}, the fundamental lemma for the whole Hecke algebra is true in characteristic $0$ without any restriction on the residual characteristic $p$. It should be possible to deduce it in characteristic $p>0$, without any restriction on $p$, via the method of close fields as in \cite{Lem}. So if we could prove (H1) and (H2) in general (for a $k$--split adjoint group $H$ on a finite extension $k$  of ${\Bbb F}_p((t))$ and a $k$--split endoscopic group $H'$ of it), then our result would be true in characteristic $p>0$ without any restriction on $p$. 
}\end{rem}


\end{document}